\documentclass[11pt]{amsart}
\usepackage[a4paper,margin=2.5cm]{geometry}

\usepackage{tikz}
\usetikzlibrary{arrows.meta}   
\usetikzlibrary{calc}          
\usetikzlibrary{backgrounds}   

\usepackage[english]{babel}
\usepackage{graphicx,amsmath,amssymb,color, enumerate}
\numberwithin{equation}{section}
\usepackage{enumerate}
\usepackage{comment}

\usepackage{graphicx}
\usepackage{xcolor}
\definecolor{brickred}{rgb}{0.8, 0.25, 0.33}
\definecolor{blue(ryb)}{rgb}{0.01, 0.28, 1.0}
\definecolor{brandeisblue}{rgb}{0.0, 0.44, 1.0}
\definecolor{ceruleanblue}{rgb}{0.16, 0.32, 0.75}
\definecolor{cobalt}{rgb}{0.0, 0.28, 0.67}
\definecolor{coolblack}{rgb}{0.0, 0.18, 0.39}
\definecolor{darkblue}{rgb}{0.0, 0.0, 0.55}
\usepackage[hidelinks]{hyperref}
\hypersetup{
	colorlinks,
	citecolor=darkblue,
	filecolor=black,
	linkcolor=darkblue,
	urlcolor=black
}

\usepackage{comment}

\newtheorem{theorem}{Theorem}[section]
\newtheorem{lemma}[theorem]{Lemma}
\newtheorem{proposition}[theorem]{Proposition}
\newtheorem{corollary}[theorem]{Corollary}

\theoremstyle{remark}

\newtheorem{remark}[theorem]{Remark}

\newcommand{\mn}{\medskip\noindent}

\newcommand{\R}{\mathbb{R}}

\newcommand{\cE}{\mathcal E}

\newcommand{\action}{I}

\newcommand{\Ai}{\operatorname{Ai}}
\newcommand{\Bi}{\operatorname{Bi}}

\title[Singular harmonic oscillators with Aharonov--Bohm flux]
{Spectral Asymptotics of a Singular Harmonic Oscillator with Aharonov--Bohm Flux}

\author[B. Helffer]{Bernard Helffer}
\address[B. Helffer]{Laboratoire de Math\'ematiques Jean Leray, CNRS, Nantes Universit\'e,
	44000 Nantes, France.}
\email{Bernard.Helffer@univ-nantes.fr}
\author[A. Kachmar]{Ayman Kachmar}
\address[A. Kachmar]{Department of Mathematics and PDE Research Unit–Center for Advanced Mathematical Sciences, American University of Beirut, P.O.Box 11-
	0236, Riad El-Solh, Beirut 1107 2020, Lebanon.}
\email{ak292@aub.edu.lb}
\author[F. Nicoleau]{Fran\c{c}ois Nicoleau}
\address[F. Nicoleau]{Laboratoire de Math\'ematiques Jean Leray, CNRS, Nantes Universit\'e,
	44000 Nantes, France.}
\email{francois.nicoleau@univ-nantes.fr}

\begin{document}

\begin{abstract}
Motivated by a recent article of R. Vanlaere (2026), we investigate the spectral asymptotics of a singular harmonic oscillator arising from the radial reduction of the magnetic Schrödinger operator on the unit disk in the presence of a constant magnetic field and an Aharonov–Bohm flux. The corresponding eigenvalue problem is equivalent to the study of the zeros of the Kummer confluent hypergeometric function with respect to its first parameter. Combining  now standard semiclassical methods with the uniform asymptotic theory of Whittaker and confluent hypergeometric functions  developed by Dunster (1989), Gabutti-Gatteschi (2001) and improved quite recently by Dunster (2026), we obtain a comprehensive description of the different spectral regimes. When the boundary lies in the classically forbidden region, we derive the exponentially small asymptotic correction to the eigenvalue caused by the tunneling effect. For energy levels above the tunneling regime, we derive a shifted Bohr–Sommerfeld asymptotics for highly excited states at a fixed magnetic field. Finally, for the transition regime where the energy level touches the boundary, we derive a three-term asymptotic expansion of the eigenvalue. In an appendix, we translate these spectral results into explicit asymptotic formulas for the zeros of the Kummer function.
\end{abstract}

\maketitle

\section{Introduction}
\subsection{Harmonic oscillator with inverse square potential}
We study the semi-classical asymptotics of the high energy levels of the operator, 
\begin{equation}\label{eq:def-Th}
T_h = -h^2 \frac{d^2}{dx^2} + x^2 + h^2 \frac{\nu^2 - 1/4}{x^2}\,, \qquad \nu \ge 0\,,
\end{equation}
with Dirichlet boundary conditions at \(x=0\) and \(x=1\)\,. Here $\nu\geq 0$ is a given parameter, which will be fixed throughout this paper. 

Thanks to the Hardy inequality, $T_h$ is semi-bounded on $C_c^\infty(0,1)$. The Friedrichs extension yields a self-adjoint operator in $L^2(0,1)$ which we denote by $T_h$ hereafter. 

If we would like to stress that the Dirichlet boundary condition is imposed at $x=1$, we will add a superscript $D$ and write $T_h^D$. Note that if we start with the domain $C_c^\infty((0,1]),$ the Friedrichs extension will be the Neumann realization and we denote it by $T_h^N$. For $T_h^N$, Neumann boundary condition is imposed at $x=1\,$, while Dirichlet is maintained at $x=0\,$.

By the Sturm--Liouville theory, the spectrum of $T_h$ is purely discrete,
consisting of simple eigenvalues
\[
\lambda_1(h)<\lambda_2(h)<\cdots,
\]
whose distribution we would like to understand in the semiclassical
limit.

\subsection{Motivation} One motivation behind the study of $T_h$ is a problem in control theory \cite{Vanlaere2, BMM}. In this context, one encounters the operator $T_h$, with $h=\xi^{-1}$,  which arises by taking the partial Fourier transform of the $y$-variable in the Baouendi-Grushin operator
\[-\partial_x^2-x^2\partial_y^2+\frac{\nu^2-1/4}{ x^2}\,.\]

Another  motivation of our study comes from the magnetic Laplacian. It is known that the presence of a magnetic field leads to an interesting interaction between the magnetic field,
the geometry of the domain, and its topology. A fundamental example is the magnetic Dirichlet Laplacian on the
unit disk
\begin{equation}\label{disk}
D=\{x\in\mathbb R^2;\ |x|<1\}\,,
\end{equation}
submitted to a constant magnetic field together with an
Aharonov--Bohm solenoid located at the origin. Passing to polar
coordinates and separating variables (see \cite{HL} and \cite[Section~1.3]{Vanlaere}), the two-dimensional problem
decomposes into a family of one-dimensional singular
Sturm--Liouville operators indexed by the angular momentum. After  a standard Liouville transformation, one obtains the singular harmonic oscillator acting in $L^2(0,1)$,
\[-\frac{d^2}{dx^2} + \frac{B^2x^2}{4} + \frac{(m+\alpha)^2 - 1/4}{x^2}\,,\]
where $B$ is the intensity of the magnetic field, $m\in\mathbb Z$ is the angular momentum, and $\alpha$ is the Aharonov-Bohm flux.
Setting $h=2/B$ and $ \nu=|m+\alpha|$, and imposing Dirichlet boundary condition on $\partial D,$ we recover the operator $T_h$. The semi-classical limit $h\to0^+$ considered here is equivalent to the strong field limit, $B\to+\infty$ in the context of the magnetic Laplacian. 
\subsection{Universal comparison}
We first observe that, for  any integer $k\geq 1$, we have 
\begin{subequations}\label{eq:1.3}
\begin{equation}
\lambda_k(h)\geq e_k(h)\,,
\end{equation}
where 
\begin{equation}
e_k(h):=(4k+2\nu-2)h
\end{equation}
\end{subequations}
is the $k$'th eigenvalue of the operator (see \cite{ESV})
\[ T_h^\infty=- h^2\frac{d^2}{dx^2}+x^2+h^2\frac{\nu^2-1/4}{x^2}\quad \text{in $L^2(\mathbb R^+)$}.\]

The proof is easy. Extending functions in the domain of $T_h$ by zero, we get by the min-max principle that the eigenvalues of $T_h$ are bounded below by the eigenvalues of $T_h^\infty$, the operator in $\R_+$. 

\subsection{Low lying energy levels}
Fixing  the eigenvalue  label $k$, the following asymptotics hold as \(h\to0^+\,\),
\begin{equation}\label{eq:Gannot*}
\lambda_k(h)
=
(4k+2\nu-2)h
+
\frac{4h^{-2k-\nu+2}}
{(k-1)!\,\Gamma(k+\nu)}
e^{-1/h}
\left(1+\mathcal O(h)\right)\,.
\end{equation}
This
is a special case of a theorem by 
O.~Gannot \cite{Gannot}, which holds for a generic class of scalar potentials, beyond the harmonic potential  discussed in this paper\footnote{Recently, M.~Baur and T.~Weidl \cite{BaurWeidl} derived \eqref{eq:Gannot*} via special functions, and  the second author together with  G.~Miranda \cite{KM} derived it and the analogous formula for the Neumann boundary condition via the Temple inequality (see also \cite{HL}).}. Therefore, the asymptotics \eqref{eq:Gannot*} reflects the exponentially weak influence of the boundary as $h$ tends to $0$. This is a manifestation of tunneling, since the Dirichlet boundary condition models the hard wall potential\footnote{If we add to $T_h$ the potential $V_a=a\mathbf 1_{[1,+\infty)}\,$, then the Dirichlet boundary condition at $x=1$ corresponds to  taking $a=+\infty$.}. 

\subsection{Persistence of tunneling}
Vanlaere recently observed that  tunneling persists up to higher energy levels. More precisely, given $0<\rho^*<1$, it follows from \cite[Theorem~1.4]{Vanlaere}  that there exist
 constants $c,h_0>0$ such that for all pairs $(h,k) \in (0,h_0]\times \mathbb N^*$ satisfying
\[ e_k(h)\leq \rho^*\,,\]
the following bounds hold
\begin{equation}\label{eq:Vanlaere}
\lambda_k(h)\leq e_k(h)+e^{-c/h}\,.
\end{equation}
 Our objective is to prove an accurate asymptotics, specifically to obtain  the optimal exponential decay in \eqref{eq:Vanlaere} and to exhibit  the influence of the parameter $\nu$ (recall that $\nu$ encapsulates the Aharonov-Bohm flux). This quantifies the tunneling mechanism for excited states corresponding to these higher energy levels.

Our first result is to prove that \eqref{eq:Gannot*}
 continues to hold when
$e_k(h)\ll 1\,$, up to $k\lesssim h^{-\gamma}\,$, for any fixed
$0<\gamma<1/2\,$.

\begin{theorem}\label{thm:ext-Gannot}
For any fixed $\gamma\in(0,1/2)$ 
and for any integer $k$ satisfying
\[
1\leq k\leq 
h^{-\gamma}\,, 
\]
we have
\[\lambda_k(h)=e_k(h)+\frac{4}
{(k-1)!\,\Gamma(k+\nu)}\,h^{-2k-\nu+2}\,
e^{-1/h}
\bigl(1+\varepsilon_k(h)\bigr)\,,\]
where
\[\sup_{k\in[0, h^{-\gamma}]}|\varepsilon_k(h)|\longrightarrow 0\quad\text{as $h\longrightarrow0^+$\,.}\] 
\end{theorem} 
Following the terminology of \cite{Sj}, the eigenvalues of order $\mathcal O(h^\delta)$, with $ 0<\delta<1$, are called semi-excited states. 

\mn
Our next task is to explore the regime beyond the semi-excited states, which are not covered by Theorem~\ref{thm:ext-Gannot}.

\mn
For a given $ E\in (0,1)$, we introduce (see for example \cite{He,HS}) the Agmon distance (associated with the potential $V(s)=s^2$) to the ``classical  region" $\{s\in (0,1)\,,\, s^2 \leq E\,\}$:
\begin{equation}\label{eq:def-Agmon}
S(x;E)=\int_0^x\sqrt{(s^2-E)_+}\,ds\,,
\end{equation}
whose maximum on $[0,1]$ is
\begin{equation}\label{eq:S-tunneling}
S_*(E)=\int_0^1\sqrt{(s^2-E)_+}\,ds\,,
\end{equation}
which is the Agmon distance from the boundary $\{s=1\}$ to the classical region.
 We can extend by continuity $S(x,E)$ and $S_*(E)$ at $0$ and notice that $S(x,0)= x^2/2$ and  $S_*(0)=1/2\,$.
\mn
When $E=e_k(h)$, we write
\[\mathsf S_{k,h}=S_*(e_k(h))\,.\]

Our second result is

\begin{theorem}[Tunneling beyond semi-excited states]\label{thm:main}
	Let \(0<\rho^\ast<1\) be given. Let
	\[
	\ell:(0,1]\longrightarrow \mathbb R_+
	\]
	satisfy
	\[
	\ell(h)\longrightarrow 0\,\mbox{ and }
		h^{-1}\ell(h)\longrightarrow +\infty
	\quad\text{ as }h\longrightarrow0^+\,.
	\]
	For every \(h\in(0,1]\)\,, define
	\begin{equation}\label{eq:cond-mu}
	\mathcal J(h)
	=
	\bigl\{
	k\in\mathbb N:
	\ell(h)\leq e_k(h)\leq \rho^\ast
	\bigr\}\,.
	\end{equation}
	Then, for \(k\in\mathcal J(h)\)\,,
	\begin{equation}\label{eq:main-asy}
	\lambda_k(h)
	=
	e_k(h)
	+
	\frac{2}{\pi}\,h\,
	\exp\!\left(
	-\frac{2\mathsf S_{k,h}}{h}
	\right)
	\bigl(1+r_k(h)\bigr)\,,
	\end{equation}
	where
	\[\sup_{k\in\mathcal J(h)}|r_k(h)|\longrightarrow0\quad \text{as $h\longrightarrow0^+\,.$}\]
\end{theorem}
\begin{remark}\label{rem:main} A few remarks on \eqref{eq:cond-mu} and \eqref{eq:main-asy} are in order.
\begin{enumerate}[\rm i)] 
\item We may choose $\ell(h)=h^{1-\delta}$, with $0<\delta<1\,$.
If, in addition, $\delta<1/2$, then for any
$\gamma\in(\delta,1/2)$, the regimes covered by
Theorems~\ref{thm:ext-Gannot} and~\ref{thm:main} overlap in the range
$
h^{-\delta}\lesssim k\lesssim h^{-\gamma}.
$
\item When $k=k(h)\to+\infty$ and $kh\to0^+\,$, we get that $\mathsf S_{k,h}/h\sim 1/2h\,$, in accordance with \eqref{eq:Gannot*}. In Remark~\ref{rem:Gannot}, we will verify that the asymptotics in Theorem~\ref{thm:ext-Gannot} and Theorem~\ref{thm:main} match  in their overlapping region of validity.
\item The case $\rho^*\geq 1$ corresponds to energy levels  above the
ones allowed by \eqref{eq:cond-mu}, for which the tunneling effect
disappears.
One heuristic explanation is the following. Tunneling towards the
boundary $x=1$ occurs when the outer classical turning point lies in
the interior of the domain. For
\[
V_h(x):=x^2+h^2\frac{\nu^2-1/4}{x^2},
\]
and for an energy $E\gg h\,$, the outer turning point is given by
\[
x_+(E,h)
=
\left(
\frac{E}{2}
+
\sqrt{\frac{E^2}{4}
-h^2\left(\nu^2-\frac14\right)}
\right)^{1/2}
=
\sqrt{E}
+
\mathcal O\left(\frac{h^2}{E^{3/2}}\right).
\]
Depending on the sign of $\nu^2-1/4\,$, there may also be an inner
turning point near $x=0$, which is irrelevant for the tunneling
towards the boundary $x=1$. Thus, \eqref{eq:cond-mu} ensures that the
outer turning point remains in the interior of $(0,1)$ and stays away
from the boundary $x=1$.
\item A similar result holds for the Neumann realization, namely
\[\lambda_k^N(h)
	=
	e_k(h)
	-
	\frac{2}{\pi}\,h\, \exp\!\left(-\frac{2\mathsf S_{k,h}}{h}\right)
	\bigl(1+o(1)\bigr)\,,\]
%
which leads to 
\[\lambda_k^D(h)-\lambda_k^N(h)=
	\frac{4}{\pi}\,h\,\exp\!\left(-\frac{2\mathsf S_{k,h}}{h}\right)
	\bigl(1+o(1)\bigr)\,.\]
\end{enumerate}
\end{remark}

\subsubsection*{Application} 
Like in \cite[Corollary~1.3]{Gannot}, for a given integer $l \geq 0$, we can consider the Coulomb Hamiltonian in $L^2(0,1)\,$,
\[\mathcal H_h=-h^2\frac{d^2}{dy^2}+h^2\frac{l (l+1)}{y^2}-\frac{2}{y}\,,\]
with Dirichlet boundary conditions at $y=0$ and $y=1\,$. Denoting the negative eigenvalues by $E_{n}(h)$, $n\geq l+1$ and applying a change of function, we  map  an eigenspace of $\mathcal H_h$ to an eigenspace of $T_h$ with $\nu=2l +1$  (see the proof of Corollary~1.3 in \cite{Gannot}). More precisely, it holds
\[\frac{4}{\sqrt{-E_{n}(h)}}=L_n^2\lambda_k(hL_n^{-2})\quad\text{where $L_n^2=2\sqrt{-E_{n}(h)}$ and $n=k+l$\,.} \]
For $n=k+l$ and $k$ as in Theorem~\ref{thm:main}, we have $hL_n^{-2}\to0^+$ and $e_k(h)=4nh L_n^{-2}\,$. Applying Theorem~\ref{thm:main} with $\nu=2l +1\,$, we get for $n=k+l\,$,
\begin{itemize}
\item the  approximation modulo $\mathcal O (h^\infty)$
\[L_n^{-2}=\frac{nh}{2}+\mathcal O(h^\infty)\,,\]
\item the  asymptotics with exponentially small second term:
\[E_{n}(h)=-\frac{1}{n^2 h^2}+\frac{2}{2n^2h\pi}\exp\left(-\frac{1}{h}\int_0^1 \sqrt{(s^2-2n^2h^2)_+}\,ds\right)\bigl(1+o(1)\bigr)\,.\]
\end{itemize}

\subsection{Bohr-Sommerfeld rule}

We turn now to the study of eigenvalues corresponding to energies above  $\sup_{x\in (0,1)} V(x)\,$, i.e.  when the energy satisfies $E>1\,$, since $V(x)=x^2$.  This implies in particular that the condition in \eqref{eq:cond-mu} is violated. In simple terms, we consider eigenvalues with label $k$ satisfying
\[
1<\rho_0\leq 4kh
\leq
\rho_1\,,
\]
for some constants $\rho_0,\rho_1\,$.  This condition emerges naturally if  we think of  a Bohr-Sommerfeld quantization rule for energy levels that are $>1$ (see Section~\ref{s3}).

The asymptotic distribution of the eigenvalues is described in the Theorem~\ref{thm:supercritical-intro} below, in a regime where there is no
tunneling phenomenon as described in Theorem~\ref{thm:main}.

\begin{theorem}[Shifted Bohr-Sommerfeld rule]\label{thm:supercritical-intro}
Let $1<\rho_0<\rho_1<+\infty$\,. Then, uniformly for integers $k$ satisfying
\[
\rho_0\leq 4kh\leq \rho_1\,,
\]
we have, as $h\to0^+$\,,
\[
\lambda_k(h)
=
E_0\left(
4h\left(k+\frac{\nu}{2}-\frac14\right)
\right)
+\mathcal O(h^2)\,,
\]
where, for $\rho>1\,$,  $E_0(\rho)>1$ is uniquely determined by 
\[
I(E_0(\rho))=\frac{\pi\rho}{4}\,,
\]
with
\[
I(E)
=
\int_0^1\sqrt{E-y^2}\,dy
=
\frac12\sqrt{E-1}
+
\frac{E}{2}\arcsin(E^{-1/2})\,,
\qquad E>1\,.
\]
\end{theorem}
\begin{remark}\label{rempostth}
Theorem~\ref{thm:supercritical-intro} exhibits a
Bohr--Sommerfeld quantization rule adapted to the singular potential,
in which the eigenvalue label enters through the shifted quantity
$
k+\frac{\nu}{2}-\frac14\,,
$ as predicted in \eqref{eq:BS-generic-V}.
The correction $\nu/2-1/4$ reflects the contribution of the inverse
square singularity at $x=0$ and can be interpreted as a
Bohr-Sommerfeld-Langer correction. Notice in particular that, when $\nu=1/2$, the inverse square term
vanishes and the shifted quantity
$
k+\frac{\nu}{2}-\frac14
$
reduces simply to $k$. Accordingly, the Bohr--Sommerfeld quantization condition takes the form
$
I(E)=\pi kh\,.
$
Since $E>1$, there are no turning points in $(0,1)$, and this is
precisely the usual semiclassical quantization rule for a classically
allowed interval with Dirichlet boundary conditions at both endpoints, (see, for instance, \cite[Chapter~6, Section~9, Eq.~(9.05)]{Olver}). This interpretation will be discussed in more detail in
Section~\ref{s3} and in Appendix~\ref{AppA}.
\end{remark}

\subsection{Transition regime}
In order to understand the transition between the tunneling asymptotics (Theorem~\ref{thm:main})  and the shifted Bohr-Sommerfeld rule (Theorem~\ref{thm:supercritical-intro}), we investigate the regime
\[
	|4kh-1|=t h^\delta,
	\qquad
	t\in[t_0,t_1],
	\qquad
	0<\delta<\frac23.
\]
which we refer to it as critical.  This regime is interesting because the classical turning point
can reach the boundary both from the interior of the domain (when $4kh-1<-t_0h^\delta$) and also from  the exterior (when $4kh-1>  t_0 h^\delta$). Actually the nature of the transition from the left 
will be treated in Proposition \ref{prop:2.5} and Remark \ref{rem:2.6} and  is quite different from the right which is treated for $\delta < 2/3$ in Proposition \ref{lem:rough-transition}.
Nevertheless, when
 $\delta=2/3$ and $|4kh-1|\leq T h^{2/3}\,$, we obtain a
three-term expansion for the $k$'th eigenvalue.
\begin{theorem}\label{thm:transition-intro}
There exists a smooth decreasing function $\theta:\mathbb R\to\mathbb R^+$ such that, 	for every $T>0$ and for integers $k$ satisfying
	\begin{equation}\label{eq:transition-window-intro}
		|4kh-1|\leq T h^{2/3}\,,
	\end{equation}
	setting
	\begin{equation}\label{eq:tau-k-intro}
		\tau_k(h)
		:=
		\frac{1-4kh}{2^{2/3}h^{2/3}}\,,
	\end{equation}
	then, uniformly for these $k$'s, the following holds
	\begin{equation}\label{eq:transition-expansion-intro}
		\begin{split}
			\lambda_k(h)
			={}&e_k(h)
			+\frac{4h}{\pi}\theta\bigl(\tau_k(h)\bigr)\\
			&-\frac{4}{\pi 2^{2/3}}h^{4/3}
			\left(
			2\nu-2+\frac{4}{\pi}\theta\bigl(\tau_k(h)\bigr)
			\right)
			\theta'\bigl(\tau_k(h)\bigr)
			+\mathcal O(h^{5/3})\,,
		\end{split}
	\end{equation}
	as $h\to0^+\,$. 
\end{theorem}

The function $\theta$ is explicitely introduced in  \eqref{eq:Airy-phase}, as a phase associated with the Airy functions.
When we restrict to $|4kh-1|= th^{\delta}\,$, with $t$ fixed, we get by the smoothness of $\theta$ the leading order asymptotics of $\lambda_k(h)-e_k(h)$, which recovers the tunneling asymptotics of Theorem~\ref{thm:main} in the subcritical side ($4kh-1=-th^\delta$ and $t\gg 1$), while on the supercritical side ($4kh-1=th^\delta$ and $t\gg 1$) it recovers the threshold expansion of the Bohr--Sommerfeld formula in Theorem~\ref{thm:supercritical-intro} (more details are given in Proposition~\ref{lem:rough-transition} and Subsection~\ref{subsec:matching}). 

In the regime $\delta=2/3$ and $t=t(h)$ large, we do not obtain an asymptotics, but we prove that $\lambda_k(h)$ is still exponentially close to $e_k(h)$, see Proposition~\ref{prop:2.5} and Remark~\ref{rem:2.6}.

\subsection{Spectral gap}  For the applications in control theory, one needs 
a uniform estimate of the gap\footnote{We thank R. Vanlaere for pointing this question to us.}
\[\lambda_{k+1}(h)-\lambda_k(h)\]
valid for all $k\geq 1$ and $h>0\,$. The asymptotics in Theorems~\ref{thm:ext-Gannot}, \ref{thm:main} and \ref{thm:transition-intro} yield the asymptotic lower bound
\[\lambda_{k+1}(h)-\lambda_k(h)\geq 4h+o(h)\]
in the relevant spectral windows. 

We  present a non-asymptotic bound  consistent with the eigenvalue asymptotics (also discovered independently by R. Vanlaere\,\footnote{Personal communication.}).
\begin{theorem}\label{thm:gap}
For all $\nu\geq 0$, $h>0$ and $k\geq 1$, it holds
\begin{equation}\label{eq:gap}
\lambda_{k+1}(h)-\lambda_k(h)>4h\, .
\end{equation}
\end{theorem}

The proof of this theorem relies on the Sturm-Liouville theory and the introduction of a creation-like operator.

\subsection{Possible extensions and heuristics}
\subsubsection{Origin of the study}
Our operator $T_h$ arises  naturally by fixing the angular momentum and keeping the radial variable, leaving us with the inverse square potential plus the harmonic potential. More generally, we can study the operator on $(0,1)$ with a general potential
\begin{equation}\label{eq:generic-operator}
T_{h,V}=-h^2 \frac{d^2}{dx^2}+h^2\frac{\nu^2-1/4}{x^2}+V\,.
\end{equation}
\subsubsection{Tunneling} Denoting the  Dirichlet and Neumann eigenvalues of $T_{h,V}$ by $\lambda_k^D(h,V)$ and $\lambda_k^N(h,V)$ respectively, and looking at the energy levels satisfying
\[\min_{[0,1]} V<\rho_* \leq \lambda_k^N(h,V)\leq \rho^*<\max_{[0,1]}V\,,\] 
it would be desirable to study whether we have
\begin{equation}\label{eq:splitting-ND}\lambda_k^D(h,V)-\lambda_k^N(h,V)\sim C_k(h)e^{-2S_{k,h}(V)/h}\,.
\end{equation}
The exponential decay is defined by the Agmon distance,
\[S_{k,h}(V)=\int_0^1\sqrt{(V(s)-E)_+}\,ds\,, \quad E=\frac 12 \bigl(\lambda_k^D(h,V)+\lambda_k^N(h,V)\bigr)\,,\]
while the prefactor $C_k(h)$ is expected to be of the form $\alpha h^{\beta}$, where both $\alpha$ ($\alpha\neq 0$) and $\beta$ depend possibly on $k$, $V$ and $\nu$. 

Note that it is possible to prove the weaker upper bound
\[\log (\lambda_k^D(h,V)-\lambda_k^N(h,V))\leq -\frac{2S_{k,h}(V)}{h}\bigl(1+o(1)\bigr)\,,\]
via  Agmon estimates (we illustrate this for the harmonic potential in Subsection~\ref{subsec:Agmon}).
If $V\in C^\infty(\R)$ is even, vanishes at $x=0\,$, positive for $x\not=0$, and satisfies
\[ V''(0)>0\,,\quad \lim_{|x|\to+\infty}V(x)>0\,,\]
then for a fixed $k$\,, the following asymptotics is proven by Gannot \cite[Theorem~2]{Gannot}
\begin{equation}\label{eq:Gannot-V}\lambda_k^D(h,V)=e_k(h,V)+\alpha_0(V) h^{-2k-\nu+2}e^{-2S_0(V)/h}\bigl(1+o(1)\bigr)\,,\end{equation}
where $e_k(h,V)$ is the eigenvalue of the operator
\[T_h^\infty=-h^2\frac{d^2}{dx^2}+h^2\frac{\nu^2-1/4}{x^2}+V\quad\text{in $\R_+$\,.}\]
The constants $S_0(V)$ and $\alpha_0(V)$ are explicit,
\[S_0(V)=\int_0^1\sqrt{V(t)}\,dt\,,\quad \alpha_0(V)=\frac{4 A\sqrt{V(1)}}{(k-1)!\Gamma(k+\nu)}\Bigl(\frac{V''(0)}{2}\Bigr)^{2k+\nu-1}\,,\]
where $A$ is dependent on $V$ and computed as a limit. 

We expect that 
\[\lambda_k^N(h,V)=e_k(h,V)-\alpha_0(V) h^{-2k-\nu+2}e^{-2S_0(V)/h}\bigl(1+o(1)\bigr)\,,\]
which would answer \eqref{eq:splitting-ND} for fixed $k$. That said, it would be interesting to extend \eqref{eq:Gannot-V} for higher energy levels as we did in Theorem~\ref{thm:main} for the harmonic potential $V(x)=x^2\,$.

The formula \eqref{eq:splitting-ND} is related to tunneling in double wells.   
In the particular case of a unique non-degenerate well at $x=0\,$, we extend
\[V_h(x)=h^2\frac{\nu^2-1/4}{x^2}+V(x)\,,\]
by even symmetry about $x=1\,.$ Neglecting the inverse square potential by taking $\nu=1/2\,$, we obtain a potential $\tilde V$ with double wells at $x=0$ and $x=2$, and the operator
\[\tilde T_h=-h^2\frac{d^2}{dx^2}+\tilde V\,,\]
in $L^2(0,2)\,$, with Dirichlet boundary condition at $x=0,2$. By the min-max principle, $\lambda_1^N(h,V)$ is the lowest eigenvalue of $\tilde T_h\,,$ while $\lambda_1^D(h,V)$ is the second eigenvalue of $\tilde T_h\,$. Thus, we have reduced to analyze the tunneling in double wells \cite{Ha, HS}. Thanks to well known tunneling asymptotics (\cite[p.~59, Eq. (4.5.21)]{He})
\[\lambda_1^D(h,V)-\lambda_1^N(h,V)=\lambda_2(h,\tilde V)-\lambda_1(h,\tilde V)= c_0\pi^{-1/2}h^{1/2}e^{-2S_0(V)/h}\bigl(1+o(1)\bigr)\,, \]
where $c_0>0$ is a constant dependent on $V$ and computed as a limit.
Notice that this sketch is only heuristic since the potentials are not regular.

\subsubsection{Formal derivation of a Bohr-Sommerfeld rule with Langer correction}
We turn to the study of the eigenvalues of the Dirichlet realization on $(0,1)$ of the operator $T_{h,V}$ introduced in \eqref{eq:generic-operator}. For $\nu=\frac12\,$, the inverse square potential disappears and the problem becomes regular. We  follow the dependence on the parameter $\nu$ in the vicinity of the regular case $\nu=1/2\,$. We denote the $k$'th eigenvalue by
\[\cE_k(\nu)=\lambda_k^D(h,V)\]
and we denote by $u_{k,\nu}$ an associated
$L^2$-normalized eigenfunction. We set
\[
E=\cE_k\left(\frac12\right)\,,
\]
and we assume that
\[
E>\max_{x\in[0,1]}V(x)\,.
\]
By the Feynman--Hellmann formula,
\[
\partial_\nu \cE_k(\nu)
=
2h^2\nu\int_0^1\frac{|u_{k,\nu}(x)|^2}{x^2}\,dx\,.
\]
We evaluate this expression heuristically at $\nu=\frac12$\,.
We introduce
\[
p(x,E)=\sqrt{E-V(x)}\,,
\qquad
I_V(E)=\int_0^1p(x,E)\,dx\,.
\]
For $\nu=1/2\,$, a normalized WKB eigenfunction has the form
\[
u_{k,1/2}(x)\sim
\frac{A}{\sqrt{p(x,E)}}
\sin\left(\frac1h\int_0^x p(s,E)\,ds\right)\,.
\]
Averaging the oscillations and using the $L^2$-normalization gives
\[
1\sim
\frac{A^2}{2}\int_0^1\frac{dx}{p(x,E)}
=
A^2I_V'(E)\,,
\]
and hence
\[
A^2\sim\frac{1}{I_V'(E)}\,.
\]
Making the change of variable $x=h\,y\,$, we obtain
\[
\int_0^1\frac{|u_{k,1/2}(x)|^2}{x^2}\,dx
=
\frac{1}{h}\int_0^{1/h}\frac{|u_{k,1/2}(hy)|^2}{y^2}\,dy\,.
\]
The singular weight suggests that the leading contribution comes from
the boundary scale $x=\mathcal O(h)\,$. Writing
\[
p_0=\sqrt{E-V(0)}\,,
\]
we obtain, for $y=\mathcal O(1)$\,,
\[
u_{k,1/2}(hy)\sim
\frac{A}{\sqrt{p_0}}\sin(p_0y)\,.
\]
Therefore, formally,
\[
\int_0^1\frac{|u_{k,1/2}(x)|^2}{x^2}\,dx
\sim
\frac{A^2}{h}
\int_0^\infty\frac{\sin^2 z}{z^2}\,dz
=
\frac{\pi A^2}{2h}
=
\frac{\pi}{2hI_V'(E)}\,,
\]
where we used the change of variable $z=p_0\, y$ and the classical
Dirichlet integral. It follows that
\[
\left.\partial_\nu \cE_k(\nu)\right|_{\nu=1/2}
\sim
\frac{\pi h}{2I_V'(E)}\,.
\]
Since
\[
\partial_\nu I_V(\cE_k(\nu))
=
I_V'(\cE_k(\nu))\,\partial_\nu \cE_k(\nu)\,,
\]
we obtain
\[
\left.
\partial_\nu I_V(\cE_k(\nu))
\right|_{\nu=1/2}
\sim
\frac{\pi h}{2}\,.
\]
A first-order Taylor expansion with respect to $\nu$ at $\nu=\frac12$
then gives, formally,
\[
I_V(\cE_k(\nu))
\sim
I_V(
\cE_k(\tfrac12))
+
\frac{\pi h}{2}\left(\nu-\frac12\right)\,.
\]
For $\nu=\frac12\,$, the problem is regular and the usual
Bohr--Sommerfeld quantization (see Olver \cite{Olver}) gives, since we have the Dirichlet condition at the two ends of the interval, 
\[
I_V(\cE_k(\tfrac12))=\pi hk+\mathcal O(h^2)\,.
\]
Consequently, we are led heuristically to
\begin{equation}\label{eq:BS-generic-V}
I_V(\cE_k(\nu))
\sim
\pi h\left(k+\frac{\nu}{2}-\frac14\right)\,.
\end{equation}
We have established \eqref{eq:BS-generic-V} for all $\nu\geq 0$ in the harmonic potential case, $V(x)=x^2\,$,  by exploiting the explicit solvability in terms of special functions; see Theorem~\ref{thm:supercritical-intro} above.

\subsubsection{Role of Special functions}
In the case of the harmonic potential $V(x)=x^2$, 
the eigenvalues are characterized in terms of the zeros of  the Kummer confluent hypergeometric function. That remarkable observation has been used in several papers \cite{BaurWeidl, Vanlaere}.

For a given eigenvalue $\lambda$ of $T_h$, a corresponding eigenfunction regular at the origin is given by
\begin{equation}\label{regular-solution}
u(x)
=
x^{\nu+\frac12}
e^{-x^2/2h}
M\!\left(
-\frac{\lambda-2(1+\nu)h}{4h}\,,
\,1+\nu\,,\,
x^2/h
\right)\,,
\end{equation}
where \(M(a,b,z)\), also denoted by \({}_1F_1(a;b;z)\)\,, is the Kummer
confluent hypergeometric function (see
\cite[pp.~262--264]{MOS1966} or
\cite[\S13.2, Eq.~13.2.4 and \S13.3]{DLMF}).
Imposing the Dirichlet boundary condition at \(x=1\) leads to the transcendental equation 
\begin{equation}\label{spectral-equation}
M\!\left(
-\frac{\lambda-2(1+\nu)h}{4h}\,,
\,1+\nu\,,\,
1/h
\right)
=0\,,
\end{equation}
and reduces the eigenvalue problem to the study of the zeros of the Kummer function with respect
to its first parameter (see \cite[Theorem~1.2]{Vanlaere}).

Although the eigenvalue condition is naturally expressed in terms of
the Kummer function, its asymptotic analysis is more conveniently
performed in the Whittaker setting. Indeed, the classical identity
\cite[\S13.14, Eq.~13.14.2]{DLMF}
\begin{equation}\label{Kummer-Whittaker}
M_{\kappa,\mu}(z)
=
e^{-z/2}z^{\mu+\frac12}
M\!\left(\mu-\kappa+\frac12\,,\,2\mu+1,\,z\right)
\end{equation}
shows that the two formulations are equivalent, where
\(M_{\kappa,\mu}\) denotes the Whittaker function. With
\begin{equation}\label{kappa-mu}
\kappa
=
\frac{\lambda}{4h}\,,
\qquad
\mu=\frac{\nu}{2}\,,
\end{equation}
the equation
\eqref{spectral-equation} becomes
\begin{equation}\label{Whittaker-spectral}
M_{\kappa,\mu}(1/h)=0\,.
\end{equation}
Thus, for $\lambda$ to be an eigenvalue, $z_h=1/h\gg1$ must be a zero of the Whittaker function.  For large $\kappa$, the asymptotic distribution of the positive zeros
of $M_{\kappa,\mu}(z)$ with respect to the variable $z$ was described
by Gabutti and Gatteschi~\cite{GG2001}.
By inverting their asymptotic expansion in the spectral relation
$M_{\kappa,\mu}(1/h)=0$\,, we obtain the eigenvalue asymptotics in Theorem~\ref{thm:supercritical-intro}.
\subsubsection{Summarizing picture}~\\
The different spectral regimes considered in this paper, together with
the corresponding asymptotic results, are summarized in Figure~\ref{fig:spectral-regimes*}.

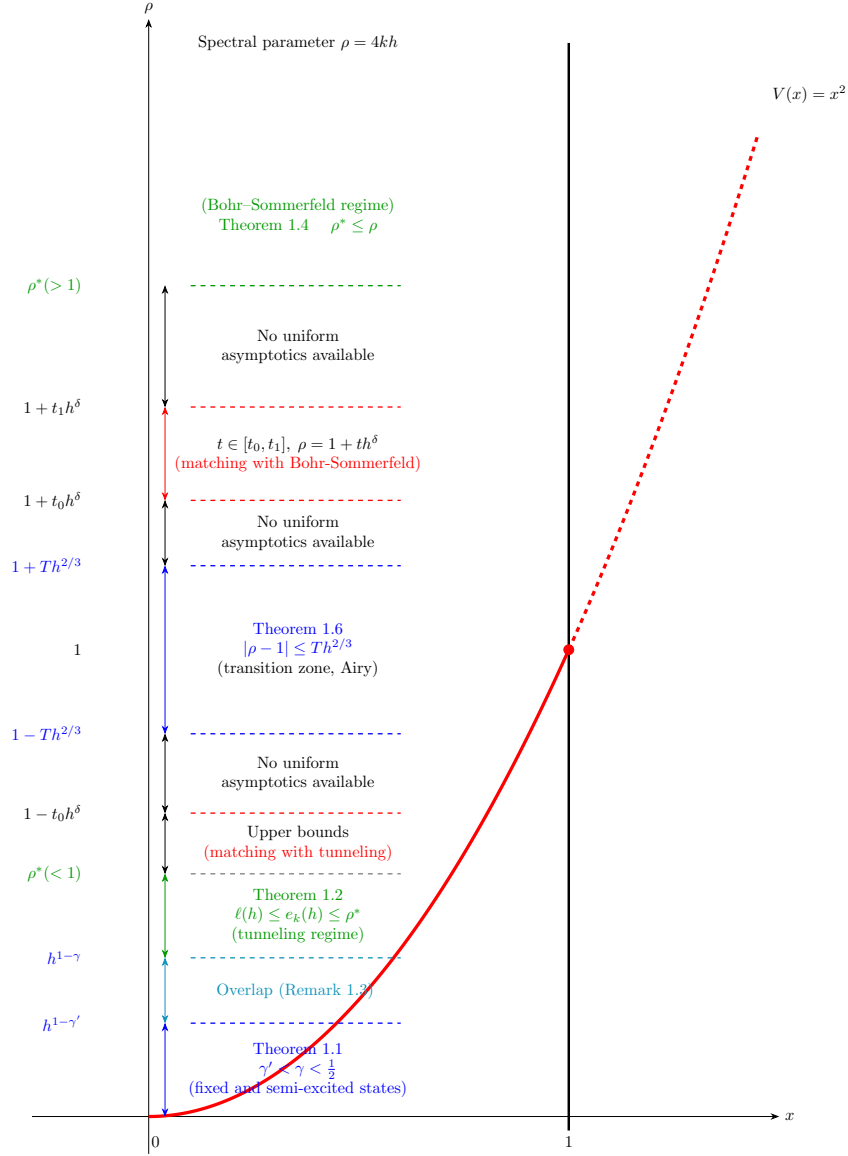
\begin{figure}[h]
  \centering
  \resizebox{0.7\textwidth}{!}{%
\begin{tikzpicture}[
  >=Stealth,
  every node/.style={font=\small},
  hline/.style={dashed, line width=0.8pt},
]


\def\xone{9}        
\def\xmax{13}       
\def\ymax{23.5}     

\def\yhmg{2.0}       
\def\yhmg{2.0}
\def\yhmgp{3.4}      
\def\yrhosl{5.2}     
\def\yrhotl{6.5}     
\def\yrhoTh{8.2}     
\def\yrhoone{10.0}   
\def\yrhoThh{11.8}   
\def\yrhotoh{13.2}   
\def\yrhot1h{15.2}   
\def\yrhosh{17.8}    

\draw[->] (-2.5,0) -- (13.5,0) node[right] {$x$};
\draw[->] (0,-0.8) -- (0,\ymax) node[above] {$\rho$};

\draw[line width=1.5pt] (\xone,-0.3) -- (\xone,\ymax-0.5);
\node[below] at (\xone,-0.3) {$1$};
\node[below] at (0.15,-0.3) {$0$};

\draw[red, line width=2pt]
  plot[domain=0:1, samples=200, variable=\t]
  ({\t*\xone}, {\t*\t*10});

\draw[red, line width=2pt, dashed]
  plot[domain=1:1.45, samples=100, variable=\t]
  ({\t*\xone}, {\t*\t*10});

\node[right] at ({1.47*\xone}, {1.47*1.47*10+0.3}) {$V(x)=x^2$};

\filldraw[red] (\xone, \yrhoone) circle (3pt);

\def\hxl{0.9}   
\def\hxr{5.4}   

\draw[hline, blue]  (\hxl,\yhmg)  -- (\hxr,\yhmg);
\draw[hline, cyan!70!black] (\hxl,\yhmgp) -- (\hxr,\yhmgp);
\draw[hline, gray]  (\hxl,\yrhosl) -- (\hxr,\yrhosl);
\draw[hline, red]   (\hxl,\yrhotl) -- (\hxr,\yrhotl);
\draw[hline, blue]  (\hxl,\yrhoTh) -- (\hxr,\yrhoTh);
\draw[hline, blue]  (\hxl,\yrhoThh) -- (\hxr,\yrhoThh);
\draw[hline, red]   (\hxl,\yrhotoh) -- (\hxr,\yrhotoh);
\draw[hline, red]   (\hxl,\yrhot1h) -- (\hxr,\yrhot1h);
\draw[hline, green!60!black] (\hxl,\yrhosh) -- (\hxr,\yrhosh);

\def\lx{-1.3}   

\node[left] at (\lx, \yhmg)    {\textcolor{blue}{$h^{1-\gamma'}$}};
\node[left] at (\lx, \yhmgp)   {\textcolor{blue}{$h^{1-\gamma}$}};
\node[left] at (\lx, \yrhosl)  {\textcolor{green!60!black}{$\rho^*(<1)$}};
\node[left] at (\lx, \yrhotl)  {$1-t_0 h^\delta$};
\node[left] at (\lx, \yrhoTh)  {\textcolor{blue}{$1-Th^{2/3}$}};
\node[left] at (\lx, \yrhoone) {$1$};
\node[left] at (\lx, \yrhoThh) {\textcolor{blue}{$1+Th^{2/3}$}};
\node[left] at (\lx, \yrhotoh) {$1+t_0 h^\delta$};
\node[left] at (\lx, \yrhot1h) {$1+t_1 h^\delta$};
\node[left] at (\lx, \yrhosh)  {\textcolor{green!60!black}{$\rho^*(>1)$}};

\def\ax{0.35}   

\draw[<->, blue]          (\ax,0)        -- (\ax,\yhmg);
\draw[<->, cyan!70!black] (\ax,\yhmg)   -- (\ax,\yhmgp);
\draw[<->, green!60!black](\ax,\yhmgp)  -- (\ax,\yrhosl);
\draw[<->]                (\ax,\yrhosl) -- (\ax,\yrhotl);
\draw[<->]                (\ax,\yrhotl) -- (\ax,\yrhoTh);
\draw[<->, blue]          (\ax,\yrhoTh) -- (\ax,\yrhoThh);
\draw[<->]                (\ax,\yrhoThh)-- (\ax,\yrhotoh);
\draw[<->, red]           (\ax,\yrhotoh)-- (\ax,\yrhot1h);
\draw[<->]                (\ax,\yrhot1h)-- (\ax,\yrhosh);

\def\tx{3.2}   

\node[align=center] at (\tx, {(\yhmg)/2}) {
  \textcolor{blue}{Theorem 1.1}\\
  \textcolor{blue}{$\gamma'<\gamma<\frac{1}{2}$}\\
  \textcolor{blue}{(fixed and semi-excited states)}
};
%
\node[align=center] at (\tx, {(\yhmg+\yhmgp)/2}) {
  \textcolor{cyan!70!black}{Overlap (Remark 1.3)}
};
%
\node[align=center] at (\tx, {(\yhmgp+\yrhosl)/2}) {
  \textcolor{green!60!black}{Theorem 1.2}\\
  \textcolor{green!60!black}{$\ell(h)\leq e_k(h)\leq\rho^*$}\\
  \textcolor{green!60!black}{(tunneling regime)}
};

\node[align=center] at (\tx, {(\yrhosl+\yrhotl)/2}) {
  Upper bounds\\
  \textcolor{red}{(matching with tunneling)}
};

\node[align=center] at (\tx, {(\yrhotl+\yrhoTh)/2}) {
  No uniform\\
  asymptotics available
};

\node[align=center] at (\tx, {(\yrhoTh+\yrhoThh)/2}) {
  \textcolor{blue}{Theorem 1.6}\\
  \textcolor{blue}{$|\rho-1|\leq Th^{2/3}$}\\
  (transition zone, Airy)
};

\node[align=center] at (\tx, {(\yrhoThh+\yrhotoh)/2}) {
  No uniform\\
  asymptotics available
};

\node[align=center] at (\tx, {(\yrhotoh+\yrhot1h)/2}) {
 %
  $t\in[t_0,t_1],\;\rho=1+th^\delta$\\
  \textcolor{red}{(matching with Bohr-Sommerfeld)}
};

\node[align=center] at (\tx, {(\yrhot1h+\yrhosh)/2}) {
  No uniform\\
  asymptotics available
};

\node[align=center] at (\tx, {\yrhosh+1.5}) {
  \textcolor{green!60!black}{(Bohr--Sommerfeld regime)}\\
  \textcolor{green!60!black}{Theorem 1.4 $\quad\rho^*\leq\rho$}
};

\node at (\tx, \ymax-0.5) {Spectral parameter $\rho=4kh$};

\end{tikzpicture}}
  \caption{Summary of the different spectral regimes and the corresponding asymptotic results.} \label{fig:spectral-regimes*}
\end{figure}


%

\subsection{Organization}

The paper is organized as follows. Section~\ref{s2} is devoted to the tunneling
regime and to the proofs of Theorems~\ref{thm:ext-Gannot} and~\ref{thm:main}, including the low-energy
regime and a brief discussion of the Neumann problem.
Section~\ref{s3} is devoted to the supercritical regime and the proof of Theorem~\ref{thm:supercritical-intro}, using asymptotic
properties of Whittaker functions and their zeros.
Section~\ref{sec:transition} treats the transition near the critical energy $1$ as in Theorem~\ref{thm:transition-intro}, described
by uniform Airy asymptotics, and discusses its matching with the
subcritical and supercritical regimes.  Section~\ref{sec:gap} proves Theorem~\ref{thm:gap}.
Appendix~A discusses the Bohr--Sommerfeld--Langer quantization, while
Appendix~B reformulates our spectral results as asymptotic formulas for
the zeros of the Kummer function with respect to its first parameter.


\section{Tunneling asymptotics}\label{s2}

This section is devoted to the proofs of Theorems~\ref{thm:ext-Gannot} and~\ref{thm:main}. 
The proof of Theorem~\ref{thm:main} relies on three ingredients:
\begin{itemize}
    \item Agmon estimates (Subsection~\ref{subsec:Agmon}),
    \item Plancherel--Rotach asymptotics in the classically forbidden region (Subsection~\ref{subsec:Plancherel-Rotach}),
    \item Construction of an explicit trial state in the domain of the operator $T_h$ (Subsection~\ref{subsec:trial-state}).
\end{itemize}
The proof of Theorem~\ref{thm:ext-Gannot} follows the same scheme, with the Plancherel--Rotach asymptotics being replaced by elementary properties of the associated Laguerre polynomials (Subsection~\ref{subsec:vl-regime}). The section is concluded with a brief presentation of the Neumann problem (Subsection~\ref{subsec:Neumann}).
\subsection{Eigenfunction on $\mathbb R_+$}\label{subsec:Phi}~\\
Given a positive integer $k$, we introduce the function
    \begin{equation}\label{eq:def-Phi*}
\Phi_k (z) =C_k \, z^{\nu+\frac12}\,L_{k-1}^{(\nu)}(z^2)\, e^{-z^2/2}\,,
    \end{equation}
    where $L_{k-1}^{(\nu)}$ is the associated Laguerre polynomial of degree $k-1$ and $C_k$ is a positive constant  chosen so that $\Phi_k$ is normalized in $L^2(\R_+)$. 
We calculate $C_k$ from the orthogonality relations for Laguerre polynomials
\[
\int_{0}^{\infty} y^{\nu} L_{m}^{(\nu)}(y) L_{n}^{(\nu)}(y) e^{-y} \, dy = \frac{\Gamma(n + \nu + 1)}{n!} \delta_{mn}\,,
\]
which yield
\[\int_{0}^{\infty} y^{\nu} |L_{n}^{(\nu)}(y)|^2 e^{-y} \, dy = \frac{\Gamma(n + \nu + 1)}{n!},\quad n=0,1,2\ldots.\]
Thus, for $n=k-1$, we obtain by performing the change of variable $y=z^2$ that 
\begin{equation}\label{eq:vl-Ck}
C_k=\sqrt{\frac{2(k-1)!}{\Gamma(k+\nu)}}\,,
\end{equation}
and by  Stirling's approximation as $k\to+\infty\,$, we have
 \begin{equation}\label{eq:C-kh}
 C_{k}= 2^{\frac12}k^{-\frac{\nu}{2}}\biggl(1-\frac{\nu(\nu-1)}{4k}+\mathcal O\Bigl(\frac1{k^2}\Bigr)\biggr)\,.
    \end{equation}
   Now we introduce an $h$-dependent function defined on $\R_+$ as
    \begin{equation}\label{eq:def-Phi}
    \Phi_{k,h}(x)=h^{-1/4} \Phi_{k} \bigl(x/\sqrt{h}\bigr)\,,
    \end{equation}
where $\Phi_k$ is the function in \eqref{eq:def-Phi*}.     Note that $\Phi_{k,1}=\Phi_k$.  The function $\Phi_{k,h}$ is normalized in $L^2(\R_+)$ and  satisfies on $\mathbb R_+$,
    \[T_h\Phi_{k,h}=e_k(h)\Phi_{k,h}\,.\]
It will be used to construct a trial state satisfying the Dirichlet boundary condition at $x=1$. Towards that end, we need accurate information about the decay of $\Phi_{k,h}$ in the classically forbidden region $\{x\geq \sqrt{e_k(h)}\}\,$. 
\subsection{Agmon estimate and applications}\label{subsec:Agmon}~\\ 
We start with an Agmon estimate\footnote{The  Agmon estimate in Proposition~\ref{prop:dc-Phi} is valid in a neighborhood of $+\infty$ and can be easily extended to more general potentials.}. The novelty in comparison with for example \cite{HS} is that there is a singularity due to $\nu$.

    \begin{proposition}[rough decay of $\Phi_{k,h}$]\label{prop:dc-Phi}
    Let $0<\rho^*<\bar\rho<1\,$. There exist constants $C,h_0>0$ such that, for all $h\in(0,h_0]$ and for all integers $k$ satisfying $e_k(h)\leq \rho^*$, it holds 
    \[ \int_{0}^{+\infty}|e^{\beta/h}\Phi_{k,h}|^2dx\leq C\,,\quad \int_{\bar\rho^{1/2}}^{+\infty}|e^{\beta/h}\Phi'_{k,h}|^2\,dx\leq C\, h^{-2}\,,\]
    where 
    \[\beta(x)=\int_0^x\sqrt{(s^2-\bar\rho)_+}\,ds\,.\]
\end{proposition}
\begin{proof}
For simplicity, we skip the reference to $k$ and $h$ and write $\Phi=\Phi_{k,h}\,$. Let $ u=e^{ \beta/h}\Phi$ and $\tilde u=e^{\tilde \beta/h}\Phi\,$, where $\tilde \beta$ is defined by
\[ \tilde \beta(x)=\begin{cases}
    \beta(x)&x< R\,,\\
    \beta(R)&x\geq R\,.
\end{cases}\]
Here $R>1$. Writing
    \[\langle (T_h-e_k(h))\Phi,e^{2\tilde \beta/h}\Phi\rangle=0\,,\]
    and integrating by parts (on $\mathbb R_+$), we obtain the identity
    \[\int_0^{+\infty}\Big( |h \tilde u '|^2 +h^2\frac{\nu^2-1/4}{x^2}| \tilde u|^2+(x^2-e_k(h)-|\tilde \beta'|^2)| \tilde u|^2\Bigr)\,dx=0\,,\]
    which can be written in the form
     \[
     \begin{aligned}
     \int_0^{+\infty} |h \tilde u'|^2\,dx\, +\, h^2\int_{0}^{+\infty} \frac{\nu^2-1/4}{x^2}|\tilde u|^2\,dx +  \int_{\bar \rho^{1/2}}^{+\infty}(x^2-e_k(h)-|\tilde \beta'|^2)| \tilde u|^2\,dx\\ 
     =  \int_{0}^{\bar \rho^{1/2}} (e_k(h)-x^2)\, | \tilde u|^2\,dx \,.
     \end{aligned}
     \]
    Using $|\tilde\beta'(x)|^2\leq (x^2-\bar\rho)_+$, $e_k(h)\leq\rho^*$, and the Hardy inequality
    \[\int_0^{\rho_1} \frac{|\tilde u(x)|^2}{x^2}\,dx\leq 4\int_0^{\rho_1}| \tilde u'(x)|^2\,dx\,,\]
 with $\rho_1=\bar\rho^{1/2}$, we deduce that,
 \[\int_{\bar\rho^{1/2}}^{+\infty}|h  \tilde u'|^2\,dx
    +\int_{\bar\rho^{1/2}}^{+\infty}\Bigl(\bar\rho-\rho^*-\frac{h^2}{4\bar\rho}\Bigr)|\tilde u|^2\,dx
    \leq \rho^*\int_{0}^{\bar\rho^{1/2}}|\tilde u|^2\leq \rho^*\,.\] 
Here we have also used that,  
for $x^2\leq \bar\rho\,,$ we have $\tilde\beta(x)=\beta(x)=0$ and $\tilde u=\Phi$, where $\Phi$ is normalized in  $\R_+$.\\

 We put $A=\frac{\bar\rho-\rho^*}{2}>0$ and we choose $h_0$ sufficiently small so that, for $h\in(0,h_0]\,$, we have  $\bar\rho-\rho^*-\frac{h^2}{4\bar\rho}\geq A\,$. Thus, 
    \[\int_{\bar\rho^{1/2}}^{+\infty} |h \tilde u'|^2\,dx
    +A\int_{0}^{+\infty}| \tilde u|^2\,dx\\
    \leq  2 A\int_0^{\bar\rho^{1/2}}|\Phi|^2\,dx+\rho^*\leq A+1
  \,.\] 
 Since $\tilde u=u$ on $(0,R)$, we get
     \[\int_{\bar\rho^{1/2}}^{R} |h  u'|^2\,dx
    +A\int_{0}^{R}|u|^2\,dx\\
    \leq 2A+1
  \,.\]
  Taking $R\to+\infty\,$, we get by monotone convergence that
    \[\int_{\bar\rho^{1/2}}^{+\infty}|{ hu'}|^2\,dx
    +A\int_0^{+\infty}|u|^2\leq C\,.\] With $C= (2A+1)$, we obtain the desired inequality.
\end{proof}

We present a short proof of Vanlaere's bound \eqref{eq:Vanlaere} via the  Agmon estimate in Proposition~\ref{prop:dc-Phi}  with a nearly optimal decay rate.

\begin{proposition}\label{prop:Vanlaere}
  Let $0<\rho^*<1\,$ and set
 \[\alpha=2\int_0^1\sqrt{(s^2-\rho^*)_+}\,ds.\]
For any $\eta\in(0,1)$, there exists $h_0>0$ such that, for all $h\in(0,h_0]$ and for all integers $k$ satisfying $e_k(h)\leq \rho^*\,$, we have
\[e_k(h)\leq \lambda_k(h)\leq e_k(h)+e^{-(1-\eta)\alpha/h}\,.\]
\end{proposition}
\begin{proof}
The lower bound  was established in \eqref{eq:1.3}.

We turn to the proof of the upper bound.  Let 
    \[0<\varepsilon<\frac{1-\rho^*}{4},\quad \bar\rho=\rho_*+\frac{\varepsilon}{2},\quad \alpha_\varepsilon
      =2\beta\bigl(1-\varepsilon\bigr),
    \] 
    where $\beta$ is defined in Proposition~\ref{prop:dc-Phi}. Choose a cut-off function $\chi\in C^\infty(\R)$ such that
    \[\chi=1\text{ on $(-\infty,1-\frac{\varepsilon}{2}],$}\quad \chi=0\text{ on $[1-\frac{\varepsilon}{4},+\infty)$}\,.\]
      Consider the subspace $M=\{\chi\Phi_{j,h}\colon 1\leq j\leq k\}$.  For sufficiently small $h$, it follows from Proposition~\ref{prop:dc-Phi} that
    $\dim M=k$ and that
    \[\sup_{u\in M\setminus\{0\}}\frac{\langle T_h\phi,\phi\rangle}{\|\phi\|^2}\leq e_k(h)+e^{-\alpha_\varepsilon/h}\,,\]
    since the commutator terms in $[\chi_h,T_h]\Phi_{j',h}\Phi_{j,h}$ (with $j\leq k\,,\, j'\leq k$) are supported in the interval  $(1-\frac{\varepsilon}{2},1-\frac{\varepsilon}{4})$, and $2\beta(1-\frac{\varepsilon}{2})>\alpha_\varepsilon$. 
    
    Thus, by the min-max principle, we obtain that $\lambda_k(h)\leq e_k(h)+e^{-\alpha_\varepsilon/h}.$  This proves the claimed upper bound, since for a given $\eta\in(0,1)$, we can choose $\varepsilon$ sufficiently small so that $|\alpha_\varepsilon -\alpha|<\eta$.\end{proof}

\begin{remark}\label{rem:prop2.1}
The following almost optimal upper bound results from Proposition~\ref{prop:Vanlaere},
\[0\leq \lambda_k(h)-e_k(h)\leq \exp\left(-\frac{\alpha+o(1)}{h}\right)\,.\]
To write down the asymptotics,  we will use precise decay estimates given by the Plancherel-Rotach asymptotics for the associated Laguerre polynomials.
\end{remark}

We will need the following  direct consequence of Proposition~\ref{prop:Vanlaere}.
\begin{corollary}\label{corol:gap}
    Let $0<\rho^*<1\,$. There exists $h_0>0$ such that, if $h\in(0,h_0]$ and $e_k(h)\leq \rho^*\,$, then
    \[|\lambda_j(h)-e_k(h)|\geq 2h\quad\text{for all $j\not=k$\,.}\]
\end{corollary}
\begin{proof}
For $j\geq k+1\,,$ the result in the corollary follows from the universal lower bound \eqref{eq:1.3}:
\[\lambda_j(h)\geq e_{k+1}(h)=e_k(h)+4h\,.\]
For $j\leq k-1,$ the result follows from the asymptotic upper bound in Proposition~\ref{prop:Vanlaere},
\[\lambda_j(h)\leq\lambda_{k-1}(h)\leq  e_{k-1}(h)+e^{-\alpha/h}=e_k(h)-4h+e^{-\alpha/h}\,.\]
\end{proof}

We can push the upper bound in Proposition~\ref{prop:Vanlaere} up to $\rho^*\sim 1$.

\begin{proposition}\label{prop:2.5}
 Let $t>0$ and $0<\delta<2/3$ be fixed.  For every $\eta\in(0,1)$, there exist $h_0>0$ such that, for all $h\in(0,h_0]$ and for all integers $k$ satisfying $e_k(h)\leq 1-th^\delta\,$, we have
\[\lambda_k(h)\leq e_k(h)+\exp\left(-\frac{2(1-\eta)}{3}t^{3/2}h^{\frac{3\delta}{2}-1}\right)\,.\]
\end{proposition}
\begin{proof}
    The argument used in the proof of Proposition~\ref{prop:dc-Phi} applies with
    $\rho^*=1-th^\delta$ and $\bar\rho=1-t_1h^\delta$, where $0<t_1<t$. It yields that  $\Phi_{k,h}$ satisfies
    \[ \int_{0}^{+\infty}|e^{\beta/h}\Phi_{k,h}|^2dx\leq Ch^{-\delta}\,,\quad \int_{\bar\rho^{1/2}}^{+\infty}|e^{\beta/h}\Phi'_{k,h}|^2dx\leq C\, h^{-2-\delta}\,,\]
    where
    \[\beta(x)=\int_0^{x}\sqrt{(s^2-\bar\rho)}\,ds=\int_0^x\sqrt{(s^2-1+t_1h^\delta)_+}\,ds.\]
    Note that, for $0\leq t_2<\frac12 t_1$, we have
    \[\beta(1-t_2h^\delta)=\frac13(t_1-2t_2)^{3/2}h^{3\delta/2}+\mathcal O(h^{5\delta/2}).\]
    We choose $\varepsilon\in(0,\frac12)$ sufficiently small but fixed,  $t_1=(1-\varepsilon)t$ and $t_2=\varepsilon t$. Let $\chi$ be a cut-off function satisfying
    \[\chi=1\text{ on $[0,1-t_2 h^\delta]$,}\quad \chi=0\text{ on $[1,+\infty)$.}\]
    With $M=\mathrm{Span}\{\chi\Phi_{j,h}\colon 1\leq j\leq k\},$ and $\alpha_\varepsilon=2\beta(1-\varepsilon th^\delta)$, we obtain
    \[\sup_{u\in M\setminus\{0\}}\frac{\langle T_h\phi,\phi\rangle}{\|\phi\|^2}\leq e_k(h)+h^{-2\delta}e^{-\alpha_\varepsilon/h}\,.\]
     Note that $\alpha_\varepsilon$ depends on $h$. For a given $\eta\in(0,1)$, we can choose $\varepsilon,h_0$ sufficently small such that
    \[\left|\alpha_\varepsilon+2\delta h\ln h-\frac{t^{3/2}}{3}h^{3\delta/2}\right|<\eta \frac{t^{3/2}}{2}h^{3\delta/2}\quad\text{for }h\in(0,h_0].\]
\end{proof}

It follows from Proposition~\ref{prop:2.5} that for 
a label $k=k(h)$ satisfying 
\[4kh=1-th^{\delta} +o(h^\delta)\,,\]
with $0<\delta<2/3\,$,
the corresponding eigenvalue $\lambda_k(h)$ satisfies
\begin{equation}
e_k(h)\leq \lambda_k(h)\leq e_k(h)+\exp\left(-\frac{2t^{3/2}h^{\frac{3\delta}{2}-1}}{3}\bigl(1+o(1)\bigr)\right)\,.
\end{equation}
Hence we are still in a regime where the correction is exponentially small.

\begin{remark}\label{rem:2.6}
A close inspection of the proof of Proposition~\ref{prop:2.5} shows that we can obtain the following generalization for $\delta =2/3$ and $t=t(h)\to+\infty$: For any $\eta\in(0,1)$ and $k$ satisfying
$e_k(h)\leq 1-t(h)h^{2/3}$, we have
\[e_k(h)\leq \lambda_k(h)\leq e_k(h) + \exp\left(-\frac{2(1-\eta)}{3}t(h)^{3/2}\right)\,,\]
for $h$ in a right neighborhood of $0\,$.
\end{remark}
\subsection{Plancherel-Rotach asymptotics}\label{subsec:Plancherel-Rotach}

We recall the formulas of Plancherel-Rotach type for Laguerre polynomials \cite[Theorem~8.22.8\,(b)]{Szego}. After a change of variable, this gives accurate asymptotics for the function in \eqref{eq:def-Phi} that are much  stronger than the rough Agmon estimate in Proposition~\ref{prop:dc-Phi}.

Let $\epsilon$ and $\omega$ be fixed positive numbers
 such that $\epsilon < \omega$. For $n$ sufficiently large,
$y = (4n + 2\nu + 2) \cosh^2 \phi$, and $\epsilon \leq \phi \leq \omega$, we have
\begin{equation}\label{eq:Plancherel-Rotach}
e^{-y/2} L_n^{(\nu)}(y) = \frac{1}{2} (-1)^n (\pi \sinh \phi)^{-1/2} y^{-\frac{\nu}{2} - \frac{1}{4}} n^{\frac{\nu}{2} - \frac{1}{4}} \exp\bigl(E_n(y)\bigr)\bigl\{ 1 +\mathcal  O(n^{-1}) \bigr\}.
\end{equation}
where
\[
E_n(y) = \left(n + \frac{\nu+1}{2}\right)(2\phi - \sinh 2\phi)=
-\int_{\sqrt{4n + 2\nu + 2}}^{\sqrt{y}} \sqrt{s^2 - (4n + 2\nu + 2)} \, ds\,,
\]
and $\mathcal O(n^{-1})$ is uniform with respect to $\phi\in [\epsilon,\omega]\,$.\mn

 We show that the Plancherel-Rotach asymptotics continue to hold if we take $\omega=+\infty$.

\begin{proposition}\label{prop:Plancehere-Rotach}
Let $\epsilon>0$ be  fixed. There exist $C,N_0>0$ such that, if $n\geq n_0$, $\phi\geq\epsilon$, and
$y = (4n + 2\nu + 2) \cosh^2 \phi,$ then
\[e^{-y/2} L_n^{(\nu)}(y) = \frac{1}{2} (-1)^n (\pi \sinh \phi)^{-1/2} y^{-\frac{\nu}{2} - \frac{1}{4}} n^{\frac{\nu}{2} - \frac{1}{4}} \exp\bigl(E_n(y)\bigr)\bigl\{ 1 +R_n(y) \bigr\},\]
with
\[|R_n(y)|\leq Cn^{-1}.\]
\end{proposition}
\begin{proof}
Fix $\nu\in\mathbb{R}$ and set
\[
N=4n+2\nu+2\,,
\qquad
y=N\cosh^2\phi\,,\, X=\cosh^2\phi\,.
\]
For $\phi\geq \epsilon\,$, we have $X>1\,$, and we define $\zeta>0$ as in \cite[Eq. (18.15.21)]{DLMF} by
\[
\zeta^{3/2}
=
\frac34\left(
\sqrt{X^2-X}
-
\operatorname{arccosh}\sqrt X
\right)=
\frac38\left(\sinh(2\phi)-2\phi\right).
\]
Thus,
\begin{equation}\label{eq:Nzeta}
-\frac23N\zeta^{3/2}
=
\frac N4\left(2\phi-\sinh(2\phi)\right)=E_n(y).
\end{equation}
By 
\cite[Eq. (18.15.22) for $M=1$]{DLMF},  we have
\[
\begin{split}
L^{(\nu)}_{n}\left(N X\right)=&(-1)^{n}\frac{{e}^{\frac{1}{2}N X
}}{2^{\nu-\frac{1}{2}}X^{\frac{1}{2}\nu+\frac{1}{4}}}\*\left(\frac{\zeta
}{X-1}\right)^{\frac{1}{4}}\times\\
&\left(\frac{\operatorname{Ai}\left(N^{\frac{2}{3}
}\zeta\right)}{N^{\frac{1}{3}}}
+\frac{\operatorname{Ai}'\left(N^{\frac{2}{3}}\zeta\right)}{N^{\frac{5}{3}
}}F_{0}(\zeta)+\operatorname{envAi}\left(N^
{\frac{2}{3}}\zeta\right)R_n(X)\right)\,,
\end{split}\]
where $\operatorname{envAi}$ is the envelop of the Airy function (see \cite[Eq. (2.8.20)]{DLMF})
\[ \operatorname{env}\Ai(x)
=\begin{cases}
\bigl(\Ai(x)^2+\Bi(x)^2\bigr)^{1/2},&
x\leq c,\\
\sqrt{2}\,\Ai(x)\,,&
x\geq c\,,
\end{cases}
\]
where $c$ is the real root of smallest absolute value of
$\Ai(c)=\Bi(c)$, namely
$
c\sim -0.36\ldots ,
$
$F_0$ is defined in \cite[Eq. (18.15.23)]{DLMF} as
\[F_0(\zeta)=-\frac{5}{48\zeta^{2}}+\left(\frac{X-1}{X\zeta}\right)^{\frac{1}{%
2}}\left(\frac{1}{2}\nu^{2}-\frac{1}{8}-\frac{1}{4}\frac{X}{X-1}+\frac{5}{2%
4}\left(\frac{X}{X-1}\right)^{2}\right)\,,\]
and
\[R_n(X)=\mathcal O\left(\frac{1}{N^{\frac{4}{3}}}\right)\,\]
uniformly for $\delta\leq X<+\infty\,$, for any $\delta>0\,$, hence for
$\epsilon\leq\phi<+\infty$\,.

Thus, we have an expansion of $L_n^{(\nu)}(y)\,$, since $y=N\,X\,$. 
Using the Airy asymptotic formulas as $z\to+\infty$ (see \cite[Eq. (9.7.5) and (9.7.6)]{DLMF})
\[\operatorname{Ai}\left(z\right)\sim\frac{e^{-\frac23z^{3/2}}}{2\sqrt{\pi}z^{1/4}},\quad \operatorname{Ai}'\left(z\right)\sim-\frac{z^{1/4}e^{-\frac23z^{3/2}}}{2\sqrt{\pi}}\,,\]
we obtain by \eqref{eq:Nzeta},
\[
e^{-y/2}L_n^{(\nu)}(y)
=
\frac{(-1)^n}{2\sqrt{\pi}}\,
2^{-\nu+\frac12}N^{-\frac12}
(\cosh\phi)^{-\nu-\frac12}
(\sinh\phi)^{-\frac12}
\exp\left(E_n(y)\right)
\left(1+\mathcal O(N^{-1})\right)\,,
\]
uniformly for \(\epsilon\leq\phi<+\infty.\)

It remains to observe that
\[
y^{-\nu/2-1/4}n^{\nu/2-1/4}
=
N^{-\nu/2-1/4}n^{\nu/2-1/4}
(\cosh\phi)^{-\nu-1/2}\,,
\]
and
\[
2^{-\nu+1/2}
\left(\frac{N}{n}\right)^{\nu/2-1/4}
=
1+\mathcal O(n^{-1})\,.
\]
Therefore
\[
\begin{aligned}
e^{-y/2}L_n^{(\nu)}(y)
=
\frac{(-1)^n}{2\sqrt{\pi\sinh\phi}}\,
y^{-\nu/2-1/4}
n^{\nu/2-1/4}
\exp\left(E_n(y)\right)
\left(1+\mathcal O(n^{-1})\right)\,,
\end{aligned}
\]
uniformly for
\(\epsilon\leq\phi<+\infty\,.
\)
\end{proof}

Let $1<A<+\infty$. In terms of the function $\Phi_k$ in \eqref{eq:def-Phi*},  the asymptotics in \eqref{eq:Plancherel-Rotach} reads as
\[\Phi_k(z) = C_k \frac{(-1)^{k-1} k^{\frac{\nu}{2}-\frac{1}{4}}}{2\sqrt{\pi}} \left(\frac{z^2}{e_k(1)}-1\right)^{-1/4} \exp(E_{k-1}(z^2)) \left\{1 + \mathcal{O}(k^{-1})\right\}\,,
\] 
 for $z\geq A\sqrt{e_k(1)}\,$.\\
Furthermore, by Proposition~\ref{prop:Plancehere-Rotach}, $\mathcal O(k^{-1})$ is uniform with respect to $z\geq A\sqrt{e_k(1)}.$

Doing the change of variable $z=x/\sqrt{h}$ and noting that
\[e_k(h)=he_k(1),\quad  E_{k-1}(z^2)=-S(x;e_k(h))/h\,,\]
we obtain, for $x\geq A\sqrt{e_k(h)}\,,$
\[\begin{split}
		\Phi_{k,h}(x)
		=
		\hat C_{k,h}
		\left(\frac{x^2}{e_k(h)}-1\right)^{-1/4}
		\exp\bigl(-S(x;e_k(h))/h\bigr)
		\bigl\{1+\mathcal O(k^{-1})\bigr\}\,,
	\end{split}
	\]
where $\Phi_{k,h}$ is introduced in \eqref{eq:def-Phi} and
\[\hat C_{k,h}=C_k\frac{(-1)^{k-1}k^{\frac{\nu}{2}-\frac{1}{4}}h^{-1/4}}{2\sqrt{\pi}}\,.\]

Later, we need an accurate asymptotics of $\Phi_{k,h}$ in a neighborhood of $x=1$. We manage to do this by imposing the condition
\begin{equation}\label{eq:cond-k,h}
\ell(h)\leq e_k(h)\leq \rho^*<1\,,
\end{equation}
where $\rho^*>0$ is given, $\ell(h)\to0$ and $h^{-1}\ell(h)\to+\infty$ .  This in particular implies that $k\to+\infty$ as $h\to0^+$. Fixing  $a\in(\sqrt{\rho^*},1)$, we can select $A>1$ in such a way that
\[A\sqrt{e_k(h)}<a\,. \]
Consequently, we obtain the following asymptotic formula
    \begin{equation}\label{eq:asy-Phi}
\Phi_{k,h}(x)
    		=
    		\hat C_{k,h}
    		\left(\frac{x^2}{e_k(h)}-1\right)^{-1/4}
    		e^{-S_{k,h}(x)/h}
    		\left(1+\mathcal O(k^{-1})\right)\,,
    \end{equation}
uniformly with respect to $x\in[a,+\infty)\,$. \\
Here
    \begin{equation}\label{eq:def-S-kh}
    S_{k,h}(x)=S(x;e_k(h)),
    \end{equation} and $S(x;E)$ is introduced in \eqref{eq:def-Agmon}. Furthermore, thanks to \eqref{eq:C-kh}, we have
\begin{equation}\label{eq:def-hat-C-kh}
\hat C_{k,h}= \frac{(-1)^{k-1}}{\sqrt{2\pi} (kh)^{1/4}} (1+\mathcal O(k^{-1}))
\sim (-1)^{k-1}\pi^{-\frac12}e_k(h)^{-\frac{1}{4}}\,.
\end{equation}
We prove next that we can differentiate \eqref{eq:asy-Phi}.
\begin{proposition}\label{prop:asy-Phi}
Let $a\in(\sqrt{\rho^*},1)$ be fixed. There exists $C,h_0>0$ such that,   if $(k,h)$ satisfyies \eqref{eq:cond-k,h},  $h<h_0$, and $x\in[a,1]$, then 
\[\Phi'_{k,h}(x)=h^{-1}S'_{k,h}(x)\Phi_{k,h}(x)\bigl(1+R_{k,h}(x)\bigr),\]
with
\[|R_{k,h}(x)|\leq C k^{-1}\,.\]
\end{proposition}
\begin{proof}
For simplicity, we will omit the reference to $(k,h)$ and write $S(x)$ instead of $S_{k,h}(x)$, and $\Phi(x)$ instead of $\Phi_{k,h}(x)$. Let $E=e_k(h)$ and $\Phi_*(x)=\hat C_{k,h}
	\left(\frac{x^2}{E}-1\right)^{-1/4}
	e^{-S(x)/h}.$ Fixing $b>1$, we know that $\Phi(x)= \Phi_*(x)\big(1+\mathcal O(k^{-1})\big)$ for $x\in[a,b]\,$.

Note that $\Phi_*'(x)\sim -h^{-1}S'(x)\Phi_*(x)$ and that $S'(x)=\sqrt{x^2-E}$ is bounded in $[a,b]$. Moreover, $\Phi_*(b)/\Phi_*(x)\to0$ uniformly for $x\in[a,1]$\,.

The function $\Phi_*$ satisfies
\[-h^2\Phi_*''+h^2\frac{\nu^2-1/4}{x^2}\Phi_*+(x^2-E)\Phi_*=R_h \,\Phi_*\,,\]
with $R_h=\mathcal O(h^2)\,,$ on $[a,b].$ Indeed, writing
	\[
	q(x)=\left(\frac{x^2}{E}-1\right)^{-1/4}\,,
	\]
	we have
	\[
	2S'(x)q'(x)+S''(x)q(x)=0\,,
	\]
	and
	\[
	R_h
	=
	h^2\left(
	-\frac{q''(x)}{q(x)}
	+\frac{\nu^2-1/4}{x^2}
	\right)
	=\mathcal O(h^2)\,,
	\]
	uniformly on $[a,b]$\,.

Thus
\[(\Phi'\Phi_*-\Phi\Phi_*)'=-h^{-2}R_h\Phi\Phi_*\,.\]
Consequently
\[\Phi'(x)\Phi_*(x)-\Phi(x)\Phi_*'(x)=h^{-2}\int_{x}^{1}R_h(t)\Phi(t)\Phi_*(t)\,dt\,.\]
Dividing both sides by $\Phi_*'(x)\Phi(x)$ and using that $\Phi_*(t)\sim -h \Phi_*'(t)/S'(t)\,$, we infer from the aforementioned formula that
\[\frac{\Phi'(x)}{\Phi_*'(x)}-1=\Big(\frac{1}{\Phi(x)\Phi_*(x)}\int_{x}^{b} \Phi(t)\Phi_*'(t)dt \Big)\, \mathcal O(h^2)\,. \]
Using that $\Phi(t)\sim\Phi_*(t)\,,$ we obtain
\[\frac{1}{\Phi(x)\Phi_*(x)}\int_{x}^{b} \Phi(t)\Phi_*'(t)dt\sim\frac{1}{\Phi_*(x)^2}\int_{x}^{b} \Phi_*(t)\Phi_*'(t)dt=\frac12-\frac{\Phi_*(b)^2}{\Phi_*(x)^2}=\frac12+o(1)\,.\]
Thus, $\Phi'(x)=\Phi_*'(x)(1+\mathcal O(h^2))\,,$ uniformly for $x\in[a,1]\,$.
\end{proof}

\subsection{The trial state}\label{subsec:trial-state}

We will adjust the function $\Phi_{k,h}$ by adding to it another function, defined by a growing solution of \eqref{eq:def-Phi*},  resulting in a trial state $v_{k,h}$ satisfying the Dirichlet boundary condition at $x=1$ (see \eqref{eq:quasimode-def}). This line of thought goes back to \cite{Bo,BoHe}.

Choose $a\in(\sqrt{\rho^*},1)\,$. Thanks to \eqref{eq:asy-Phi}, 
$\Phi_{k,h}$ has a constant sign on $[a,1]$, for sufficiently small $h$. 
We define a second solution of
\[
(T_h-\mu)u=0
\]
on $[a,1]$ by reduction of order:
\begin{equation}\label{eq:def-Psi}
	\Psi_{k,h}(x)
	=
	\Phi_{k,h}(x)\int_a^x\frac{dt}{\Phi_{k,h}(t)^2}\,.
\end{equation}
Then
\begin{equation}\label{eq:Wronskian-normalization}
	W(\Phi_{k,h},\Psi_{k,h})
	:=
	\Phi_{k,h}\,\Psi_{k,h}'-\Phi_{k,h}'\,\Psi_{k,h}
	=
	1\,.
\end{equation}
Choose (fixed) numbers $\epsilon,a_0$ and $a_1$ such that
\[
0<\epsilon<a<a_0<a_1<1
\]
and
\begin{equation}\label{eq:cond-a-i}
2S_{k,h}(a_1)-S_{k,h}(a_0)-S_{k,h}(1)<-\epsilon\,.
\end{equation}
Then, choose a function $\chi\in C^\infty([0,1])$ such that
\[
\chi=0 \text{ on }[0,a_0]\,,
\qquad
\chi=1 \text{ on }[a_1,1]\,.
\]
Set
\begin{equation}\label{eq:J-def}
	J_{k,h}
	:=
	\int_a^1\frac{1}{\Phi_{k,h}(t)^2}\,dt\,,
\end{equation}
and
\begin{equation}\label{eq:quasimode-def}
	c_{k,h}:=-\frac{1}{J_{k,h}}\,,
	\qquad
	v_{k,h}:=\Phi_{k,h}+c_h\chi\Psi_{k,h}\,.
\end{equation}
Since
\[
\Psi_{k,h}(1)=\Phi_{k,h}(1)J_{k,h}\,,
\]
we have
\[
v_{k,h}(1)=\Phi_{k,h}(1)+c_{k,h}\Psi_{k,h}(1)=0\,.
\]
Moreover, $v_{k,h}$ has the regular behavior at the singular endpoint
$x=0$, since $\chi\Psi_{k,h}$ vanishes in a neighborhood of $0$\,.
Consequently, $v_{k,h}$ belongs to the domain of  $T_h\,$.

\mn 

Recall that we fixed $\rho^*\in(0,1)$ and we introduced the set $\mathcal J(h)$ in \eqref{eq:cond-mu}.
\begin{proposition}\label{prop:boundary-quasimode}
	As $h\to0^+$, and for $k\in\mathcal J(h)$, we have
	\begin{equation}\label{eq:asy-ch*}
c_{k,h}=-\frac{2}{\pi h}e^{-2S_{k,h}(1)/h}\bigl(1+o(1)\bigr)\,,
\end{equation}
and
\begin{equation}\label{eq:eigenvalue-shift-J}
		\lambda_k(h)-e_k(h)
		=-h^2c_{k,h}\bigl(1+o(1)\bigr)\,.
	\end{equation}
\end{proposition}

With Proposition~\ref{prop:boundary-quasimode} in hand, we have:

\begin{proof}[Proof of Theorem~\ref{thm:main}]
Combine \eqref{eq:asy-ch*} and \eqref{eq:eigenvalue-shift-J}, and use that $\mathsf S_{k,h}=S_{k,h}(1)$\,.
\end{proof}

\begin{proof}[Proof of Proposition~\ref{prop:boundary-quasimode}]
	We divide the proof into several steps. For simplicity, we write $E=e_k(h)$, $S(x)$ instead of $S_{k,h}(x)=S(x;E)$, and $\Phi(x),\Psi(x),v(x)$ instead of $\Phi_{k,h}(x)\,,\Psi_{k,h}(x),v_{k,h}(x)\,$. Also, $\|\cdot\|$ and $\langle\cdot,\cdot\rangle$ denote the standard norm and inner product in $L^2(0,1).$
	
	\medskip
	\noindent
	\emph{Step 1.} We derive \eqref{eq:asy-ch*}. 
	Using \eqref{eq:asy-Phi}, we have
	\[
	\frac{1}{\Phi(t)^2}
	=
	\hat C_{k,h}^{-2}
	\left(\frac{t^2}{E}-1\right)^{1/2}
	e^{2S(t)/h}
	\bigl(1+o(1)\bigr)\,.
	\]
	Hence, by the Laplace method at the endpoint $t=1$,
	\[
	J_{k,h}
	=
	\frac{\hat C_{k,h}^{-2}h}{2S'(1)}
	\left(\frac{1}{E}-1\right)^{1/2}
	e^{2S(1)/h}
	\bigl(1+o(1)\bigr)\,.
	\]
	Since
	\[
	S'(1)=\sqrt{1-E},
	\qquad
	\left(\frac{1}{E}-1\right)^{1/2}
	=
	\frac{\sqrt{1-E}}{\sqrt E}\,,
	\]
	we obtain
	\[
	J_{k,h}
	=
	\frac{\hat C_{k,h}^{-2}h}{2\sqrt E}
	e^{2S(1)/h}
	\bigl(1+o(1)\bigr)\,.
	\]
	Finally, using \eqref{eq:def-hat-C-kh},
	\[
	\hat C_{k,h}^{-2}
	=
	\pi E^{1/2}\bigl(1+o(1)\bigr)\,,
	\]
	and therefore
	\[
	J_{k,h}
	=
	\frac{\pi h}{2}
	e^{2S(1)/h}
	\bigl(1+o(1)\bigr)\,.
	\]
In light of \eqref{eq:def-hat-C-kh}, the coefficient $c_{k,h}=-J_{k,h}^{-1}$ satisfies
\begin{equation}\label{eq:asy-ch}
c_{k,h}=
	-\frac{2}{\pi h}
	e^{-2S(1)/h}
	\bigl(1+o(1)\bigr)
\mbox{ as } h\to 0^+\,.
\end{equation}
Since $S(1)=S_{k,h}(1)\,,$ \eqref{eq:asy-ch*} follows from \eqref{eq:asy-ch}. 

This allows us to show that $v$ is almost normalized in $L^2(0,1)$. Since $\Phi$ is normalized in $L^2(\R_+)$, it follows from \eqref{eq:asy-ch} and from Proposition~\ref{prop:dc-Phi}  that
\begin{equation}\label{eq:norm-vh}
\|v\|=1+o(1)\qquad(h\to0^+)\,.
\end{equation}

\medskip
	\noindent
\emph{Step 2.} We calculate the inner product $\langle L_hv,v\rangle$, where
\[
L_h:=T_h-E\,.
\]
Since \(L_h\Phi=L_h\Psi=0\), we have pointwise
\begin{equation}\label{eq:residual-exact}
	L_hv
	=
	c_{k,h}L_h(\chi\Psi)
	=
	c_{k,h}[L_h,\chi]\Psi\,,
\end{equation}
where
\begin{equation}\label{eq:commutator}
	[L_h,\chi]\Psi
	=
	-h^2\bigl(2\chi'\Psi'+\chi''\Psi\bigr)\,.
\end{equation}
	We first compute \(
	\bigl\langle L_h(\chi\Psi),\Phi\bigr\rangle\).
Since $\chi\Psi$ vanishes in a neighborhood of $0$, whereas
$\chi=1$ in a neighborhood of $1$, two integrations by parts and the
identity $L_h\Phi=0$ yield
\[
\begin{aligned}
	\bigl\langle L_h(\chi\Psi),\Phi\bigr\rangle
	&=
	-h^2
	\left[
	(\chi\Psi)'\Phi-(\chi\Psi)\Phi'
	\right]_{0}^{1}\\
	&=
	-h^2\bigl(\Psi'(1)\Phi(1)-\Psi(1)\Phi'(1)\bigr)\\
	&=
	-h^2W(\Phi,\Psi)=
	-h^2\,.
\end{aligned}
\]
Therefore, we infer from \eqref{eq:residual-exact},
\begin{equation}\label{eq:rayleigh-numerator}
	\langle L_hv,v\rangle
	=
	-h^2c_{k,h}
	+
	c_{k,h}^2
	\bigl\langle [L_h,\chi]\Psi,\chi\Psi\bigr\rangle\,.
\end{equation}
It remains to prove that the second term on the right hand side is negligible compared to the first term. Since
\[
\operatorname{supp}\chi'\cup\operatorname{supp}\chi''
\subset [a_0,a_1],
\]
it follows from \eqref{eq:asy-Phi} and Proposition~\ref{prop:asy-Phi},
\[
	\left|
	\bigl\langle [L_h,\chi]\Psi,\chi\Psi\bigr\rangle
	\right|
	=
	\mathcal O\bigl(h^{-1/2}e^{\frac{3S(a_1)-S(a_0)}{h}}\bigr)\,.
\]
Thanks to \eqref{eq:asy-ch} and \eqref{eq:cond-a-i}, it holds
\begin{equation}\label{eq:quadratic-small}
c_{k,h}^2\langle [L_h,\chi]\Psi,\chi\Psi\bigr\rangle=o(h^2c_{k,h})\quad\mbox{ as } h\to0^+\,.
\end{equation}
In a similar fashion, we prove that
\begin{equation}\label{eq:residual-bound}
	\|L_hv\|=\mathcal O(h^{-1}|c_{k,h}|e^{\frac{2S(a_1)-S(a_0)}{h}})=o\big(h^2\sqrt{|c_{k,h}|}\big)\,.
	\end{equation}
	\medskip
	\noindent
	\emph{Step 3.} It follows from \eqref{eq:rayleigh-numerator}, \eqref{eq:quadratic-small} and  \eqref{eq:norm-vh} that the Rayleigh quotient satisfies
    \begin{equation}\label{eq:Rayleigh}
    \frac{\langle L_hv,v\rangle}{\|v\|^2}=-h^2\,c_{k,h}\, (1+o(1)).
    \end{equation}
An argument via the spectral theorem yields that
\[\lambda_k(h)	=	E-h^2c_{k,h}	+o\bigl(h^2|c_{k,h}|\bigr).\]
We provide the details for the convenience of the reader.  We write the spectral resolution of $T_h$ in $L^2(0,1)$ as
\[T_h=\sum_{j=1}^{+\infty}\lambda_j(h)\Pi_j\,,\quad I=\sum_{j=1}^{+\infty}\Pi_j\,,\]
where $\Pi_j$ is the orthogonal projection on the normalized eigenfunction associated to $\lambda_j(h)\,$. It follows then that
\[\begin{gathered}
\langle L_h(v-\Pi_kv),v-\Pi_kv\rangle=\langle L_hv,v-\Pi_kv\rangle\,,\\
    \|L_hv\|^2=\sum_{j=1}^{+\infty}|\lambda_j(h)-E|^2\|\Pi_jv\|^2\geq 4h^2\sum_{j\not=k}\|\Pi_jv\|^2\,,
\end{gathered}\]
where we used the result of Corollary~\ref{corol:gap} in the last step. Thus, we obtain 
	\begin{equation}\label{eq:projection-error}
		\left|
		\left\langle
		L_h(v-\Pi_{k}v),
		v-\Pi_{k}v
		\right\rangle
		\right|\leq \|L_hv\|\|v-\Pi_kv\|,\quad 		\|v-\Pi_{k}v\|
		\le
		\frac{1}{2h}\|L_hv\|\,.
	\end{equation}
We use next the orthogonal decomposition
	\[
	v=\Pi_{k}v+(v-\Pi_{k}v)
	\]
to write
\[\langle L_hv,v\rangle=(\lambda_k(h)-E)\|\Pi_kv\|^2+\langle L(v-\Pi_kv),v-\Pi_kv\rangle-E\|v-\Pi_kv\|^2.\]

Using \eqref{eq:residual-bound} and \eqref{eq:projection-error}, we get
	\[
	\langle L_hv,v\rangle
	=
	(\lambda_k(h)-E)
	\|\Pi_{k}v\|^2
	+
	o(h^2c_{k,h})\,,\quad \|\Pi_kv\|^2=\|v\|^2+o(1).
	\]
Inserting these estimates into \eqref{eq:Rayleigh}, we obtain the formula in \eqref{eq:eigenvalue-shift-J}.
\end{proof}

\begin{remark}[Explicit dependence on the flux term]
		With $\rho=4kh$ and
		\[
		E=e_k(h)=\rho+2(\nu-1)h\,,
		\]
		we have, by a straightforward calculation,
		\[
		S(1;E)
		=
		S_0(\rho)
		-(\nu-1)h
		\log\!\left(
		\frac{1+\sqrt{1-\rho}}{\sqrt{\rho}}
		\right)
		+\mathcal O(h^2)\,,
		\]
		where
		\[
		S_0(\rho)
		=
		\frac{\sqrt{1-\rho}}{2}
		-
		\frac{\rho}{2}
		\log\!\left(
		\frac{1+\sqrt{1-\rho}}{\sqrt{\rho}}
		\right)\,.
		\]
		Thus,
		\[
		e^{-2S(1;E)/h}
		=
		\left(
		\frac{1+\sqrt{1-\rho}}{\sqrt{\rho}}
		\right)^{2\nu-2}
		e^{-2S_0(\rho)/h}
		\bigl(1+\mathcal O(h)\bigr)\,.
		\]
		Consequently, the splitting asymptotics in \eqref{eq:eigenvalue-shift-J} becomes
		\begin{subequations}\label{eq:splitting-asymptotic-rho}
		\begin{equation}
			\lambda_k(h)
				=
				(4k+2\nu-2)h
				+
				C(\rho,\nu)h
				e^{-2S_0(\rho)/h}
				\bigl(1+o(1)\bigr)\,,
		\end{equation}
		where
		\begin{equation}
		C(\rho,\nu)
		=
		\frac{2}{\pi}
		\left(
		\frac{1+\sqrt{1-\rho}}{\sqrt{\rho}}
		\right)^{2\nu-2}\,,
		\qquad
		h\ll\rho=4kh\leq\rho_1<1\,.
		\end{equation}
		\end{subequations}
		In particular, the leading term $S_0(\rho)$ is independent
		of the flux term $\nu\,$. This is expected since the Aharonov--Bohm
		flux appears in the potential as a term of order $h^2$, hence it
		contributes as  an $\mathcal O(h)$ correction,
		yielding the explicit algebraic prefactor in the exponentially small 
		asymptotics.
	\end{remark}

	\begin{remark}[Low energy regime]\label{rem:small-rho-corrected}
		Let us consider more precisely the regime
		\[
		h\ll \rho=4kh\ll 1\,.
		\]
		As $\rho\to0\,$, we have
		\[
		S_0(\rho)
		=
		\frac12
		+\frac{\rho}{4}\log\rho
		-\frac{\rho}{2}\log 2
		-\frac{\rho}{4}
		+\frac{\rho^2}{16}
		+\mathcal O(\rho^3)\,.
		\]
		Consequently, since $\rho=4kh$\,,
		\[
		-\frac{2S_0(\rho)}{h}
		=
		-\frac1h
		-2k\log(kh)
		+2k
		-2k^2h
		+\mathcal O(k^3h^2)\,.
		\]
		In particular, in the regime
		\[
		k\to+\infty\,,
		\qquad
		k^2h\to0\,,
		\]
		we obtain
		\[
		e^{-2S_0(\rho)/h}
		=
		e^{-1/h}e^{2k}(kh)^{-2k}
		\bigl(1+o(1)\bigr)\,.
		\]
		Moreover,
		\[
		\left(
		\frac{1+\sqrt{1-\rho}}{\sqrt{\rho}}
		\right)^{2\nu-2}
		=
		(kh)^{1-\nu}\bigl(1+o(1)\bigr)\,.
		\]
		Hence, the  asymptotics in \eqref{eq:splitting-asymptotic-rho} reads as
		\begin{equation}\label{eq:splitting-asymptotic-rho*}
			\lambda_k(h)-e_k(h)
			=
			\frac{2}{\pi}
			e^{2k}
			k^{1-\nu-2k}
			h^{2-\nu-2k}
			e^{-1/h}
			\bigl(1+o(1)\bigr)\,.
		\end{equation}
		\end{remark}
		
		\begin{remark}[Comparison with Gannot's formula]\label{rem:Gannot}
		It is interesting to compare \eqref{eq:splitting-asymptotic-rho*} with the fixed-$k$
		asymptotics in \eqref{eq:Gannot*}. Indeed, as $k\to+\infty$, Stirling's formula gives
		\[
		(k-1)!\Gamma(k+\nu)
		=
		2\pi k^{2k+\nu-1}e^{-2k}
		\bigl(1+o(1)\bigr)\,,
		\]
		and therefore the explicit coefficient in \eqref{eq:Gannot*} satisfies
		\[
		\frac{4h^{-2k-\nu+2}}
		{(k-1)!\Gamma(k+\nu)}
		e^{-1/h}
		=
		\frac{2}{\pi}
		e^{2k}
		k^{1-\nu-2k}
		h^{2-\nu-2k}
		e^{-1/h}
		\bigl(1+o(1)\bigr)\,.
		\]
Thus, formally, the asymptotics in \eqref{eq:Gannot*} and \eqref{eq:splitting-asymptotic-rho*} match in the overlap regime 
		\[
		k\to+\infty\,,
		\qquad
		k^2h\to0\,.
		\]
		We emphasize, however, that \eqref{eq:Gannot*} is proved for fixed $k\,$,
		so this comparison does not by itself provide a uniform version of
		the fixed-$k$ asymptotics. We address this rigorously in the next subsection.
	\end{remark}


\subsection{Semi-excited states}\label{subsec:vl-regime}

We now consider the energy levels satisfying
\begin{equation}\label{eq:vl-energy}
	0<e_k(h)<\ell(h),
\end{equation}
where
\begin{equation}\label{eq:vl-ell}
	\ell(h)\longrightarrow0\,,
	\qquad
	\frac{\ell(h)}{h}\longrightarrow+\infty\,,
	\qquad
	\frac{\ell(h)^2}{h}\longrightarrow0
	\qquad (h\to0^+)\,.
\end{equation}

\begin{proposition}\label{prop:excited-states}
If \eqref{eq:vl-energy} and \eqref{eq:vl-ell} hold, then
\begin{equation}\label{eq:vl-final}
	\lambda_k(h)-e_k(h)
	=
	\frac{4h^{-2k-\nu+2}}
	{(k-1)!\Gamma(k+\nu)}
	e^{-1/h}
	\bigl(1+o(1)\bigr)\,,
\end{equation}
where $o(1)$ is uniform with respect to $k$.
\end{proposition}
\begin{remark}\label{rem:vl-gannot}~ 
\begin{itemize}
\item For fixed $k$, \eqref{eq:vl-final} is precisely Gannot's asymptotic
formula \eqref{eq:Gannot*}. The argument above shows, moreover, that this formula remains
uniform in the regime
$	k^2h\longrightarrow0\,.$
\item We may eventually choose
\[\ell(h)=h^{q},\quad \frac12<q<1.
\]
\end{itemize}
\end{remark}

The proof of Theorem~\ref{thm:ext-Gannot} is  now straightforward.
\begin{proof}[Proof of Theorem~\ref{thm:ext-Gannot}]
Apply Proposition~\ref{prop:excited-states} with $\ell(h)=4h^{1-\gamma}+(2\nu-2)h$\,.
\end{proof}
The proof of Proposition~\ref{prop:excited-states} does not  use the Plancherel--Rotach asymptotics.
Instead, we exploit directly the exact Laguerre representation of
the normalized eigenfunction on the half-line.

\medskip\noindent
We start with an elementary uniform estimate for Laguerre
polynomials. Later on we shall apply this result with $m=k-1$.
The $C^1$ estimate, and in particular the logarithmic derivative
estimate below, provide the estimates on the second
solution obtained by reduction of order.

\begin{lemma}\label{lem:vl-laguerre}
	Let $a\in(0,1)$ be fixed and let $m=m(h)\geq0$ satisfy
	\[
		m^2h\longrightarrow 0\quad\text{ as $h\to0^+$}\,.
	\]
	Then,  for $x\in[a,1]$\,,
	\begin{equation}\label{eq:vl-laguerre}
		L_m^{(\nu)}\left(\frac{x^2}{h}\right)
		=
		\frac{(-1)^m}{m!}
		\left(\frac{x^2}{h}\right)^m
		\bigl(1+R_{m,h}(x)\bigr)\,,
	\end{equation}
	where $R_{m,h}(x)$ belongs to $C^1([a,1])$ and satisfies
	\begin{equation}\label{eq:vl-R}
		\|R_{m,h}\|_{C^1([a,1])}=\mathcal O(m^2h)\,.
	\end{equation}
	In particular,
	\begin{equation}\label{eq:vl-log-L}
		\frac{d}{dx}
		\log\left|
		L_m^{(\nu)}\left(\frac{x^2}{h}\right)
		\right|
		=
		\frac{2m}{x}+\mathcal O(m^2h)\,,
	\end{equation}
	uniformly on $[a,1]$\,.
\end{lemma}
\begin{proof}
	We use the explicit finite-sum representation of the generalized
	Laguerre polynomials (see \cite{DLMF}, Eq.~(18.5.12)):
	\[
		L_m^{(\nu)}(X)
		=
		\sum_{j=0}^{m}
		(-1)^j
		\binom{m+\nu}{m-j}
		\frac{X^j}{j!}\,.
	\]
Here
\[
\binom{m+\nu}{m-j}
=
\frac{(m+\nu)(m+\nu-1)\cdots(\nu+j+1)}
{(m-j)!}\,.
\]
Setting $r=m-j$, we obtain
	\[
		L_m^{(\nu)}(X)
		=
		\frac{(-1)^mX^m}{m!}
		\left(
		1+
		\sum_{r=1}^m
		(-1)^r
		\frac{m!}{(m-r)!}
		\binom{m+\nu}{r}X^{-r}
		\right)\,.
	\]
	Thus, for $x\in[a,1]$, we define
	\[
		R_{m,h}(x)
		:=
		\sum_{r=1}^m
		(-1)^r
		\frac{m!}{(m-r)!}
		\binom{m+\nu}{r}
		\left(\frac{h}{x^2}\right)^r\,.
	\]
	Since $\nu$ is fixed and  non negative, for $1\leq r\leq m$, we have
\[
	\frac{m!}{(m-r)!}\leq m^r\,,
	\qquad
	\binom{m+\nu}{r}
	\leq
	\frac{(m+\nu)^r}{r!}\,.
\]
Consequently, since $x\in[a,1]$\,,
\[
	|R_{m,h}(x)|
	\leq
	\sum_{r=1}^{m}
	\frac{1}{r!}
	\left(
		\frac{m(m+\nu)h}{a^2}
	\right)^r
	\leq
	\exp\left(
		\frac{m(m+\nu)h}{a^2}
	\right)-1\,.
\]
Since $\nu$ and $a>0$ are fixed and $m^2h\to0$, we obtain
\[
	|R_{m,h}(x)|=\mathcal O(m^2h)\,,
\]
uniformly for $x\in[a,1]$. Differentiating the finite sum defining $R_{m,h}$ gives
\[
	R_{m,h}'(x)
	=
	-\frac{2}{x}
	\sum_{r=1}^m
	r(-1)^r
	\frac{m!}{(m-r)!}
	\binom{m+\nu}{r}
	\left(\frac{h}{x^2}\right)^r\,.
\]
Hence, since $x\in[a,1]$\,,
\[
	|R_{m,h}'(x)|
	\leq
	\frac{2}{a}
	\sum_{r=1}^{m}
	\frac{1}{(r-1)!}
	\left(
		\frac{m(m+\nu)h}{a^2}
	\right)^r\,.
\]
Since
\[
	\sum_{r=1}^{\infty} \frac{t^r}{(r-1)!}
	=
	t e^t\,,
\]
we obtain
\[
	|R_{m,h}'(x)|
	\leq
	\frac{2}{a}
	\frac{m(m+\nu)h}{a^2}
	\exp\left(
		\frac{m(m+\nu)h}{a^2}
	\right)
	=
	\mathcal O(m^2h)\,,
\]
uniformly for $x\in[a,1]$. This proves \eqref{eq:vl-R}. Since $R_{m,h}=o(1)$ uniformly,
	we also have
	\[
		\frac{R_{m,h}'(x)}{1+R_{m,h}(x)}
		=
	\mathcal	O(m^2h)\,,
	\]
	and \eqref{eq:vl-log-L} follows by taking the logarithmic
	derivative in \eqref{eq:vl-laguerre}.
\end{proof}

\medskip\noindent
\begin{proof}[Proof of Proposition~\ref{prop:excited-states}]
We split the proof into two steps, working throughout under the
assumption
$
k^2h\to0\,.
$
We first derive accurate asymptotics for $\Phi_{k,h}$ and its derivative near $x=1\,$. We then use the tunneling argument of the previous subsection to obtain the eigenvalue asymptotics.

\medskip
\noindent\textbf{Step 1. Asymptotics of $\Phi_{k,h}$ and $\Phi_{k,h}'$ near $x=1\,$.}

\medskip
\noindent
We apply Lemma~\ref{lem:vl-laguerre} with $m=k-1$. Recall that
\[	\Phi_{k,h}(x)
	=
	h^{-1/4}C_k
	\left(\frac{x}{\sqrt h}\right)^{\nu+1/2}
	L_{k-1}^{(\nu)}\left(\frac{x^2}{h}\right)
	e^{-x^2/(2h)}\,,\]
where $C_k$ is introduced in \eqref{eq:vl-Ck}.
Setting
\[
	p_k=2k+\nu-\frac32\,,
\]
and
\[
	D_{k,h}
	=
	\frac{(-1)^{k-1}C_k}{(k-1)!}
	h^{-k-\frac{\nu}{2}+\frac{1}{2}}\,,
\]
we get by Lemma~\ref{lem:vl-laguerre},
\begin{equation}\label{eq:vl-Phi}
	\Phi_{k,h}(x)
	=
	D_{k,h}x^{p_k}e^{-x^2/(2h)}
	\bigl(1+R_{k-1,h}(x)\bigr)\,,
\end{equation}
uniformly for $x\in[a,1]$, with
\[
	\|R_{k-1,h}\|_{C^1([a,1])}
	=
	\mathcal O(k^2h)
	=
	o(1)\,.
\]
In particular, for $h$ sufficiently small, $\Phi_{k,h}$ has a constant sign on $[a,1]$. Taking the logarithmic derivative in \eqref{eq:vl-Phi}, we obtain
\begin{equation}
	\frac{\Phi_{k,h}'(x)}{\Phi_{k,h}(x)}
	=
	\frac{p_k}{x}
	-\frac{x}{h}
	+
	\frac{R_{k-1,h}'(x)}{1+R_{k-1,h}(x)}\,.
\end{equation}
Since
\[
	\|R_{k-1,h}\|_{C^1([a,1])}=\mathcal O(k^2h)\,,
\]
it follows that
\begin{equation}\label{eq:vl-log-Phi}
	\frac{\Phi_{k,h}'(x)}{\Phi_{k,h}(x)}
	=
	\frac{p_k}{x}-\frac{x}{h}+\mathcal O(k^2h)=-\frac{x}{h}\Bigl(1+\mathcal O(k^2h)\Bigr)\,,
\end{equation}
uniformly for $x\in[a,1]\,$.

\vspace{0.5cm}
\noindent\textbf{Step 2. Tunneling argument.}

\medskip
\noindent
We now follow the argument of Subsection~\ref{subsec:trial-state}, replacing the
Plancherel-Rotach estimates by \eqref{eq:vl-Phi} and
\eqref{eq:vl-log-Phi}. Define, as in \eqref{eq:def-Psi},
\[	\Psi_{k,h}(x)
	=
	\Phi_{k,h}(x)
	\int_a^x
	\frac{dt}{\Phi_{k,h}(t)^2}\,.\]
Using \eqref{eq:vl-Phi}, we have
\[	\frac{1}{\Phi_{k,h}(t)^2}
	=
	D_{k,h}^{-2}
	t^{-2p_k}
	e^{t^2/h}
	\bigl(1+\mathcal O(k^2h)\bigr)\,,\]
uniformly for $t\in[a,1]$. An endpoint Laplace estimate then gives, uniformly on fixed compact
subintervals of $(a,1]$,
\begin{equation}\label{eq:vl-Psi-asymp}
	\Psi_{k,h}(x)
	=
	\frac{h}
	{2x-2p_kh/x}
	\frac{1}{\Phi_{k,h}(x)}
	\bigl(1+o(1)\bigr)=\frac{h}{2x}{\Phi_{k,h}(x)}
	\bigl(1+o(1)\bigr)\,.
\end{equation}
Moreover, using the Wronskian identity
\[	W(\Phi_{k,h}\,,\,\Psi_{k,h})=1\,,\]
together with \eqref{eq:vl-log-Phi} and
\eqref{eq:vl-Psi-asymp}, we obtain
\begin{equation}\label{eq:vl-Psi-prime}
	\Psi_{k,h}'(x)
	=
	\frac{1}{2\Phi_{k,h}(x)}
	\bigl(1+o(1)\bigr)\,.
\end{equation}
Thus, the estimates required in the reduction-of-order and quasimode
argument of Subsection~\ref{subsec:trial-state} remain valid in the present regime. Define
\[v_{k,h}=\Phi_{k,h}+c_{k,h}\chi\Psi_{k,h}\]
with $\chi=0$ on $[0,a]$ and $\chi=1$ on $[1,+\infty)$. Then $v_{k,h}$ satisfies the Dirichlet boundary condition at $x=1$ when
\[c_{k,h}=-\frac{\Phi_{k,h}(1)}{\Psi_{k,h}(1)}\,.\]
Using \eqref{eq:vl-Psi-asymp}, we obtain
\[c_{k,h}
	=
	-\frac{2}{h}\Phi_{k,h}(1)^2
	\bigl(1+o(1)\bigr)\,.\]
The proof of Proposition~\ref{prop:boundary-quasimode} can then be
repeated with the same argument to obtain
\[	\lambda_k(h)-e_k(h)
	=
	2h\,\Phi_{k,h}(1)^2
	\bigl(1+o(1)\bigr)\,,\]
uniformly in the regime $k^2h\to0$.

\medskip\noindent
Using \eqref{eq:vl-Phi} at $x=1$ and the definition
of $D_{k,h}$, we obtain
\[	\Phi_{k,h}(1)^2
	=
	\frac{2h^{-2k-\nu+1}}
	{(k-1)!\Gamma(k+\nu)}
	e^{-1/h}
	\bigl(1+o(1)\bigr)\,,
\]
which finishes the proof of Proposition~\ref{prop:excited-states}.
\end{proof}
\subsection{Neumann realization}\label{subsec:Neumann}

If we impose Neumann boundary condition at $x=1\,$, we obtain the Neumann realization $T_h^N$ with eigenvalues
\[\lambda_1^N(h)<\lambda_2^N(h)<\cdots.\]
We establish the analogue of Theorem~\ref{thm:main} for the Neumann realization.

\begin{proposition}\label{prop:Neumann}
Let $0<\rho^*<1$ be given. If $(k,h)$ satisfy
\[    \ell(h)\leq e_k(h)\leq \rho^*\,,\quad \ell(h)=|\ln h|^{-1},\]
then
\[\lambda_k^N(h)=e_k(h)- 
	\frac{2h}{\pi}
	\exp\!\left(-\frac{2S_{k,h}(1)}{h}\right)
	\bigl(1+\delta_k(h)\bigr)\,,
\]
where 
\[\sup_{\ell(h)\leq e_k(h)\leq \rho^*}|\delta_k(h)|\longrightarrow0\quad\text{as $h\longrightarrow0^+$}.\]
\end{proposition}
\begin{proof}

Let $\bar\rho\in(\rho*,1)$. We start with rough bounds for energy levels  with label $j$ satisfying $e_j(h)\leq \bar\rho$. Firstly, by the min-max principle and Proposition~\ref{prop:Vanlaere}, we have
\[\lambda^N_j(h)\leq \lambda_j^D(h)\leq e_j(h)+\mathcal O(h^\infty)\,.\]
Secondly, by Agmon estimates, the normalized eigenfunction corresponding to
$\lambda_j^N(h)$ obeys an exponential decay law in $[\bar\rho,1]$. Thus, by the min-max principle,
\[\lambda_j^N(h)\geq \lambda_k^D(h)+\mathcal O(h^\infty)\geq e_j(h)+\mathcal O(h^\infty)\,.\]
Therefore, there exists $h_0>0$ such that, for $h\in(0,h_0]$ and for every integer $j\not=k$, we have
\[|\lambda_j^N(h)-E|\geq 2h\,.\]
Now we consider a trial state
\[u=\Phi+d_{k,h}\chi\Psi\]
with $\Phi=\Phi_{k,h}$ and $\Psi=\Psi_{k,h}$ as in \eqref{eq:def-Phi} and \eqref{eq:def-Psi} respectively. We require that $u$ satisfies Neumann boundary condition, $u'(1)=0,$ which yields
\[d_{k,h}=-\frac{\Phi'(1)}{\Psi'(1)}\,.\]
By the Plancherel-Rotach asymptotics for $\Phi$ and $\Phi'$ in \eqref{eq:asy-Phi} and Proposition~\ref{prop:asy-Phi}, we obtain that
\[d_{k,h}=\frac{1}{J_{k,h}}\left(1+o(1)\right)=-c_{k,h}(1+o(1))\,.\]
A slight adjustment of the proof of Proposition~\ref{prop:boundary-quasimode} yields
\[\lambda_k^N(h)-e_k(h)=h^2c_{k,h}\, (1+o(1))\,.\]
To finish the proof, we use the asymptotics for $c_{k,h}$ in Proposition~\ref{prop:boundary-quasimode}.
\end{proof}
\section{Above the tunneling regime}\label{s3}

\subsection{Preliminaries}
In this section, we consider eigenvalues
\[E> 1\,,\]
which are not covered by Theorem~\ref{thm:main}. Neglecting the term of order $h^2$ in the potential, and identifying  the Dirichlet boundary condition with a hard wall, our operator $T_h$ corresponds to
\[T_h^{\rm eff}=- h^2 \frac{d^2}{dx^2}+V\,,\]
with
\[V(x)=\begin{cases}
    x^2&0\leq x\leq 1\,,\\
    +\infty&x>1\,.
\end{cases}\]
Since $\{V(x)\leq E\}=[0,1],$ the Bohr-Sommerfeld quantization condition reads \cite{Olver, HR}
\begin{equation}\label{eq:Bohr-Sommerfeld}
\int_0^1\sqrt{E-x^2}\,dx\sim k\pi h\,,
\end{equation}
that is, the energy level $E>1$ satisfies
\begin{equation}\label{eq:Bohr-Sommerfeld2}
\frac{E}{2} \arcsin\left(\frac{1}{\sqrt{E}}\right) + \frac{1}{2}\sqrt{E - 1} \sim k\pi h\,,
\end{equation}
where the term on the left hand side is bounded below by $\pi/4$, thereby enforcing the following condition on the label $k$
\begin{equation}\label{eq:BS-label}
\rho:=4kh\geq 1+\mathcal O(h)\,.
\end{equation}
We will analyze this regime directly by means of properties of the
	Whittaker function, without passing through the hard wall potential.
	Unlike the tunneling regime covered in Theorem~\ref{thm:main}, we will not use
	a semiclassical argument.
	Instead, our analysis relies entirely on the uniform asymptotic
	expansions of Whittaker functions established by Dunster
	\cite{Dunster} together with the corresponding asymptotic
	description of their positive zeros obtained by Gabutti and
	Gatteschi \cite{GG2001}. The final inversion step is combined with the
	classical McMahon expansion for the positive zeros of Bessel functions
	(see, e.g., \cite[\S10.21]{DLMF}), which we simply invert in the present
	spectral setting.

	\mn
	Throughout this section, we assume that
	\begin{equation}\label{eq:sc}
		\rho_0\leq \rho=4kh\leq \rho_1\,,
		\qquad
		1<\rho_0<\rho_1<+\infty\,.
\end{equation}

\subsection{Supercritical regime}\label{subsec:sc}
\subsubsection{Preliminaries and main statements}
 We assume that \eqref{eq:sc} holds and derive an accurate asymptotic
 expansion of the $k$-th eigenvalue of $T_h\,$.

\begin{theorem}\label{thm-oscillatory-regime} 
Assume that \eqref{eq:sc} holds, and let 
\[\mu=\frac{\nu}{2}\,,\quad \rho=4kh\,.\]
Then, as $h\to0^+\,$,  
\[
\lambda_k(h)
=
E_0(\rho)
+
h E_1(\rho)
+
\mathcal O(h^2)\,,
\]
where  
\begin{itemize}
\item \(E_0(\rho)>1\) is the solution of
\[
\action(E_0(\rho))=\frac{\pi\rho}{4}\,,
\]
\item For $E>1$\,,
\[\action(E)
=
\frac12\sqrt{E-1}
+
\frac E2\arcsin(E^{-1/2}),\]
\item $E_1(\rho)$ is defined by 

	\[
	E_1(\rho)=
	\frac{2\pi\left(\mu-\frac14\right)}
	{\arcsin\left(E_0(\rho)^{-1/2}\right)}\,.
	\]
\end{itemize}
\end{theorem}

\begin{remark}\label{rem:BSL}
The function $I$ admits the action representation
\begin{equation}\label{eq:action-I}
I(E)=\int_0^1\sqrt{E-y^2}\,dy\,,
\qquad E>1,
\end{equation}
and is the classical action of the harmonic oscillator on the
interval $(0,1)\,$. Thus, the leading-order equation
\begin{equation}\label{eq:BS-principal}
I(E_0(\rho))=\frac{\pi\rho}{4}
\end{equation}
can be viewed as the principal Bohr--Sommerfeld quantization
condition. The subprincipal term in Theorem~\ref{thm-oscillatory-regime}
can moreover be absorbed into 
a shift of the quantization parameter. Indeed,
\[
I'(E)=\frac12\arcsin(E^{-1/2})\,,
\]
and differentiating \eqref{eq:BS-principal} gives
\[
E_0'(\rho)
=
\frac{\pi}
{2\arcsin(E_0(\rho)^{-1/2})}\,.
\]
Since $\mu=\nu/2\,$, we therefore have
\begin{equation}\label{eq:E1-E0prime}
E_1(\rho)
=
4\left(\mu-\frac14\right)E_0'(\rho).
\end{equation}
Consequently, Taylor's formula yields
\[
E_0(\rho)+hE_1(\rho)
=
E_0\left(\rho+4h\left(\mu-\frac14\right)\right)
+\mathcal O(h^2)\,,
\]
uniformly for $\rho\in[\rho_0,\rho_1]\,$. Since $\rho=4kh$, the asymptotic formula of Theorem~\ref{thm-oscillatory-regime}
can equivalently be written as
\begin{equation}\label{eq:BSL-quantization}
\lambda_k(h)
=
E_0\left(
4h\left(k+\frac{\nu}{2}-\frac14\right)
\right)
+\mathcal O(h^2)\,,
\end{equation}
which is the form stated in Theorem~\ref{thm:supercritical-intro}.
The shift $\nu/2-1/4$ is the Bohr--Sommerfeld--Langer correction
associated with the singular inverse square term,  and the dependence on the magnetic field appears  (modulo $\mathcal O (h^2)$) through the shift in the quantization rule.
\end{remark}

\subsubsection{Heuristics}

We fix the notation as
\begin{equation}\label{eq:sc-notation}
\mu=\frac{\nu}{2},\quad \rho=4kh\,,\quad E(\rho):=\lambda_k(h).
\end{equation}
The inverse-square potential in $T_h$ is a lower-order term of order
$h^2$, so it does not contribute to the principal symbol of the
operator.
Consequently, the
underlying model operator is simply the truncated harmonic oscillator
\begin{equation}
-h^2\frac{d^2}{dx^2}+x^2
\end{equation}
on the interval \((0,1)\) with a Dirichlet condition at \(x=1\). Its  associated classical action is
\begin{equation}\label{eq:action}
\action(E)=\int_0^1\sqrt{(E-x^2)_+}\,dx\,.
\end{equation}
Forgetting a possible difficulty related to the singularity at the origin, the
singular inverse-square potential only contributes to the
subprincipal correction and is therefore responsible for the second
term in the asymptotic expansion of Theorem~\ref{thm-oscillatory-regime}, that is
\[E(\rho)=E_0(\rho)+\text{subleading terms}.\]

\mn
The condition in \eqref{eq:sc} ensures that the whole interval \((0,1)\) is classically allowed for this
principal Hamiltonian. 
Formally, the leading term \(E_0(\rho)\)  is characterized by the formal
Bohr--Sommerfeld quantization rule \eqref{eq:Bohr-Sommerfeld}, which reads as
\[
\action(E_0(\rho))
=
\frac{\pi\rho}{4}\,.\]

\mn Recall that the eigenvalues of $T_h$ are given by
\begin{equation}\label{lambda-kappa}
\lambda_k(h)=4h \kappa_k\,,
\end{equation}
where $\kappa_k=\kappa_k(h)$ is determined by the spectral equation
\begin{equation}\label{spectral-kappa}
M_{\kappa_k,\mu}(\xi)=0\,,\quad \xi=1/h\gg 1\,.
\end{equation}

\subsubsection{Gabutti--Gatteschi's asymptotic description of the
positive Whittaker zeros}

The starting point of our analysis is the asymptotic description of the
positive zeros of the Whittaker function obtained by
Gabutti--Gatteschi \cite[Theorems~2.1 and~2.3]{GG2001}, which is itself
based on the uniform asymptotic expansions of Dunster
\cite{Dunster}.

\mn
Throughout this subsection, $\kappa$ denotes the large asymptotic
parameter, whereas $\mu$ is fixed. We introduce the scaled variable

\begin{equation}\label{scaled-variable}
X=\frac{\xi}{\kappa}\,,
\end{equation}
together with
\begin{equation}\label{gamma-def}
\gamma=\frac{2\mu}{\kappa}\,,
\end{equation}
and define
\begin{equation}\label{turning-points}
\xi_1=2-\sqrt{4-\gamma^2}\,,
\qquad
\xi_2=2+\sqrt{4-\gamma^2}\,,
\end{equation}
as well as
\begin{equation}\label{R-def}
R=\sqrt{(X-\xi_1)(\xi_2-X)}\,.
\end{equation}
Furthermore, let $\zeta=\zeta(X)$ be defined implicitly by
\begin{equation}\label{GG-zeta-equation}
\begin{aligned}
\sqrt{\zeta-\gamma^2}
-\gamma
\arctan\left(
\frac{\sqrt{\zeta-\gamma^2}}{\gamma}
\right)
={}&
\frac12R
-\frac{\gamma}{2}
\arctan\left(
\frac{2X-\gamma^2}{\gamma R}
\right)
\\
&
-\arctan\left(
\frac{2-X}{R}
\right)
+
\frac{\pi}{2}
\left(
1-\frac{\gamma}{2}
\right)\,.
\end{aligned}
\end{equation}

\mn
Finally, for every positive integer \(k\,\), define

\begin{equation}\label{zeta-k}
\zeta_k=
\left(
\frac{j_{2\mu,k}}{\kappa}
\right)^2\,,
\end{equation}
where \(j_{\nu,k}\) denotes the \(k\)-th positive zero of the Bessel
function \(J_\nu\)\,.

\mn
Theorem~2.1 of \cite{GG2001} shows that, if \(X_k\) is obtained by
inverting \eqref{GG-zeta-equation} with
\(\zeta=\zeta_k\,\), then the \(k\)-th positive zero
\(m_k\) of \(M_{\kappa,\mu}\) satisfies

\begin{equation}\label{GG-leading}
m_k=\kappa X_k+\mathcal O(\kappa^{-1})\,.
\end{equation}
More precisely, for every fixed \(q\in(\xi_1,\xi_2)\)\,, the estimate
\eqref{GG-leading} holds uniformly for all zeros satisfying
\(m_k\le q\kappa\).

\mn
Define
\begin{equation}\label{C0-def}
C_0(\zeta,\gamma,X)
=
\frac14\frac1{\sqrt{\zeta-\gamma^2}}
+
\frac5{12}
\frac{\gamma^2}{(\zeta-\gamma^2)^{3/2}}
+
\frac16
\frac{2X^3+3(\gamma^2-8)X^2+24X-2\gamma^2(\gamma^2-2)}
{R^3(\gamma^2-4)}\,,
\end{equation}
with \(R\) given by \eqref{R-def}.\\
Then Theorem~2.3 of \cite{GG2001} yields the refined asymptotic expansion

\begin{equation}\label{GG-zero}
m_k
=
\kappa X_k
-
\frac{X_k}{\kappa R_k}
\,C_0(\zeta_k,\gamma,X_k)
+
\mathcal O(\kappa^{-3})\,,
\end{equation}
where \(R_k\) is defined by \eqref{R-def} with \(X\) replaced by \(X_k\).  More precisely, for every fixed \(q\in(\xi_1,\xi_2)\,\), the expansion
\eqref{GG-zero} holds uniformly for all positive zeros satisfying
\(m_k\le q\kappa\)\,.

\begin{remark}
The implicit relation \eqref{GG-zeta-equation} admits a natural
semiclassical interpretation. It can be written as the equality
\begin{equation}
S_W(X,\gamma)=S_B(\zeta,\gamma)\,,
\end{equation}
where \(S_W\) and \(S_B\) denote respectively the classical actions
associated with the Whittaker equation and with its Bessel model. This identity is the key ingredient in Dunster's Liouville
transformation and may be viewed as a Bohr--Sommerfeld--Langer
quantization condition. For the convenience of the reader, its
derivation and geometric interpretation are recalled in
Appendix \ref{AppA}.
\end{remark}

\subsubsection{Inversion strategy}
 The Gabutti--Gatteschi asymptotics describe, for large \(\kappa\,\), the
\(k\)-th positive zero \(m_k\) (\(k=1,2,\ldots\)) of the Whittaker function
\(M_{\kappa,\mu}\) with respect to its argument.

\mn
In light of \eqref{lambda-kappa} and \eqref{spectral-kappa}, we would like to analyze, as $\xi\to+\infty\,$, the spectral condition
\[
M_{\kappa_k,\mu}(\xi)=0\,.
\]
Thus, for \(\kappa=\kappa_k\)\,, the parameter \(\xi\) is a positive zero
of the Whittaker function \(M_{\kappa,\mu}\). We prove below that this zero coincides with the $k$-th positive zero in the
Gabutti--Gatteschi ordering.

\begin{lemma}\label{lemma-index}
For every \(k\ge 1\), the parameter \(\xi\) is the
\(k\)-th positive zero \(m_k\) of the Whittaker function
\(M_{\kappa_k,\mu}\)\,.
\end{lemma}
\begin{proof}
The regular solution of the singular harmonic oscillator is, up to a
non-zero multiplicative constant,
\begin{equation}\label{regular-solution-bis}
u_k(x)
=
x^{\nu+\frac12}e^{-\xi x^2/2}
M_{\kappa_k,\mu}(\xi x^2)\,.
\end{equation}
By \eqref{regular-solution}, the change of variable
$
z=\xi x^2
$
is strictly increasing from \((0,1)\) onto \((0,\xi)\), the zeros of
\(u_k\) in \((0,1)\) are in one-to-one correspondence with the positive
zeros of \(M_{\kappa_k,\mu}\) in \((0,\xi)\,\). By the Sturm oscillation
theorem, the \(k\)-th eigenfunction has exactly \(k-1\) zeros in
\((0,1)\,\). Since \(u_k(1)=0\,\), it follows that \(\xi\) is the
\(k\)-th positive zero of \(M_{\kappa_k,\mu}\)\,.
\end{proof}

\mn 
Consequently, the asymptotic expansion of the eigenvalues reduces to
determining the corresponding value of \(\kappa_k\). Since the spectral
condition is equivalent to
\begin{equation}
m_k(\kappa_k,\mu)=\xi\,,
\end{equation}
we invert the two-term asymptotic expansion of the Whittaker zero
\(m_k(\kappa,\mu)\) obtained by Gabutti and Gatteschi. In this expansion,
the leading term is determined by the implicit action relation
\eqref{GG-zeta-equation}, with
\begin{equation}\label{zetak}
\zeta_k=\left(\frac{j_{2\mu,k}}{\kappa}\right)^2\,.
\end{equation}
Substituting the McMahon expansion of the Bessel zero \(j_{2\mu,k}\) and
then inverting the resulting relation yields the desired expansion of
\(\kappa_k\), and hence of \(\lambda_k(h)\,\), in light of \eqref{lambda-kappa} and \eqref{spectral-kappa}. This formulation keeps the
classical action \(\action(E)\) visible while making clear where the
Gabutti--Gatteschi correction enters.

\subsubsection{Asymptotic inversion of the Whittaker zero expansion}\label{sss3.1.4}

We now derive the asymptotic expansion of the eigenvalues by inverting
the two-term asymptotic expansion of the Whittaker zeros established by
Gabutti and Gatteschi. Recall that

\begin{equation}\label{spectral-kappabis}
\kappa_k(h)
=
\frac{\lambda_k(h)}{4h}\,.
\end{equation}
By Lemma~\ref{lemma-index}, and writing $\kappa_k=\kappa_k(h)$ for short, we have
\begin{equation}\label{spectral-zero}
m_k(\kappa_k,\mu)=\xi\,,
\end{equation}
where \(m_k(\kappa,\mu)\) denotes the \(k\)-th positive zero of the
Whittaker function with respect to its argument. Since 
\[\lambda_k(h)\geq 4kh+2(\nu-1)h\geq \rho_0+\mathcal O(h)>1\,.\]
We can then choose a fixed $q\in(0,4)$ such that, for $h$ sufficiently
small,
\[
\xi_1<\frac{\xi}{\kappa_k}<q<\xi_2\,.
\]
We can then select a fixed $q\in(0,4)$ such that $\xi=1/h<q\kappa$, and therefore apply the refined asymptotic expansion \eqref{GG-zero} with
\(\kappa=\kappa_k\); we obtain
\begin{equation}\label{spectral-mk}
\xi
=
\kappa_kX_k
-
\frac{X_k}{\kappa_kR_k}
\,C_0(\zeta_k,\gamma_k,X_k)
+
\mathcal O(\kappa_k^{-3})\,,
\end{equation}
where $\zeta_k$ is given by (\ref{zetak}), $ \gamma_k
=
\frac{2\mu}{\kappa_k}$ and  \(X_k\) is determined by the implicit relation
\eqref{GG-zeta-equation}. 

\mn 
We now introduce the scaled energy
\begin{equation}\label{scaled-energy}
E=\frac{4}{X_k}.
\end{equation}
Since
\begin{equation}\label{kappa-energy}
 X_k=\frac{\xi}{\kappa_k}\,,
\qquad
\kappa_k=\frac{\xi E}{4}\,,
\end{equation}
it follows from \eqref{spectral-kappabis} that

\begin{equation}\label{spectral-E}
\lambda_k(h) =E\,.
\end{equation}

Consequently, the asymptotic expansion of the eigenvalues reduces to
determining the corresponding asymptotic expansion of the scaled energy
\(E\)\,.

\mn
The next step consists in expanding the implicit relation
\eqref{GG-zeta-equation} in the regime
$
\xi\to+\infty\,,
$
with
\[
\rho=\frac{4k}{\xi}\in[\rho_0,\rho_1]\,.
\]
To this end, we first derive an asymptotic expansion of the quantity
\(\zeta_k\) appearing in \eqref{zetak} by using the classical
McMahon expansion of the Bessel zeros, see \cite[\S10.21, Eq.~(10.21.19)]{DLMF}:
\begin{equation}
j_{2\mu,k}
=
\beta_k
-
\frac{16\mu^2-1}{8\beta_k}
+
\mathcal O(\beta_k^{-3})\,,
\qquad
\beta_k=\pi\left(k+\mu-\frac14\right)\,.
\end{equation}
Using \eqref{zetak}, we obtain
\begin{equation}\label{zeta-McMahon}
\zeta_k
=
\left(
\frac{\beta_k}{\kappa_k}
\right)^2
-
\frac{16\mu^2-1}{4\kappa_k^2}
+
\mathcal O(\beta_k^{-2}\kappa_k^{-2})\,.
\end{equation}
Moreover,
\begin{equation}\label{beta-over-kappa}
\frac{\beta_k}{\kappa_k}
=
\frac{\pi\rho}{E}
+
\frac{4\pi}{\xi E}
\left(\mu-\frac14\right)\,.
\end{equation}
Substituting \eqref{beta-over-kappa} into
\eqref{zeta-McMahon}, we get
\begin{equation}\label{zeta-expansion-E}
\zeta_k
=
\left(
\frac{\pi\rho}{E}
\right)^2
+
\frac{8\pi^2\rho}{\xi E^2}
\left(\mu-\frac14\right)
+
\mathcal O(\xi^{-2})\,.
\end{equation}
Using the identity
$
X=\frac4E\,,
$
the implicit relation \eqref{GG-zeta-equation} is now expressed in terms
of the scaled energy \(E\). Moreover, since
\[
\kappa_k=\frac{\xi E}{4}\,,
\]
we have
\begin{equation}\label{gamma-energy}
\gamma_k
=
\frac{2\mu}{\kappa_k}
=
\frac{8\mu}{\xi E}\,.
\end{equation}
We shall now expand both sides of \eqref{GG-zeta-equation} as
\(\xi\to+\infty\), uniformly for 
\[
\rho=\frac{4k}{\xi}\in[\rho_0,\rho_1]\,.
\]
In what follows, the expansions are first performed for \(E\) in an
arbitrary compact subset of \((1,+\infty)\). The leading equation obtained
below will show that the relevant solutions \(E\) indeed remain in such a compact subset.

\mn 
The right-hand side of
\eqref{GG-zeta-equation} becomes
\begin{equation}\label{SW-energy-expansion}
\begin{aligned}
S_W\left(\frac4E,\gamma_k\right)
&=
\frac12\sqrt{\frac4E\left(4-\frac4E\right)}
+
2\arcsin\left(\frac1{\sqrt E}\right)
-
\frac{4\pi\mu}{\xi E}
+
\mathcal O(\xi^{-2})
\\
&=
\frac{2}{E}\sqrt{E-1}
+
2\arcsin(E^{-1/2})
-
\frac{4\pi\mu}{\xi E}
+
\mathcal O(\xi^{-2})\,.
\end{aligned}
\end{equation}
We next expand the Bessel action. Recall that
\[
S_B(\zeta,\gamma)
=
\sqrt{\zeta-\gamma^2}
-
\gamma\arctan
\left(
\frac{\sqrt{\zeta-\gamma^2}}{\gamma}
\right)\,.
\]
Using \eqref{zeta-expansion-E}, we first get
\begin{equation}\label{sqrt-zeta-expansion-E}
\sqrt{\zeta_k}
=
\frac{\pi\rho}{E}
+
\frac{4\pi}{\xi E}
\left(\mu-\frac14\right)
+
\mathcal O(\xi^{-2})\,.
\end{equation}
Moreover, 
we have
\[
\sqrt{\zeta_k-\gamma_k^2}
=
\sqrt{\zeta_k}
+
\mathcal O(\xi^{-2})\,,
\]
and
\[
\arctan
\left(
\frac{\sqrt{\zeta_k-\gamma_k^2}}{\gamma_k}
\right)
=
\frac{\pi}{2}
+
\mathcal O(\xi^{-1})\,.
\]
Therefore,
\begin{equation}\label{SB-energy-expansion}
\begin{aligned}
S_B(\zeta_k,\gamma_k)
&=
\sqrt{\zeta_k}
-
\frac{\pi}{2}\gamma_k
+
\mathcal O(\xi^{-2})
\\
&=
\frac{\pi\rho}{E}
+
\frac{4\pi}{\xi E}
\left(\mu-\frac14\right)
-
\frac{4\pi\mu}{\xi E}
+
\mathcal O(\xi^{-2})
\\
&=
\frac{\pi\rho}{E}
-
\frac{\pi}{\xi E}
+
\mathcal O(\xi^{-2}).
\end{aligned}
\end{equation}
Identifying \eqref{SW-energy-expansion} and
\eqref{SB-energy-expansion}, we obtain
\[
\frac{2}{E}\sqrt{E-1}
+
2\arcsin(E^{-1/2})
-
\frac{4\pi\mu}{\xi E}
=
\frac{\pi\rho}{E}
-
\frac{\pi}{\xi E}
+
\mathcal O(\xi^{-2})\,.
\]
Multiplying by \(E/2\), we get
\begin{equation}\label{pre-Phi}
\sqrt{E-1}
+
E\arcsin(E^{-1/2})
=
\frac{\pi\rho}{2}
+
\frac{\pi}{\xi}
\left(
2\mu-\frac12
\right)
+
\mathcal O(\xi^{-2})\,.
\end{equation}
It is therefore natural to introduce the function
\begin{equation}\label{Phi-def}
\action(E)
=
\frac12\sqrt{E-1}
+
\frac{E}{2}\arcsin(E^{-1/2})\,.
\end{equation}
Equation \eqref{pre-Phi} then becomes
\begin{equation}\label{Phi-expansion}
\action(E)
=
\frac{\pi\rho}{4}
+
\frac{\pi}{\xi}
\left(
\mu-\frac14
\right)
+
\mathcal O(\xi^{-2}).
\end{equation}
We now invert the relation \eqref{Phi-expansion}. One has
\begin{equation}\label{Phi-prime}
\action'(E)
=
\frac12\arcsin(E^{-1/2})
>0,
\qquad E>1\,.
\end{equation}
Thus, \(I\) is a smooth increasing diffeomorphism from
\((1,+\infty)\) onto \( (\frac{\pi}{4},+\infty) \). Hence, for every
\[
\rho\in[\rho_0,\rho_1],
\qquad 1<\rho_0\le \rho_1<+\infty\,,
\]
there exists a unique \(E_0(\rho)>1\) such that
\begin{equation}\label{E0-def}
\action(E_0(\rho))
=
\frac{\pi\rho}{4}\,.
\end{equation}
Since \(E_0(\rho)\) ranges over a compact subset of \((1,+\infty)\) as
\(\rho\) varies in \([\rho_0,\rho_1]\)\,, and since
\[
\action'(E_0(\rho))\neq 0\,,
\]
the implicit function theorem applies uniformly with respect to
\(\rho\in[\rho_0,\rho_1]\)\,. It follows from \eqref{Phi-expansion} that
\begin{equation}\label{E-expansion}
E
=
E_0(\rho)
+
\frac{E_1^{\sharp}(\rho)}{\xi}
+
\mathcal O(\xi^{-2})\,,
\end{equation}
where
\begin{equation}\label{E1-sharp-def}
E_1^{\sharp}(\rho)
=
\frac{\pi\left(\mu-\frac14\right)}
{\action'(E_0(\rho))}
=
\frac{2\pi\left(\mu-\frac14\right)}
{\arcsin(E_0(\rho)^{-1/2})}\,.
\end{equation}
We now return to the correction term in the Gabutti--Gatteschi expansion
\eqref{spectral-mk}. From \eqref{E-expansion}, we know that \(E\) remains
in a fixed compact subset of \((1,+\infty)\), uniformly for
\(\rho\in[\rho_0,\rho_1]\,\). Hence
\[
X_k=\frac4E
\]
remains in a fixed compact subset of \((0,4)\). Hence Theorems~2.1 and~2.3 of \cite{GG2001} apply uniformly with respect to
\(\rho\in[\rho_0,\rho_1]\)\,, and
\[
\kappa_k=\frac{\xi E}{4}=\mathcal O(\xi)\,,
\qquad
\kappa_k^{-1}=\mathcal O(\xi^{-1})\,.
\]
Since $E$ remains in a compact subset of $(1,+\infty)$, the quantities
$X_k$ and $\zeta_k$ remain uniformly bounded away from the relevant
singular endpoints, while $\gamma_k=\mathcal O(\xi^{-1})$. In particular,
\[
R_k=R(X_k)\geq c>0\,,
\]
for some constant $c>0\,$, and
\[
C_0(\zeta_k,\gamma_k,X_k)=\mathcal O(1)\,.
\]
Consequently,
\begin{equation}\label{GG-correction-size}
\frac{X_k}{\kappa_kR_k}
\,C_0(\zeta_k,\gamma_k,X_k)
=
\mathcal O(\xi^{-1})\,.
\end{equation}
We now explain why the correction term in \eqref{spectral-mk} does not
modify the coefficients already obtained in \eqref{E-expansion}. By
\eqref{GG-correction-size}, equation \eqref{spectral-mk} becomes
\begin{equation}\label{kappaX-corrected}
\kappa_kX_k
=
\xi
+
\mathcal O(\xi^{-1}).
\end{equation}
Since \(\kappa_k\asymp\xi\,\), a division by \(\kappa_k\) yields
\begin{equation}\label{X-corrected}
X_k
=
\frac{\xi}{\kappa_k}
+
\mathcal O(\xi^{-2})\,.
\end{equation}
Since \(X_k\) remains in a compact subset of \((0,4)\), the map
\[
X\longmapsto \frac4X
\]
is smooth in a neighborhood of \(X_k\). Hence the estimate
\eqref{X-corrected} implies that the corresponding correction to the
scaled energy is also of order \(\mathcal O(\xi^{-2})\). Consequently, the correction term in
\eqref{spectral-mk} does not affect the coefficients
\(E_0(\rho)\) and \(E_1^\sharp(\rho)\) in
\eqref{E-expansion}.
\mn
The previous analysis immediately yields the  result of Theorem~\ref{thm-oscillatory-regime}.

\section{The transition regime}\label{sec:transition}

The Bohr--Sommerfeld--Langer quantization derived in the previous
section is valid uniformly in the regime
\[
1<\rho_0\leq \rho:=4kh\leq\rho_1\,.
\]
As $\rho\downarrow1$, the classical turning point reaches the boundary
of the interval. Consequently, the standard semiclassical expansion
ceases to be uniform and must be replaced by a uniform Airy
approximation. 
	We first recall the relevant Airy asymptotics and derive
	a localization result for the eigenvalues in an $h^\delta$-neighborhood
	of the critical threshold, with $0<\delta<2/3$. We then specialize to
	the critical scale $\delta=2/3$ and obtain an explicit transition asymptotic expansion in the window
	$|4kh-1|=\mathcal O(h^{2/3})$.

\subsection{Uniform Airy asymptotics}

Throughout this section we use the notation
\begin{equation}\label{eq:kappa-mu}
\kappa=\frac b2-a\,,
\qquad
\mu=\frac{b-1}{2}\,,\quad \xi=1/h\,,
\end{equation}
so that
\begin{equation}\label{Whittaker}
M(a,b,\xi)
=
e^{\xi/2}\xi^{-b/2}
M_{\kappa,\mu}(\xi)\,,
\end{equation}
see \cite[Eq.~13.14.4]{DLMF}. We assume throughout that
\[
\kappa\to+\infty\,,
\]
while \(\mu\) remains fixed. Under these assumptions, we are
precisely in the framework of the uniform Airy asymptotic theory of
Dunster, reproduced in \cite[\S13.21(iii)]{DLMF}.

\mn
In the transition regime, the Airy variable changes sign at the outer
turning point
\[
\xi=2\kappa+2\sqrt{\kappa^2-\mu^2}\,.
\]
For definiteness, we write below the expression corresponding to
\[
\xi>2\kappa+2\sqrt{\kappa^2-\mu^2}\,,
\]
for which $\widehat\zeta$, introduced in \eqref{zetahat} below, satisfies $\widehat\zeta>0$. The case on the other side of the turning
point is treated similarly, using the corresponding expression for
$\widehat\zeta$ in \cite[Eq.~(13.21.20)]{DLMF}.

We therefore introduce the Airy variable
\begin{equation}\label{zetahat}
\widehat{\zeta}
=
\left(
\frac{3}{2\kappa}
\left(
\frac{X}{2}
+
\mu
\log\!\left(
\frac{\xi\sqrt{\kappa^2-\mu^2}}
{\kappa\xi-2\mu^2-\mu X}
\right)
+
\kappa
\log\!\left(
\frac{2\sqrt{\kappa^2-\mu^2}}
{\xi-2\kappa+X}
\right)
\right)
\right)^{2/3}\,,
\end{equation}
where
\[
X=\sqrt{\xi^2-4\kappa\xi+4\mu^2}\,.
\]
Furthermore, define
\begin{align*}
\widehat c(\kappa,\mu)
&=
\sqrt{2\pi}\,
\kappa^{1/6}
\left(
\frac{\kappa-\mu}{\kappa+\mu}
\right)^{\mu/2}
\left(
\frac{e}{\sqrt{\kappa^2-\mu^2}}
\right)^\kappa\,,
\\
\widehat\Psi(\kappa,\mu,\xi)
&=
\left(
\frac{\widehat\zeta}
{\xi^2-4\kappa\xi+4\mu^2}
\right)^{1/4}
(2\xi)^{1/2}\,.
\end{align*}
Then the uniform Airy expansion of Dunster takes the form
\cite[Eq.~13.21.22]{DLMF}
\begin{equation}\label{AiryM}
\begin{aligned}
M_{\kappa,\mu}(\xi)
={}&
\frac{1}{2\pi}
\Gamma(2\mu+1)
\Gamma\!\left(\kappa-\mu+\frac12\right)
\widehat c(\kappa,\mu)
\widehat\Psi(\kappa,\mu,\xi)
\\
&\times
\Bigl[
\sin\!\bigl(\pi(\kappa-\mu)\bigr)
\operatorname{Ai}\!\left(\kappa^{2/3}\widehat\zeta\right)
+
\cos\!\bigl(\pi(\kappa-\mu)\bigr)
\operatorname{Bi}\!\left(\kappa^{2/3}\widehat\zeta\right)
\\
&\hspace{2cm}
+\,
\operatorname{envBi}
\!\left(\kappa^{2/3}\widehat\zeta\right)
\mathcal O(\kappa^{-1})
\Bigr]\,.
\end{aligned}
\end{equation}

Here and in the following, $\operatorname{env}\Bi$ denotes the Airy
envelope function of Olver (see \cite[\S2.8(iii)]{DLMF}). We recall its definition for convenience. If $c$ denotes the real root of smallest absolute value of
$\Ai(c)=\Bi(c)$, then
\[\operatorname{env}\Bi(x)
=
\bigl(\Ai(x)^2+\Bi(x)^2\bigr)^{1/2},
\qquad x\leq c\,,
\]
whereas
\[\operatorname{env}\Bi(x)
=
\sqrt{2}\,\Bi(x)\,,
\qquad x\geq c\,.
\]
%

\subsection{Uniform quantization in the transition regime}

We now investigate the transition regime for the zeros of the Kummer
function
\[
a\longmapsto M(a,b,\xi)\,.
\]
In this regime, the Airy argument remains bounded and the Airy
functions appearing in the uniform expansion \eqref{AiryM} must
therefore be kept in their exact form. 

\mn

Using the relation \eqref{Whittaker}, the uniform Airy expansion
\eqref{AiryM} becomes
\[\begin{aligned}
M(a,b,\xi)
={}&
\frac{e^{\xi/2}\xi^{-b/2}}{2\pi}
\Gamma(b)
\Gamma(1-a)
\widehat c(\kappa,\mu)
\widehat\Psi(\kappa,\mu,\xi)
\\
&\times
\Bigl[
\cos(\pi a)
\operatorname{Ai}\!\left(\kappa^{2/3}\widehat\zeta\right)
+
\sin(\pi a)
\operatorname{Bi}\!\left(\kappa^{2/3}\widehat\zeta\right)
\\
&\hspace{2cm}
+\,
R(\kappa,\mu,\xi)
\Bigr]\,.
\end{aligned}
\]
where
\begin{equation}\label{eq:Airy-remainder}
\bigl|R(\kappa,\mu,\xi)\bigr|
\leq
C\kappa^{-1}
\operatorname{env}\Bi
\!\left(\kappa^{2/3}\widehat\zeta\right),
\end{equation}
uniformly in the regime under consideration. Hence, the zeros of $a\mapsto M(a,b,\xi)$ satisfy the uniform
implicit quantization condition
\begin{equation}\label{ImplicitAiryZeros}
\cos(\pi a)
\operatorname{Ai}\!\left(\kappa^{2/3}\widehat\zeta\right)
+
\sin(\pi a)
\operatorname{Bi}\!\left(\kappa^{2/3}\widehat\zeta\right)
+
R(\kappa,\mu,\xi)
=
0\,,
\end{equation}
where $R(\kappa,\mu,\xi)$ satisfies
\eqref{eq:Airy-remainder}. 
Equation \eqref{ImplicitAiryZeros} provides a uniform quantization
condition for the zeros of the Kummer function throughout the Airy
transition regime. It forms the starting point for deriving various
transition asymptotics by considering appropriate scalings of the Airy
argument
$
\kappa^{2/3}\widehat\zeta\,.
$

\mn 

	We first establish rough eigenvalue bounds near the critical energy.
	Besides providing the localization needed below, the argument explains
	the distinguished role of the scale $h^{2/3}\,$.
	
\begin{proposition}\label{lem:rough-transition}
Let $0<\delta<2/3$ and $0<t_0<t_1<+\infty$. Assume that
\[
\rho=4kh=1+th^\delta,
\qquad t\in[t_0,t_1].
\]
Then, uniformly for $t\in[t_0,t_1]$,
\[
0\leq \lambda_k(h)-e_k(h)
=\mathcal O\bigl(h^{3\delta/2}\bigr).
\]
\end{proposition}
\begin{proof}
	By \eqref{eq:1.3}, we have
	\begin{equation}\label{eq:rough-lower-bound}
		e_k(h)\leq\lambda_k(h)\,.
	\end{equation}
	We introduce a Dirichlet decoupling at $x=1$ for the half-line
	operator $T_h^\infty$. The resulting operator is
	\[
	T_h^{\mathrm{dec}}
	=
	T_{h,(0,1)}^{D}\oplus T_{h,(1,\infty)}^{D}\,.
	\]
	The first component is precisely the operator on $(0,1)$ whose
	eigenvalues are $\lambda_j(h)\,$. We denote by
	\[
	\lambda_1^{\mathrm{ext}}(h)
	\leq
	\lambda_2^{\mathrm{ext}}(h)
	\leq\cdots
	\]
	the eigenvalues of the second component, which we shall refer to as
	the exterior component. Thus
	\[
	\sigma\bigl(T_h^{\mathrm{dec}}\bigr)
	=
	\{\lambda_j(h)\}_{j\geq1}
	\cup
	\{\lambda_j^{\mathrm{ext}}(h)\}_{j\geq1}\,,
	\]
	where the union is understood with multiplicities. Let
	\[
	\mu_1(h)\leq\mu_2(h)\leq\cdots
	\]
	denote this union rearranged in increasing order. Since the
	decoupling amounts to imposing one additional Dirichlet condition
	on the form domain, the min--max principle gives
	\begin{equation}\label{eq:rough-interlacing}
		e_j(h)\leq\mu_j(h)\leq e_{j+1}(h)\,.
	\end{equation}
	We next estimate the spectrum of the exterior component. On
	$(1,\infty)\,$, we have
	\[
	x^2+h^2\frac{\nu^2-\frac14}{x^2}
	\geq
	1+2(x-1)-C_\nu h^2\,.
	\]
	Hence, after the scaling $x=1+h^{2/3}y$, the exterior operator is
	bounded from below by
	\[
	1+h^{2/3}A-C_\nu h^2,
	\qquad
	A=-\frac{d^2}{dy^2}+2y\,,
	\]
	on $\mathbb R_+\,$, with Dirichlet boundary condition at $y=0$.
	Let
	\[
	0<\alpha_1<\alpha_2<\cdots
	\]
	denote the eigenvalues of $A$. By the min--max principle,
	\begin{equation}\label{eq:rough-exterior-eigenvalue-lower-bound}
		\lambda_j^{\mathrm{ext}}(h)
		\geq
		1+h^{2/3}\alpha_j-C_\nu h^2\,.
	\end{equation}
	The Airy asymptotics give
	\begin{equation}\label{eq:rough-Airy-growth}
		\alpha_j\sim(3\pi j)^{2/3}\,,
		\qquad j\to+\infty.
	\end{equation}
	By assumption,
	\begin{equation}\label{eq:rough-supercritical-scaling-proof}
		\rho=4kh=1+th^\delta\,.
	\end{equation}
	Set
	\[
	E_h:=1+(t_1+1)h^\delta\,,
	\]
	and let
	\[
	L(h)
	:=
	\#\bigl\{j\geq1:
	\lambda_j^{\mathrm{ext}}(h)\leq E_h\bigr\}\,.
	\]
	By \eqref{eq:rough-exterior-eigenvalue-lower-bound}, if
	$\lambda_j^{\mathrm{ext}}(h)\leq E_h$, then
	\[
	1+h^{2/3}\alpha_j-C_\nu h^2
	\leq
	1+(t_1+1)h^\delta\,.
	\]
	Hence
	\[
	\alpha_j
	\leq
	(t_1+1)h^{\delta-2/3}
	+C_\nu h^{4/3}\,.
	\]
	Using \eqref{eq:rough-Airy-growth}, we infer that
	\begin{equation}\label{eq:rough-L}
		L(h)
		=
		\mathcal O\bigl(h^{3\delta/2-1}\bigr)\,.
	\end{equation}
	We next show that the first $k+L(h)$ eigenvalues of the
	decoupled operator lie below $E_h\,$. Indeed,
	\[
	e_{k+L(h)+1}(h)
	=
	4kh+
	\bigl(4(L(h)+1)+2\nu-2\bigr)h\,.
	\]
	Using \eqref{eq:rough-supercritical-scaling-proof} and
	\eqref{eq:rough-L}, we obtain
	\[
	e_{k+L(h)+1}(h)
	=
	1+th^\delta
	+\mathcal O\bigl(h^{3\delta/2}\bigr)
	+\mathcal O(h)\,.
	\]
	Since $0<\delta<2/3\,$, both remainder terms are $o(h^\delta)$.
	Therefore, uniformly for $t\in[t_0,t_1]\,$,
	\[
	e_{k+L(h)+1}(h)
	\leq
	1+(t_1+1)h^\delta
	=E_h\,,
	\]
	for $h$ sufficiently small. By \eqref{eq:rough-interlacing},
	\[
	\mu_{k+L(h)}(h)
	\leq
	e_{k+L(h)+1}(h)
	\leq E_h\,.
	\]
	Thus the first $k+L(h)$ eigenvalues of the decoupled operator lie
	below $E_h$. By the definition of $L(h)\,$, at most $L(h)$ of them
	belong to the exterior spectrum. Hence at least $k$ of them belong
	to the interior spectrum, and therefore
	\[
	\lambda_k(h)
	\leq
	\mu_{k+L(h)}(h)
	\leq
	e_{k+L(h)+1}(h)\,.
	\]
	Together with \eqref{eq:rough-lower-bound}, this gives
	\[
	0\leq\lambda_k(h)-e_k(h)
	\leq
	e_{k+L(h)+1}(h)-e_k(h)
	=
	4(L(h)+1)h\,.
	\]
	Finally, using \eqref{eq:rough-L}, we obtain
	\[
	\lambda_k(h)-e_k(h)
	=
	\mathcal O\bigl(h^{3\delta/2}\bigr)\,,
	\]
	uniformly for $t\in[t_0,t_1]\,$.
\end{proof}


\mn
The preceding proof also makes the critical exponent transparent.
Indeed, the exterior eigenvalues which may enter a window of size
$h^\delta$ satisfy
\[
	h^{2/3}\alpha_j\lesssim h^\delta\,.
\]
Since $\alpha_j\sim(3\pi j)^{2/3}\,$, this implies
$j\lesssim h^{3\delta/2-1}\,$, and hence the number of such exterior
eigenvalues is of order at most $h^{3\delta/2-1}\,$. This number may grow when $\delta<2/3\,$, whereas
	it remains bounded when $\delta=2/3\,$. At the latter scale, the same
	argument gives the following localization.

	\begin{corollary}\label{cor:kappa-k-critical}
		Let $C>0$ and set
		\begin{equation}\label{eq:kappa-k-definition}
			\kappa_k(h):=\frac{\lambda_k(h)}{4h}\,.
		\end{equation}
		Then, uniformly for integers $k$ satisfying
		\begin{equation}\label{eq:critical-window}
			|4kh-1|\leq C\,h^{2/3}\,,
		\end{equation}
		we have
		\begin{equation}\label{eq:critical-kappa-localization}
			\kappa_k(h)=k+\mathcal O(1)\,,
			\mbox{ as } h\to 0^+\,.
		\end{equation}
	\end{corollary}
	\begin{proof}
		At the scale \eqref{eq:critical-window}, the preceding counting
		argument shows that only a bounded number of exterior eigenvalues
		can occur in the relevant energy window. Hence, for some $L$
		independent of $h\,$,
		\[
		e_k(h)\leq\lambda_k(h)
			\leq e_{k+L+1}(h)
			=e_k(h)+4(L+1)h\,.
		\]
		Since
		\[	e_k(h)=(4k+2\nu-2)h\,,\]
		division by $4h$ gives
		\eqref{eq:critical-kappa-localization}.
	\end{proof}


\subsection{A particular transition scaling}

We now consider the classical Airy scaling near the turning point.
Recall that $\xi=h^{-1}\,$. In terms of the Whittaker parameter $\kappa\,$,
this scaling is
\[	\xi
	=
	4\kappa
	\left(
	1+\frac{t}{(2\kappa)^{2/3}}
	\right)\,,
	\qquad
	t=\mathcal O(1)\,.
\]
This is also the scaling used in the transition formula of
Magnus--Oberhettinger--Soni
\cite[p.~292, Case~3]{MOS1966}.  For $k$ satisfying
	$|4kh-1|\leq Th^{2/3}$, set
\[
\kappa_k(h):=\frac{\lambda_k(h)}{4h}\,,
\]
and define $t_k$ by
\begin{equation}\label{eq:def-tk}
	\xi
	=
	4\kappa_k(h)
	\left(
	1+\frac{t_k}{(2\kappa_k(h))^{2/3}}
	\right)\,.
\end{equation}
By Corollary~\ref{cor:kappa-k-critical},
\[
	\kappa_k(h)=k+\mathcal O(1)\,.
\]

Since $4kh=1+\mathcal O(h^{2/3})\,$, we have
$4h\kappa_k=1+\mathcal O(h^{2/3})\,$, and therefore
\[
t_k=\mathcal O(1)\,.
\]
Thus the transition scaling applies uniformly in the regime under consideration.
Expanding the corresponding Airy variable on either side of the
turning point under \eqref{eq:def-tk} gives, uniformly for bounded
$t_k\,$,
\begin{equation}\label{eq:Airy-variable-transition}
	\kappa_k^{2/3}\widehat\zeta
	=
	t_k+\mathcal O(\kappa_k^{-2/3})\,.
\end{equation}
Hence
\[
	\Ai(\kappa_k^{2/3}\widehat\zeta)
	=
	\Ai(t_k)+\mathcal O(\kappa_k^{-2/3})\,,
\]
and
\[
	\Bi(\kappa_k^{2/3}\widehat\zeta)
	=
	\Bi(t_k)+\mathcal O(\kappa_k^{-2/3})\,.
\]
Substituting these estimates into \eqref{ImplicitAiryZeros}, and using
\eqref{eq:Airy-remainder}, we obtain
\begin{equation}\label{eq:Airy-transition-quantization}
	\cos(\pi a)\Ai(t_k)
	+
	\sin(\pi a)\Bi(t_k)
	=
	\mathcal O(\kappa_k^{-2/3})\,.
\end{equation}
To write this condition in phase form, set
\[
	A(s):=\bigl(\Ai(s)^2+\Bi(s)^2\bigr)^{1/2}\,.
\]
Since the Airy functions $\Ai$ and $\Bi$ have no common zeros,
$A(s)>0$ for every $s\in\mathbb R\,$. We may therefore introduce a smooth phase $\theta(s)$ by
\begin{equation}\label{eq:Airy-phase}
	\Ai(s)=A(s)\sin\theta(s)\,,
	\qquad
	\Bi(s)=A(s)\cos\theta(s)\,,
\end{equation}
with the continuous determination such that $\theta(s)\to0$ as $s\to+\infty$.
Moreover, differentiating
\[
\tan\theta(s)=\frac{\Ai(s)}{\Bi(s)}
\]
and using the Wronskian identity
$W(\Ai,\Bi)=1/\pi$ (see \cite[Eq.~(9.2.7)]{DLMF}), we obtain
\begin{equation}\label{eq:Airy-phase-derivative}
	\theta'(s)
	=
	\frac{\Ai'(s)\Bi(s)-\Ai(s)\Bi'(s)}
	{\Ai(s)^2+\Bi(s)^2}
	=
	-\frac{1}{\pi\bigl(\Ai(s)^2+\Bi(s)^2\bigr)}
	<0\,.
\end{equation}
Thus $\theta$ is strictly decreasing.  Using the standard Airy asymptotics, we also have
\[
\theta(s)
=
\frac23(-s)^{3/2}+\frac{\pi}{4}+o(1),
\qquad s\to-\infty,
\]
whereas
\[
\theta(s)
=
\frac12\exp\left(-\frac43s^{3/2}\right)
\left(1+O(s^{-3/2})\right),
\qquad s\to+\infty.
\]
Then \eqref{eq:Airy-transition-quantization} becomes
\[
	\sin\bigl(\pi a+\theta(t_k)\bigr)
	=
	\mathcal O(\kappa_k^{-2/3})\,.
\]
Thus,
\[
	a=m-\frac{\theta(t_k)}{\pi}
	+\mathcal O(\kappa_k^{-2/3})\,,
	\mbox{ for } m\in\mathbb Z\,.
\]
On the subcritical side, the $k$-th zero approaches $1-k$.
Hence $m=1-k\,$, and consequently
\begin{equation}\label{eq:transition-tk}
	a_k(\xi)
	=
	1-k-\frac{\theta(t_k)}{\pi}
	+\mathcal O(\kappa_k^{-2/3})\,.
\end{equation}
Since $\kappa_k^{-2/3}=\mathcal O(h^{2/3})\,$, we finally obtain
\begin{equation}\label{eq:transition-ak}
	a_k(\xi)
	=
	1-k-\frac{\theta(t_k)}{\pi}
	+\mathcal O(h^{2/3})\,.
\end{equation}
By Sturm's oscillation theorem, the zero $a_k(\xi)$ in
\eqref{eq:transition-ak} corresponds to the $k$-th eigenvalue
of the singular harmonic oscillator. We can now state the resulting
transition asymptotics directly in terms of $k$ and $h$.

\begin{proposition}\label{prop:critical}
	For every $T>0\,$ and for integers $k$ satisfying
	\begin{equation}\label{eq:critical-window1}
		|4kh-1|\leq T h^{2/3}\,,
	\end{equation}
	set
	\begin{equation}\label{eq:tau-k}
		\tau_k(h)
		:=
		\frac{1-4kh}{2^{2/3}h^{2/3}}\,.
	\end{equation}
	Then,  uniformly for these $k$'s,  we have
	\begin{equation}\label{eq:critical-expansion}
		\begin{aligned}
			\lambda_k(h)
			={}&
			e_k(h)
			+
			\frac{4h}{\pi}\theta\bigl(\tau_k(h)\bigr)
			\\
			&-
			\frac{4}{\pi\,2^{2/3}}
			h^{4/3}
			\left(
			2\nu-2+\frac{4}{\pi}\theta\bigl(\tau_k(h)\bigr)
			\right)
			\theta'\bigl(\tau_k(h)\bigr)
			+
			\mathcal O(h^{5/3}),
		\end{aligned}
	\end{equation}
	as $h\to0^+\,$. 
\end{proposition}

\begin{remark}[Asymptotic scale]
	Since $|\tau_k(h)|\leq 2^{-2/3}T\,$, the functions $\theta(\tau_k(h))$
	and $\theta'(\tau_k(h))$ remain uniformly bounded. Moreover,
	$\theta>0$ on $\mathbb R\,$.
	Therefore, in \eqref{eq:critical-expansion}, the first correction
	is of order $h$, whereas the second one is
	$\mathcal O(h^{4/3})\,$.
 Thus these terms form an asymptotic scale uniformly in the
	transition window \eqref{eq:critical-window1}.
\end{remark}

\begin{proof}[Proof of Proposition~\ref{prop:critical}]
	Let
	\[
	\tau_k=\tau_k(h)
	:=
	\frac{1-4kh}{2^{2/3}h^{2/3}}.
	\]
	Under the assumption of the proposition, we have
	\[
	|\tau_k|\leq 2^{-2/3}T\,.
	\]
	With $t_k$ defined by \eqref{eq:def-tk}, formula
	\eqref{eq:transition-ak} gives
	\begin{equation}\label{eq:transition-ak-proof}
		a_k(h^{-1})
		=
		1-k-\frac{\theta(t_k)}{\pi}
		+\mathcal O(h^{2/3})\,.
	\end{equation}
	Since the $k$-th eigenvalue and the corresponding $a$-zero are
	related by
	\[
	\lambda_k(h)
	=
	h\bigl(2b-4a_k(h^{-1})\bigr)\,,
	\]
	we obtain
	\begin{equation}\label{eq:Airy-eigenvalue-transition-ek}
		\lambda_k(h)
		=
		e_k(h)
		+
		\frac{4h}{\pi}\theta(t_k)
		+
		\mathcal O(h^{5/3})\,,
	\end{equation}
	where
	\[
	e_k(h)=(4k+2\nu-2)h\,.
	\]
	We now eliminate the auxiliary parameter $t_k\,$. For convenience,
	set
	\begin{equation}\label{eq:def-q-transition}
		q(s):=
		2\nu-2+\frac{4}{\pi}\theta(s)\,.
	\end{equation}
	Since
	\[
	4kh
	=
	1-2^{2/3}\tau_k h^{2/3}\,,
	\]
	equation \eqref{eq:Airy-eigenvalue-transition-ek} yields
	\begin{equation}\label{eq:lambda-tau-tk}
		\lambda_k(h)
		=
		1-2^{2/3}\tau_k h^{2/3}
		+hq(t_k)
		+\mathcal O(h^{5/3})\,.
	\end{equation}
	On the other hand, since
	\[
	4\kappa_k=\frac{\lambda_k(h)}{h}\,,
	\]
	the defining relation \eqref{eq:def-tk} is equivalently 
	\begin{equation}\label{eq:exact-transition-relation}
		1
		=
		\lambda_k(h)
		+
		2^{2/3}t_k h^{2/3}\lambda_k(h)^{1/3}\,.
	\end{equation}
	From \eqref{eq:lambda-tau-tk} and the boundedness of $\tau_k$ and
	$t_k$, we have
	\[
	\lambda_k(h)
	=
	1-2^{2/3}\tau_k h^{2/3}
	+\mathcal O(h)\,,
	\]
	and therefore
	\begin{equation}\label{eq:lambda-one-third}
		\lambda_k(h)^{1/3}
		=
		1-\frac{2^{2/3}}{3}\tau_k h^{2/3}
		+\mathcal O(h)\,.
	\end{equation}
	Substituting \eqref{eq:lambda-tau-tk} and
	\eqref{eq:lambda-one-third} into
	\eqref{eq:exact-transition-relation}, we obtain
	\[
	0
	=
	2^{2/3}(t_k-\tau_k)h^{2/3}
	+
	hq(t_k)
	-
	\frac{2^{4/3}}{3}t_k\tau_k h^{4/3}
	+
	\mathcal O(h^{5/3})\,.
	\]
	Since $t_k$ and $\tau_k$ remain bounded, this first implies
	\begin{equation}\label{eq:tk-tau-first}
		t_k-\tau_k=\mathcal O(h^{1/3})\,.
	\end{equation}
	The function $\theta\,$, and hence $q\,$, is smooth. Thus
	\[
	q(t_k)=q(\tau_k)+\mathcal O(h^{1/3})\,,
	\]
	and a second use of the preceding identity gives
	\begin{equation}\label{eq:tk-tau-refined}
		t_k
		=
		\tau_k
		-
		2^{-2/3}q(\tau_k)h^{1/3}
		+
		\mathcal O(h^{2/3})\,.
	\end{equation}
	The Taylor expansion of $\theta$ at $\tau_k$ now gives
	\begin{equation}\label{eq:theta-tk-tau}
		\theta(t_k)
		=
		\theta(\tau_k)
		-
		2^{-2/3}q(\tau_k)\theta'(\tau_k)h^{1/3}
		+
		\mathcal O(h^{2/3})\,,
	\end{equation}
	uniformly for $|\tau_k|\leq T\,$.\\
	Inserting
	\eqref{eq:theta-tk-tau} into
	\eqref{eq:Airy-eigenvalue-transition-ek}, we finally obtain
	\[
	\lambda_k(h)
	=
	e_k(h)
	+
	\frac{4h}{\pi}\theta(\tau_k)
	-
	\frac{4}{\pi\,2^{2/3}}
	q(\tau_k)\theta'(\tau_k)h^{4/3}
	+
	\mathcal O(h^{5/3})\,.
	\]
	Recalling \eqref{eq:def-q-transition}, this is precisely
	\eqref{eq:critical-expansion}.
\end{proof}

\begin{remark}[Eigenvalue spacing]
	Uniformly in the transition window,
	\[
	\lambda_{k+1}(h)-\lambda_k(h)
	=
	4h
	-\frac{2^{10/3}}{\pi}
	\theta'\bigl(\tau_k(h)\bigr)h^{4/3}
	+\mathcal O(h^{5/3})\,.
	\]
	Since $\theta'<0$, the first correction to the leading spacing
	$4h$ is positive.
\end{remark}


\subsection{Matching with the subcritical and supercritical regimes}\label{subsec:matching}
\subsubsection{Accurate Dunster expansions}\label{ss:accurate-Dunster}
The classical uniform Airy expansions from the NIST
Handbook~\cite{DLMF}, used in Section~\ref{sec:transition}, are perfectly
suited to the transition regime, where the Airy variable
$s=\kappa_k(h)^{2/3}\widehat\zeta$ remains bounded.
They could also be used to carry out the supercritical matching. On
the subcritical side, however, the matching involves an exponentially
small contribution, and the classical remainder estimates are too
large to detect it. Indeed, in the matching regimes introduced below,
Propositions~\ref{prop:2.5} and~\ref{lem:rough-transition} give,
respectively on the subcritical and supercritical sides,
\begin{equation}\label{eq:kappa-matching}
	\kappa_k(h)=\frac{\lambda_k(h)}{4h}
	=\frac{1+\mathcal O(h^\delta)}{4h}\,.
\end{equation}
The proof of Proposition~\ref{prop:critical} starts from the
uniform quantization condition~\eqref{ImplicitAiryZeros}, which in the
present spectral setting reads
\begin{equation}\label{eq:Airy-spectral-quantization}
	\cos(\pi a)\Ai\bigl(\kappa_k(h)^{2/3}\widehat\zeta\bigr)
	+
	\sin(\pi a)\Bi\bigl(\kappa_k(h)^{2/3}\widehat\zeta\bigr)
	+
	R(\kappa_k(h),\mu,\xi)
	=0\,,
\end{equation}
where $\widehat\zeta$ is the Liouville--Green variable. Moreover, by \eqref{eq:Airy-remainder},
\[
|R(\kappa_k(h),\mu,\xi)|
\leq
C\kappa_k(h)^{-1}
\operatorname{env}\Bi
\bigl(\kappa_k(h)^{2/3}\widehat\zeta\bigr)\,.
\]
Thus, dividing \eqref{eq:Airy-spectral-quantization} by
$\Bi\bigl(\kappa_k(h)^{2/3}\widehat\zeta\bigr)\,$, the contribution
of the remainder is only controlled by $\mathcal O(h)\,$,
where we used \eqref{eq:kappa-matching}.
This estimate is, however, far too weak to detect the exponentially small
$\Ai(s)/\Bi(s)$ ratio on the subcritical side, when $s\to+\infty\,$.
Indeed, by the standard Airy asymptotics
\cite[(9.7.5), (9.7.7)]{DLMF},
\begin{equation}\label{eq:Airy-ratio-plus}
	\frac{\Ai(s)}{\Bi(s)}
	=
	\frac12 \exp\left(-\frac43 s^{3/2}\right)
	\left(1+\mathcal O(s^{-3/2})\right)\,,
	\mbox{ as } s\to+\infty\,.
\end{equation}

\mn
We therefore use here the very recent uniform Airy expansions for
Whittaker functions obtained by Dunster~\cite[Theorem~3.1, Eqs.~(3.40)--(3.41)]
{Dunster2026}.  Our spectral condition can be written in the form
\begin{equation}\label{eq:Dunster-spectral-condition}
	\mathcal A(s,-\pi a)\,\check A
	+
	\mathcal A'(s,-\pi a)\,\check B
	=0\,,
\end{equation}
where
\[
s=\kappa_k(h)^{2/3}\widehat\zeta,
\qquad
\mathcal A(s,-\pi a)
:=
\cos(\pi a)\Ai(s)
+
\sin(\pi a)\Bi(s)\,,
\]
and $\mathcal A'$ denotes the derivative with respect to the first
variable $s\,$. The coefficient functions $\check A$ and $\check B$
are those introduced by Dunster. Moreover, his expansions
\cite[Eqs.~(3.37)--(3.38)]{Dunster2026}, together with
\eqref{eq:kappa-matching}, imply
\begin{equation}\label{eq:Dunster-ratio}
	\frac{\check B}{\check A}
	=
	\mathcal O\bigl(\kappa_k(h)^{-4/3}\bigr)
	=
	\mathcal O(h^{4/3})\,.
\end{equation}
Dividing \eqref{eq:Dunster-spectral-condition} by $\check A\,$, we recover
the quantization condition \eqref{ImplicitAiryZeros}, with the new
structured remainder
\begin{equation}\label{eq:Dunster-remainder}
	R_D
	=
	\mathcal A'(s,-\pi a)\frac{\check B}{\check A}\,,
\end{equation}
where $R_D$ denotes the Dunster remainder.\\
The crucial point for the matching argument is the particular structure
of this remainder. Indeed, setting $r=\check B/\check A$, we obtain
\begin{equation}\label{eq:Dunster-tangent}
	\tan(\pi a)
	=
	-\frac{\Ai(s)+r\,\Ai'(s)}
	{\Bi(s)+r\,\Bi'(s)},
	\qquad
	r
	=\mathcal O(h^{4/3})\,.
\end{equation}
In the transition regime of Proposition~\ref{prop:critical}, we have
\[
s=\tau_k(h)+\mathcal O(h^{1/3})\,,
\]
so that $s$ remains bounded. In this case, the new Dunster form does not
make a substantial difference. In the matching regions considered below, however, the Airy variable $s$
is no longer bounded. More precisely, for fixed $t>0\,$, we shall show that
\begin{equation}\label{eq:Airy-variable-matching}
	s
	=
	\pm\frac{t}{2^{2/3}}h^{\delta-2/3}(1+o(1))
	\longrightarrow\pm\infty\,,
\end{equation}
with the upper sign on the subcritical side and the lower sign on the
supercritical side. On the subcritical side, the crucial point is that
\begin{equation}\label{eq:Dunster-factorization}
	\frac{\Ai(s)+r\Ai'(s)}
	{\Bi(s)+r\Bi'(s)}
	=
	\frac{\Ai(s)}{\Bi(s)}
	\left(
	\frac{1+r\,\Ai'(s)/\Ai(s)}
	{1+r\,\Bi'(s)/\Bi(s)}
	\right)\,.
\end{equation}
For $s\to+\infty\,$, the standard Airy asymptotics
\cite[Sec.~9.7]{DLMF} give
\begin{equation}\label{eq:Airy-log-derivatives-plus}
	\frac{\Ai'(s)}{\Ai(s)}
	=
	-s^{1/2}\left(1+\mathcal O(s^{-3/2})\right),
	\qquad
	\frac{\Bi'(s)}{\Bi(s)}
	=
	s^{1/2}\left(1+\mathcal O(s^{-3/2})\right)\,.
\end{equation}
Since
\[
r=\mathcal O(h^{4/3})\,,
\]
and using \eqref{eq:Airy-variable-matching}, the second factor in
\eqref{eq:Dunster-factorization} is
\[
1+\mathcal O\bigl(h^{4/3}s^{1/2}\bigr)
=
1+\mathcal O\bigl(h^{1+\delta/2}\bigr)
\]
in the matching regime. Combining this with
\eqref{eq:Airy-ratio-plus}, we obtain
\begin{equation}\label{eq:Airy-ratio-subcritical}
	\frac{\Ai(s)+r\Ai'(s)}
	{\Bi(s)+r\Bi'(s)}
	=
	\frac12
	\exp\left(-\frac43s^{3/2}\right)
	\left(1+\mathcal O\bigl(h^{1-3\delta/2}\bigr)\right)\,.
\end{equation}
Thus the exponentially small term is preserved.
This is precisely what was lost in the additive remainder of the
classical Airy expansion. Then, as in Proposition~\ref{prop:critical}, to obtain the eigenvalue
asymptotics, we write
\begin{equation}\label{eq:a-eta-matching}
	a=1-k-\eta_k\,,
	\qquad
	\eta_k=\frac{\lambda_k(h)-e_k(h)}{4h}\,.
\end{equation}
The quantization condition \eqref{eq:Dunster-tangent}, together with
\eqref{eq:a-eta-matching}, gives
\[
\tan(\pi\eta_k)
=
\frac{\Ai(s)+r\Ai'(s)}
{\Bi(s)+r\Bi'(s)}\,.
\]
Using \eqref{eq:Airy-ratio-subcritical}, the right-hand side is
exponentially small. Since we are on the branch for which
$\eta_k\to0$, we have
$
\tan(\pi\eta_k)
=
\pi\eta_k\bigl(1+\mathcal O(\eta_k^2)\bigr)\,.
$
Moreover, $\eta_k$ is exponentially small, so that this additional
error is negligible with respect to
$\mathcal O(h^{1-3\delta/2})$. Therefore,
\begin{equation}\label{eq:eigenvalue-shift-Airy}
	\lambda_k(h)-e_k(h)
	=
	\frac{2h}{\pi}
	\exp\left(-\frac43s^{3/2}\right)
	\left(1+\mathcal O(h^{1-3\delta/2})\right)\,.
\end{equation}

\mn
On the supercritical side, $s\to-\infty$ and both Airy functions are
oscillatory with the same algebraic size. More precisely, setting
\begin{equation}\label{phase}
\Phi(s)=\frac23(-s)^{3/2}+\frac{\pi}{4}\,,
\end{equation}
the standard Airy asymptotics \cite[Sec.~9.7]{DLMF} give
\begin{equation}\label{eq:Airy-asymptotics-minus}
	\begin{aligned}
		\Ai(s)
		&=
		\frac{(-s)^{-1/4}}{\sqrt{\pi}}
		\left(\sin\Phi(s)+\mathcal O((-s)^{-3/2})\right),\\
		\Bi(s)
		&=
		\frac{(-s)^{-1/4}}{\sqrt{\pi}}
		\left(\cos\Phi(s)+\mathcal O((-s)^{-3/2})\right)\,,
	\end{aligned}
\end{equation}
and
\begin{equation}\label{eq:Airy-derivatives-minus}
	\begin{aligned}
		\Ai'(s)
		&=
		-\frac{(-s)^{1/4}}{\sqrt{\pi}}
		\left(\cos\Phi(s)+\mathcal O((-s)^{-3/2})\right)\,,\\
		\Bi'(s)
		&=
		\frac{(-s)^{1/4}}{\sqrt{\pi}}
		\left(\sin\Phi(s)+\mathcal O((-s)^{-3/2})\right)\,.
	\end{aligned}
\end{equation}
Thus, on the supercritical side, the matching is governed by the
oscillatory phase $\Phi(s)$. In contrast with the subcritical case, we do not form ratios of Airy
functions, which would be singular at their zeros. Instead, we insert
these oscillatory expansions directly into the quantization condition.

\subsubsection{Subcritical matching}

We consider the subcritical matching regime 
\begin{equation}\label{eq:subcritical-matching-regime}
	4kh=1-th^\delta\,,
	\qquad
	t\in[t_0,t_1]\,,
	\qquad
	0<\delta<\frac23\,.
\end{equation}
By Proposition~\ref{prop:2.5}, for any fixed $\eta\in(0,1)$,
\[
0\leq \lambda_k(h)-e_k(h)
\leq
\exp\left(
-\frac{2(1-\eta)}{3}
t^{3/2}h^{3\delta/2-1}
\right).
\]
Since
\[
e_k(h)=1-th^\delta+(2\nu-2)h\,,
\]
it follows that, uniformly for $t\in[t_0,t_1]\,$,
\begin{equation}\label{eq:lambda-subcritical-size}
	1-\lambda_k(h)
	=
	th^\delta\bigl(1+\mathcal O(h^{1-\delta})\bigr).
\end{equation}
We now recall the definition of the Liouville--Green variable entering the recent
uniform Airy expansion of Dunster~\cite{Dunster2026}. For
\[
\kappa=\frac{\lambda}{4h}\,,
\qquad
\mu=\frac{\nu}{2}\,,
\qquad
z=\frac4\lambda\,,
\qquad
\alpha=\frac{2\mu}{\kappa}
=\frac{4\nu h}{\lambda}\,,
\]
the Liouville--Green variable $\widehat\zeta=\widehat\zeta(z,\alpha)$
is defined, in our notation, by
\begin{equation}\label{eq:LG-zeta-definition}
	\frac23\widehat\zeta(z,\alpha)^{3/2}
	=
	\int_{z_+(\alpha)}^{z}
	\sqrt{\frac{v^2-4v+\alpha^2}{4v^2}}\,dv\,,
	\qquad
	z_+(\alpha)=2+\sqrt{4-\alpha^2}\,.
\end{equation}
Since $\alpha=\mathcal O(h)\,$, we have
$z_+(\alpha)=4+\mathcal O(h^2)\,$. Therefore, in the regime considered here,
\[
\widehat\zeta\left(\frac4\lambda,\frac{4\nu h}{\lambda}\right)
=
2^{2/3}(1-\lambda)
\bigl(1+\mathcal O(h^\delta)\bigr)\,.
\]
Evaluating at $\lambda=\lambda_k(h)\,$, we set
\[
\kappa_k(h)=\frac{\lambda_k(h)}{4h}\,,
\qquad
\widehat\zeta_k(h)
=
\widehat\zeta\left(
\frac4{\lambda_k(h)},
\frac{4\nu h}{\lambda_k(h)}
\right)\,,
\qquad
s_k(h)=\kappa_k(h)^{2/3}\widehat\zeta_k(h)\,.
\]
Then \eqref{eq:lambda-subcritical-size} gives
\begin{equation}\label{eq:sk-subcritical-size}
	s_k(h)
	=
	\frac{t}{2^{2/3}}h^{\delta-2/3}
	\left(1+\mathcal O\left(
	h^{\min\{\delta,1-\delta\}}\right)\right)
	\longrightarrow+\infty\,.
\end{equation}
Therefore, applying \eqref{eq:eigenvalue-shift-Airy} with
$s=s_k(h)$, we obtain
\begin{equation}\label{eq:eigenvalue-shift-subcritical}
	\lambda_k(h)-e_k(h)
	=
	\frac{2h}{\pi}
	\exp\left(-\frac43s_k(h)^{3/2}\right)
	\left(1+\mathcal O(h^{1-3\delta/2})\right)\,.
\end{equation}
It remains to compare the exponent in this formula with the
tunneling action $S^*(e_k(h))$. We shall prove that
\begin{equation}\label{eq:Airy-action-identification-goal}
	\frac43s_k(h)^{3/2}
	=
	\frac{2S^*(e_k(h))}{h}
	+\mathcal O(h^{\delta/2})\,.
\end{equation}
To prove this, we first consider a general $\lambda$. Setting
$
s=\kappa^{2/3}\widehat\zeta,
$
and using \eqref{eq:LG-zeta-definition}, we obtain
\begin{equation}\label{eq:Airy-exponent-LG}
	\frac43s^{3/2}
	=
	2\kappa
	\int_{z_+(\alpha)}^{4/\lambda}
	\sqrt{\frac{v^2-4v+\alpha^2}{4v^2}}\,dv\,.
\end{equation}
The change of variables
\[
v=\frac{4x^2}{\lambda}
\]
gives
\begin{equation}\label{eq:LG-action-x}
	\frac43s^{3/2}
	=
	\frac{2}{h}
	\int_{x_+(\lambda,h)}^1
	\sqrt{x^2-\lambda+\frac{\nu^2h^2}{x^2}}\,dx\,,
\end{equation}
where
\begin{equation}\label{eq:x-plus-subcritical}
	x_+(\lambda,h)^2
	=
	\frac{\lambda+\sqrt{\lambda^2-4\nu^2h^2}}{2}\,.
\end{equation}
Set
\[
I_h(\lambda)
:=
\int_{x_+(\lambda,h)}^1
\sqrt{x^2-\lambda+\frac{\nu^2h^2}{x^2}}\,dx\,.
\]
Since
\[
S^*(\lambda)
=
\int_{\sqrt{\lambda}}^1\sqrt{x^2-\lambda}\,dx\,,
\]
a direct comparison of the two integrals yields, uniformly in the
regime considered here,
\begin{equation}\label{eq:Ih-Sstar}
	I_h(\lambda)
	=
	S^*(\lambda)
	+
	\mathcal O\bigl(h^2\sqrt{1-\lambda}\bigr)
	+
	\mathcal O(h^3)\,.
\end{equation}
Indeed, on $[\sqrt{\lambda},1]\,$, we have
\[
0\leq
\sqrt{x^2-\lambda+\frac{\nu^2h^2}{x^2}}
-\sqrt{x^2-\lambda}
\leq
\frac{\nu^2h^2}{x^2\sqrt{x^2-\lambda}}\,,
\]
and
\[
\int_{\sqrt{\lambda}}^1
\frac{dx}{x^2\sqrt{x^2-\lambda}}
=
\frac{\sqrt{1-\lambda}}{\lambda}\,.
\]
Finally,
$\sqrt{\lambda}-x_+(\lambda,h)=O(h^2)$ and the integrand is $O(h)$
on $[x_+(\lambda,h),\sqrt{\lambda}]\,$, so that the contribution of this
interval is $O(h^3)$. Applying \eqref{eq:Ih-Sstar} to $\lambda=\lambda_k(h)$ and using
\eqref{eq:lambda-subcritical-size}, we obtain
\begin{equation}\label{eq:Ih-Sstar-lambdak}
	\frac{I_h(\lambda_k(h))}{h}
	=
	\frac{S^*(\lambda_k(h))}{h}
	+
	O\bigl(h^{1+\delta/2}\bigr)\,.
\end{equation}
Moreover,
\[
S^{*\prime}(E)
=
-\frac12\operatorname{arcosh}(E^{-1/2})
=
\mathcal O\bigl(\sqrt{1-E}\bigr),
\qquad E\uparrow1.
\]
Since $\lambda_k(h)-e_k(h)=\mathcal O(h)$, the mean value theorem gives
\[
S^*(\lambda_k(h))-S^*(e_k(h))
=
S^{*\prime}(\xi_k)
\bigl(\lambda_k(h)-e_k(h)\bigr),
\]
for some $\xi_k$ between $e_k(h)$ and $\lambda_k(h)\,$. Since
$1-\xi_k\asymp h^\delta\,$, it follows that
\[
\frac{S^*(\lambda_k(h))-S^*(e_k(h))}{h}
=
\mathcal O(h^{\delta/2})\,.
\]
Combining this with \eqref{eq:LG-action-x} and
\eqref{eq:Ih-Sstar-lambdak}, we obtain
\begin{equation}\label{eq:Airy-action-identification}
	\frac43s_k(h)^{3/2}
	=
	\frac{2S^*(e_k(h))}{h}
	+
	\mathcal O(h^{\delta/2})\,.
\end{equation}
Finally, inserting \eqref{eq:Airy-action-identification} into
\eqref{eq:eigenvalue-shift-subcritical} yields
\begin{equation}\label{eq:subcritical-matching-final}
	\lambda_k(h)
	=
	e_k(h)
	+
	\frac{2h}{\pi}
	\exp\left(-\frac{2S^*(e_k(h))}{h}\right)
	\left(
	1+
	\mathcal O\left(
	h^{\min\{\delta/2,\;1-3\delta/2\}}
	\right)
	\right)\,,
\end{equation}
uniformly for $t\in[t_0,t_1]$ in
\eqref{eq:subcritical-matching-regime}. Thus, in the matching regime
\[
4kh=1-th^\delta\,,
\qquad
t\in[t_0,t_1]\,,
\qquad
0<\delta<\frac23\,,
\]
the Airy asymptotics agree with the subcritical tunneling formula.
In particular, we recover both the exponential action and the
prefactor $2h/\pi$ in Theorem~\ref{thm:main}.

\subsubsection{Supercritical matching}

We now consider the supercritical matching regime
\begin{equation}\label{eq:supercritical-matching-regime}
	4kh=1+th^\delta\,,
	\qquad
	t\in[t_0,t_1]\,,
	\qquad
	0<\delta<\frac23\,.
\end{equation}
By Proposition~\ref{lem:rough-transition},
\[
\lambda_k(h)
=
e_k(h)+\mathcal O(h^{3\delta/2})\,.
\]
Since
\[
e_k(h)
=
1+th^\delta+(2\nu-2)h\,,
\]
we obtain, uniformly for $t\in[t_0,t_1]\,$,
\begin{equation}\label{eq:lambda-supercritical-size}
	\lambda_k(h)-1
	=
	th^\delta
	\left(
	1+
	\mathcal O\left(
	h^{\min\{1-\delta,\delta/2\}}
	\right)
	\right)\,.
\end{equation}
As on the subcritical side, we set
\[
\kappa_k(h)
=
\frac{\lambda_k(h)}{4h}\,,
\qquad
s_k(h)
=
\kappa_k(h)^{2/3}\widehat\zeta_k(h)\,.
\]
The corresponding expansion of the Liouville--Green variable gives
\begin{equation}\label{eq:sk-supercritical-size}
	s_k(h)
	=
	-\frac{t}{2^{2/3}}
	h^{\delta-2/3}
	\left(
	1+
	\mathcal O\left(
	h^{\min\{1-\delta,\delta/2\}}
	\right)
	\right)
	\longrightarrow-\infty\,.
\end{equation}
We may therefore use the oscillatory Airy expansions
\eqref{eq:Airy-asymptotics-minus}--\eqref{eq:Airy-derivatives-minus}
from Subsection~\ref{ss:accurate-Dunster}. Moreover, by \eqref{eq:Dunster-tangent} and
\eqref{eq:sk-supercritical-size},
\begin{equation}\label{eq:r-sk-supercritical}
	r(-s_k(h))^{1/2}
	=
	\mathcal O(h^{1+\delta/2})\,.
\end{equation}
The quantization condition \eqref{eq:Dunster-tangent} can equivalently
be written as
\[
\sin(\pi a_k)
\bigl(
\Bi(s_k(h))+r\Bi'(s_k(h))
\bigr)
+
\cos(\pi a_k)
\bigl(
\Ai(s_k(h))+r\Ai'(s_k(h))
\bigr)
=0\,.
\]
Using the oscillatory Airy expansions from
Subsection~\ref{ss:accurate-Dunster} and
\eqref{eq:r-sk-supercritical}, we obtain easily
\begin{equation}\label{eq:supercritical-phase-condition}
	\sin\left(
	\pi a_k
	+
	\frac23(-s_k(h))^{3/2}
	+
	\frac\pi4
	\right)
	=
	\mathcal O\left(h^{1-3\delta/2}\right)\,.
\end{equation}
Using the same labeling of the $k$-th zero as in Proposition~\ref{prop:critical},
we obtain
\[
\pi a_k
=
(1-k)\pi
-
\frac23(-s_k(h))^{3/2}
-
\frac\pi4
+
\mathcal O\left(h^{1-3\delta/2}\right)\,.
\]
Since
\[
a_k
=
\frac{\nu}{2}
-
\frac{\lambda_k(h)}{4h}
+
\frac12\,,
\]
we obtain
\begin{equation}\label{eq:lambda-supercritical-Airy}
	\lambda_k(h)
	=
	4h\left(
	k+\frac\nu2-\frac14
	\right)
	+
	\frac{8h}{3\pi}
	(-s_k(h))^{3/2}
	+
	\mathcal O\left(h^{2-3\delta/2}\right).
\end{equation}
By \eqref{eq:sk-supercritical-size},
\[
(-s_k(h))^{3/2}
=
\frac{t^{3/2}}{2}
h^{3\delta/2-1}
\left(
1+
\mathcal O\left(
h^{\min\{1-\delta,\delta/2\}}
\right)
\right).
\]
Substituting this into \eqref{eq:lambda-supercritical-Airy}, and using
$0<\delta<2/3$, we obtain
\begin{equation}\label{eq:supercritical-matching-final}
	\lambda_k(h)
	=
	4h\left(
	k+\frac\nu2-\frac14
	\right)
	+
	\frac{4}{3\pi}
	t^{3/2}h^{3\delta/2}
	+
	o(h^{3\delta/2})\,,
\end{equation}
uniformly for $t\in[t_0,t_1]$. This is precisely the threshold expansion of the supercritical
Bohr--Sommerfeld formula. Indeed, setting
\[
\rho_k
=
4h\left(
k+\frac\nu2-\frac14
\right)
=
1+th^\delta+(2\nu-1)h\,,
\]
and using
\[
E_0(1+\varepsilon)
=
1+\varepsilon
+
\frac{4}{3\pi}\varepsilon^{3/2}
+
O(\varepsilon^2),
\qquad
\varepsilon\downarrow0\,,
\]
we recover \eqref{eq:supercritical-matching-final}. Hence the Airy
asymptotics and the supercritical Bohr--Sommerfeld formula match in
the regime \eqref{eq:supercritical-matching-regime}.


\section{Uniform gap}\label{sec:gap}

We prove in this section the non-asymptotic ineqiality on the gap $\lambda_{k+1}(h)-\lambda_k(h)$\,. 

\begin{proof}[Proof of Theorem~\ref{thm:gap}]
For simplicity, we drop $h$ from the notation and write $\lambda_k$ for $\lambda_{k}(h)\,$. We will prove that $\lambda_{k+1}>\lambda_{k}+4h\,$.

For any $\mu\geq \lambda_1\,,$ recall from \eqref{regular-solution} that
\begin{equation}\label{eq:def.u-mu}
u_\mu(x)=x^{\nu+\frac12}e^{-x^2/2h}\,
M\Bigl(\tfrac{\nu+1}{2}-\tfrac{\mu}{4h},\,\nu+1\,,\,\tfrac{x^2}{h}\Bigr)
\end{equation}
is  a solution of $(T_h-\mu)u=0\,$ which is regular on $\overline\R_+$
In particular,
\[f_k=u_{\lambda_k}=x^{\nu+\frac12}e^{-x^2/2h}\,
M\Bigl(\tfrac{\nu+1}{2}-\tfrac{\lambda_k}{4h}\,,\,\nu+1,\,\tfrac{x^2}{h}\Bigr)\]
is  an eigenfunction of $T_h$ associated to the eigenvalue $\lambda_{k}\,$.

Now we take $\mu=\lambda_k+4h$ and express $u_\mu$ in terms of $f_k$ as
\begin{equation}\label{eq:identity}
u_\mu=c^{-1} Af_k\,,\quad c=\frac{\nu+1}{2}+\frac{\lambda_k}{4h},\quad A=\frac{x}{2}\frac{d}{dx}+\frac{\lambda_k+h}{4h}-\frac{x^2}{2h}\,.
\end{equation} 
This follows from a direct calculation  of $Af_k$ and by using the following identity
\[
z\frac{d}{dz}M(a,b,z)+(b-a-z)M(a,b,z)=(b-a)M(a-1,b,z)\,,
\]
which results from  combining (see   \cite[Eq. (13.3.15), (13.3.4), (13.3.1)]{DLMF})
\[\begin{gathered}
    \frac{d}{dz}M\left(a,b,z\right)=\frac{a}{b}M\left(a+1,b+1,z
\right)\,,\\
zM\left(a+1,b+1,z\right)= bM\left(a+1,b,z\right)-bM\left(a,b,z\right)\,,\\
aM\left(a+1,b,z\right)=(b-a)M\left(a-1,b,z\right)+(2a-b+z)M\left(a,b,z\right)\,.
\end{gathered}
\]
It is worth noting that in light of the  identity
\[(T_h-\lambda_k-4h)A=(A+1)(T_h-\lambda_k),\]
the operator $A$ is akin to a creation operator for $T_h\,$.

Note that $u_\mu$ does not satisfy the Dirichlet boundary condition at $x=1$ since by \eqref{eq:identity},
\[u_\mu(1)=\frac{1}{2c}f'_k(1)\not=0\,.\]
Actually, since $f_k$ is positive near $x=0$ and it has $k-1$ simple zeros in $(0,1)$, the sign of $u_\mu(1)$ is the same as $(-1)^{k}$.
Thus, $\lambda_{k+1}\not=\mu\,.$ We will prove that $\lambda_{k+1}>\mu\,$. We argue by contradiction and suppose that $\lambda_{k+1}<\mu$\,.
We claim that
\begin{equation}\label{eq:claim}
    \text{If $f_k$ does not vanish on an open interval $I\subset(0,1),$ then $u_\mu$ has at most one zero in $I$\,.}
\end{equation}
We postpone the proof of \eqref{eq:claim} and proceed with the derivation of the sought contradiction. By 
the Sturm-Liouville theory, we know that $f_{k}$ has  $k-1$ distinct roots in $(0,1)$. A root $z$ of $f_{k}$ is not a root of $u_\mu$ since
by \eqref{eq:identity}, $u_\mu(z)=\frac{z}{2c}f_k'(z)\not=0\,$. As a consequence of \eqref{eq:claim}, 
$u_\mu$ has at most $k$ roots in $(0,1)$\,.

At the same time, $f_{k+1}$ has $k$ distinct roots
\[0<y_1<\ldots<y_{k}<1\,.\]
Put $y_{k+1}=1\,$. Thanks to the assumption $\lambda_{k+1}<\mu\,$, we get by the Sturm–Picone comparison theorem \cite[Theorem~B$^*$]{H}  that $u_\mu$ has a root in every $(y_i,y_{i+1})\,,$ for $i=1,\ldots,k\,$, thereby obtaining $k$ roots of $u_\mu$ in $(y_1,1)\,$. Moreover, $u_\mu$ has 
a root in $(0,y_1)\,$. Indeed, suppose not. Since $M(a,b,0)=1$,
both $u_\mu$ and $f_{k+1}$ are positive near $0\,$, hence both $u_\mu$ and $f_{k+1}$ are positive 
on $(0,y_1)$, so that $f_{k+1}'(y_1)<0\,$. Consider the Wronskian
\[
W:=u_\mu f_{k+1}'-u_\mu' f_{k+1}\,,\qquad
W'=\frac{\mu-\lambda_{k+1}}{h^2}\,u_\mu f_{k+1}>0\ \text{ on }(0,y_1)\,.
\]
In particular, $W$ is increasing on $(0,y_1)$. Since $\lim_{x\to0}W(x)=0$, we get that $W(x)\geq 0$ on $(0,y_1]\,$. With $0<\alpha<y_1\,$, it follows by integration on $(\alpha,y_1)\,$ that
\[W(y_1)=W(\alpha)+\int_\alpha^{y_1}W'(x)dx>0\,.\]
However, since $y_1$ is a zero of $f_{k+1}\,$, we get directly that $W(y_1)=u_\mu(y_1)f_{k+1}'(y_1)\leq 0\,$, a contradiction. 

Now, under the assumption $\lambda_{k+1}<\mu\,,$  we have proved that $u_\mu$ has $k+1$ roots. This contradicts the previous conclusion that $u_\mu$ has at most $k$ roots in $(0,1)\,$.\mn

It remains to prove the claim in \eqref{eq:claim}.

Note that, on an interval $I$ where $f_k$ does not vanish, we have 
\[Af_k=\frac{x}{2}(\phi+q)f_k\,,\quad \phi=\frac{f_k'}{f_k}\,,\quad q=\frac{\lambda_k+h}{2hx}-\frac{x}{h}\,.\]
In light of \eqref{eq:identity}, $u_\mu$ and $g=\phi+q$ have the same zero set in $I\,$. We will prove that $g'<0$ at a zero of $g$\,, which means that $g$ has at most one zero in $I$\,.

Combining the identity
\[\phi'=\frac{x^2}{h^2}+\frac{\nu^2-1/4}{x^2}-\frac{\lambda_k}{h^2}-\phi^2\,,\]
with $q=-\phi$ at a zero of $g\,,$ we obtain 
\[g'=\frac{x^2}{h^2}+\frac{\nu^2-1/4}{x^2}-\frac{\lambda_k}{h^2}-q^2+q' \text{ on }\{g=0\}\,.\]
By a straightforward calculation
\[\frac{x^2}{h^2}+\frac{\nu^2-1/4}{x^2}-\frac{\lambda_k}{h^2}-q^2+q'=\frac{1}{x^2}\Bigl(\nu^2-\bigl(\tfrac{\lambda_k}{2h}+1\bigr)^2\Bigr)\,,\]
with
\[\lambda_k\geq  \lambda_1\geq e_1(h)=(2\nu+2)h\,.\]
This proves that $g'(x)<0$ if $g(x)=0$\,.
\end{proof}

\subsection*{Acknowledgement} B.H. was
supported by the Agence Nationale de Recherche grant ANR-25-CE40-7296
(\href{https://anr.fr/Projet-ANR-25-CE40-7296}{La Gabare}). A.K. is partially supported by AUB URB grant (award no. 104752, project 28565), F.N. thanks the French GDR Dynqua for its support. The authors would like to thank R. Vanlaere for valuable comments on the first version of the manuscript.

\paragraph{\bf Use of artificial intelligence.}
The authors used AI tools for assistance with language editing,
graphical preparation, verification of some intermediate calculations, and a suggestion for a proof of Theorem~\ref{thm:gap}.
All mathematical arguments and conclusions were  independently checked
and written by the authors. 

\appendix 

\section{On the Bohr--Sommerfeld--Langer interpretation of the
Gabutti--Gatteschi equation}\label{AppA}

\mn
As recalled in the introduction, the spectral condition can be
reformulated in terms of the Whittaker function
$M_{\kappa,\mu}\,$, with
\begin{equation}\label{Whittaker-parameters-App}
\kappa=\frac{\lambda}{4h}\,,
\qquad
\mu=\frac{\nu}{2}\,.
\end{equation}
The spectral condition then takes the form
\begin{equation}\label{Whittaker-spectral-condition-App}
M_{\kappa,\mu}(1/h)=0\,.
\end{equation}
The Whittaker function $M_{\kappa,\mu}$ satisfies
\begin{equation}\label{Whittaker-equation}
w''(z)+
\left(
-\frac14+\frac{\kappa}{z}
+\frac{\frac14-\mu^2}{z^2}
\right)w(z)=0\,.
\end{equation}
In the supercritical regime, where the energy stays in a compact
subset of $(1,+\infty)\,$, we have
\begin{equation}\label{kappa-semiclassical-App}
\kappa=\frac{\lambda}{4h}\longrightarrow+\infty
\qquad\text{as }h\to0^+\,.
\end{equation}

\mn
The purpose of this appendix is to explain the origin of the implicit
relation \eqref{GG-zeta-equation} used in the Gabutti--Gatteschi
asymptotics. In the semiclassical limit
$\kappa\to+\infty$, the Whittaker equation can be reduced to a Bessel
model by a suitable change of variable. This change of variable is
determined by requiring that the classical actions of the two
equations coincide. This is precisely the content of
\eqref{GG-zeta-equation}. This equality of actions will then allow us
to recover the Bohr--Sommerfeld--Langer quantization rule of
Theorem~\ref{thm:supercritical-intro}.

\mn
We now rewrite \eqref{Whittaker-equation} in semiclassical form in
order to identify its principal part as $\kappa\to+\infty\,$.
Introduce
\begin{equation}\label{Whittaker-scaling}
z=\kappa X,
\qquad
\gamma=\frac{2\mu}{\kappa}\,.
\end{equation}
Then \eqref{Whittaker-equation} becomes
\begin{equation}\label{scaled-Whittaker-equation}
\kappa^{-2}w_{XX}
+
\frac{1}{4X^2}
\left(
-X^2+4X-\gamma^2+\kappa^{-2}
\right)w
=
0\,.
\end{equation}
Thus, at leading semiclassical order, the equation reduces to
\begin{equation}\label{principal-Whittaker-equation}
\kappa^{-2}w_{XX}
+
q_W(X;\gamma)w
=
0\,,
\qquad
q_W(X;\gamma)
=
\frac{-X^2+4X-\gamma^2}{4X^2}\,.
\end{equation}
The turning points of the principal equation are the roots of
\[
-X^2+4X-\gamma^2=0\,,
\]
namely
\[
\xi_1=2-\sqrt{4-\gamma^2}\,,
\qquad
\xi_2=2+\sqrt{4-\gamma^2}\,.
\]
The oscillatory region is
\[
\xi_1<X<\xi_2\,.
\]
In this region, the classical momentum associated with the principal
equation is
\begin{equation}\label{Whittaker-momentum}
p_W(X;\gamma)
=
\sqrt{q_W(x, \gamma)}
=
\frac{1}{2X}
\sqrt{(X-\xi_1)(\xi_2-X)}\,.
\end{equation}
The corresponding Whittaker action, measured from the left turning
point $\xi_1$, is therefore
\begin{equation}\label{Whittaker-action-def}
S_W(X;\gamma)
=
\int_{\xi_1}^{X}p_W(t;\gamma)\,dt
=
\int_{\xi_1}^{X}
\frac{1}{2t}
\sqrt{(t-\xi_1)(\xi_2-t)}\,dt\,.
\end{equation}
Writing
\[
R(X)=\sqrt{(X-\xi_1)(\xi_2-X)}\,,
\]
a direct computation gives
\begin{align}
S_W(X;\gamma)
&=
\frac12R(X)
-\frac{\gamma}{2}
\arctan\left(
\frac{2X-\gamma^2}{\gamma R(X)}
\right)
\nonumber\\
&\quad
-\arctan\left(
\frac{2-X}{R(X)}
\right)
+
\frac{\pi}{2}
\left(1-\frac{\gamma}{2}\right)\,.
\label{Whittaker-action-explicit}
\end{align}
Thus, the right-hand side of \eqref{GG-zeta-equation} is precisely
the Whittaker action.

\vspace{0.5cm}\noindent
We now compare the principal Whittaker equation
\eqref{principal-Whittaker-equation} with a Bessel model.
Consider
\begin{equation}\label{Bessel-model-equation}
u''(r)
+
\left(
1+\frac{\frac14-4\mu^2}{r^2}
\right)u(r)
=
0\,.
\end{equation}
Introducing
\begin{equation}\label{Bessel-scaling}
r=\kappa Y\,,
\qquad
\gamma=\frac{2\mu}{\kappa}\,,
\end{equation}
we obtain
\begin{equation}\label{scaled-Bessel-equation}
\kappa^{-2}u_{YY}
+
\left(
1-\frac{\gamma^2}{Y^2}
+\frac{1}{4\kappa^2Y^2}
\right)u
=
0\,.
\end{equation}
Thus, at leading semiclassical order, the Bessel model reduces to
\begin{equation}\label{principal-Bessel-equation}
\kappa^{-2}u_{YY}
+
q_B(Y;\gamma)u
=
0\,,
\qquad
q_B(Y;\gamma)
=
1-\frac{\gamma^2}{Y^2}\,.
\end{equation}

\mn
Thus, both principal equations have the same semiclassical parameter
$\kappa^{-1}\,$. They depend on the parameter
$
\gamma=\frac{2\mu}{\kappa}\,,
$
although it enters the two principal potentials differently.

\mn 
We now relate these equations by a Liouville change of variable, which is
determined by requiring that the corresponding classical actions
coincide.

\mn
The Bessel potential $q_B(Y;\gamma)$ has a positive turning point at
$Y=\gamma$, and the oscillatory region is $Y>\gamma$. In this region,
the corresponding classical momentum is
\begin{equation}\label{Bessel-momentum}
p_B(Y;\gamma)
=
\sqrt{q_B(Y;\gamma)}
=
\frac{\sqrt{Y^2-\gamma^2}}{Y}\,.
\end{equation}
The Bessel action, measured from the turning point $Y=\gamma$, is
therefore
\begin{equation}\label{Bessel-action-Y}
S_B(Y;\gamma)
=
\int_{\gamma}^{Y}p_B(s;\gamma)\,ds
=
\int_{\gamma}^{Y}
\frac{\sqrt{s^2-\gamma^2}}{s}\,ds\,.
\end{equation}

\mn
To recover the variable used in the Gabutti--Gatteschi equation, we
now set
\begin{equation}\label{Bessel-zeta-variable}
\zeta=Y^2\,.
\end{equation}
With the change of variable $t=s^2$, equation
\eqref{Bessel-action-Y} becomes
\begin{equation}\label{Bessel-action-def}
S_B(\zeta;\gamma)
=
\int_{\gamma^2}^{\zeta}
\frac{1}{2t}\sqrt{t-\gamma^2}\,dt\,.
\end{equation}
A direct computation gives
\begin{equation}\label{Bessel-action-explicit}
S_B(\zeta;\gamma)
=
\sqrt{\zeta-\gamma^2}
-
\gamma
\arctan\left(
\frac{\sqrt{\zeta-\gamma^2}}{\gamma}
\right)\,.
\end{equation}
\mn
The new variable $\zeta$ in Dunster's Liouville transformation is
defined by requiring equality of the two actions:
\begin{equation}\label{Liouville-action-identity}
	S_W(X;\gamma)=S_B(\zeta;\gamma)\,.
\end{equation}
This equality determines the change of variable $X\mapsto\zeta\,$, 
which transforms the principal Whittaker equation into the principal
Bessel equation. It is exactly the implicit relation
\eqref{GG-zeta-equation}. Thus, the somewhat complicated equation
used by Gabutti--Gatteschi simply expresses, in explicit form, the
equality of the two classical actions.

\mn
\mn
We finally explain how the Bohr--Sommerfeld--Langer quantization
appearing in Theorem~\ref{thm:supercritical-intro} is encoded in this
construction. For the $k$-th positive Whittaker zero, the Bessel
variable is chosen as
\begin{equation}\label{Bessel-zero-zeta-App}
	\zeta_k
	=
	\left(
	\frac{j_{2\mu,k}}{\kappa}
	\right)^2\,,
\end{equation}
where $j_{2\mu,k}$ denotes the $k$-th positive zero of the Bessel
function $J_{2\mu}$\,. The corresponding Whittaker variable $X_k$ is then determined by the
action identity
\begin{equation}\label{action-identity-k-App}
	S_W(X_k;\gamma)
	=
	S_B(\zeta_k;\gamma)\,.
\end{equation}
For large $k$, the classical McMahon expansion gives
\begin{equation}\label{McMahon-Langer-App}
	j_{2\mu,k}
	=
	\pi\left(k+\mu-\frac14\right)
	+
	\mathcal O(k^{-1})\,.
\end{equation}
Therefore, in the semiclassical regime $h\to0^+$ with
$\rho_0\leq 4kh\leq\rho_1\,$, the action identity
\eqref{action-identity-k-App}, together with the McMahon expansion,
naturally leads to a Bohr--Sommerfeld quantization condition. The rigorous asymptotic analysis carried out in
Paragraph~\ref{sss3.1.4}  shows that this procedure yields
\begin{equation}\label{BS-Langer-final-App}
	I(E)
	=
	\pi h
	\left(
	k+\frac{\nu}{2}-\frac14
	\right)
	+
	\mathcal O(h^2)\,,
\end{equation}
which is precisely the Bohr--Sommerfeld--Langer quantization rule
appearing in Theorem~\ref{thm:supercritical-intro}.

\section{Zeros of the Kummer function with respect to the first parameter}\label{AppB}

\subsection{Introduction}

In this appendix, we reformulate the spectral asymptotics obtained in the main text as asymptotic formulas for the zeros of the Kummer confluent hypergeometric function $M(a,b,z)$ with respect to its first parameter. Our purpose is to collect these formulas in a form that can be read independently of the spectral problem. To this end, we recall the standard definitions and relations that have already been introduced at various places in the paper. 

\mn
We first recall the definition of the Kummer confluent hypergeometric
function. For $a,z\in\mathbb{C}$ and
$b\in\mathbb{C}\setminus\{0,-1,-2,\ldots\}\,$, it is defined by
\begin{equation}\label{eq:Kummer-definition}
	M(a,b,z)={}_1F_1(a;b;z)
	=
	\sum_{n=0}^{\infty}
	\frac{(a)_n}{(b)_n}\frac{z^n}{n!}\,,
\end{equation}
where $(a)_n$ denotes the Pochhammer symbol,
\[
(a)_0=1\,,
\qquad
(a)_n=a(a+1)\cdots(a+n-1)\,,
\qquad n\geq1\,.
\]
The series in \eqref{eq:Kummer-definition}
converges for every $z\in\mathbb{C}$ and defines an entire function of
$z$. It also defines an entire function of the first parameter $a$,
for fixed admissible $b$ and $z\,$. We refer to
\cite[\S 13.2, Eq.~(13.2.2)]{DLMF} for these standard definitions and
properties.

\mn
We also recall the classical relation between the Kummer and Whittaker
functions. The Whittaker function $M_{\kappa,\mu}$ is related to the
Kummer function by
\begin{equation}
M_{\kappa,\mu}(z)
=
e^{-z/2}z^{\mu+1/2}
M\left(\mu-\kappa+\frac12,2\mu+1,z\right)\,.
\end{equation}
Equivalently, setting
\[
\mu=\frac{b-1}{2}\,,
\qquad
\kappa=\frac{b}{2}-a\,,
\]
one obtains
\[
M(a,b,z)
=
e^{z/2}z^{-b/2}
M_{\frac{b}{2}-a,\frac{b-1}{2}}(z)\,.
\]
We refer to \cite[\S 13.14, Eqs.~(13.14.2) and (13.14.4)]{DLMF}
for these identities and for the standard definitions and properties
of the Whittaker functions. In particular, for fixed $b$ and $z$, the zeros of
$a\mapsto M(a,b,z)$ are in one-to-one correspondence with the zeros of
\[
\kappa\mapsto M_{\kappa,(b-1)/2}(z)\,,
\qquad
\kappa=\frac{b}{2}-a\,.
\]
Thus, the asymptotic formulas obtained below for the zeros of the
Kummer function immediately yield corresponding formulas for the
zeros of the Whittaker function with respect to its first parameter.
For simplicity, we shall state our results only in terms of the Kummer
function.

\mn 
The zeros of the Kummer function with respect to its argument $z$
have been extensively studied. We refer to
\cite[\S 13.9]{DLMF} and the references therein for results concerning
their number and location, as well as asymptotic formulas for large
zeros. In particular, for fixed $a$ and $b$, the asymptotic
distribution of the large $z$-zeros is described in
\cite[\S 13.9, Eq.~(13.9.9)]{DLMF}.

\mn 
Much less seems to be known when the zeros are considered with respect
to the first parameter $a$. For fixed $b$ and $z$, the classical
large-index asymptotics of the $a$-zeros is given in
\cite[\S 13.9, Eq.~(13.9.10)]{DLMF}. In our indexing convention (see below), the large-index expansion
in \cite[Eq.~(13.9.10)]{DLMF} reads
\begin{equation}\label{eq:Kummer-a-large-k}
	a_k(z)
	=
	-\frac{\pi^2}{4z}
	\left(
	k^2+\left(b-\frac32\right)k
	\right)
	+\mathcal O(1)\,,
	\qquad k\to+\infty\,,
\end{equation}
for fixed $b$ and $z\,$.  This asymptotic regime, however, does not
describe the behavior of the low-lying $a$-zeros when $z$ becomes
large.

\mn
For fixed $k\geq 0$, a much more precise description in the limit
$z\to+\infty$ follows from the asymptotics of Gannot
\cite{Gannot}, recalled in \eqref{eq:Gannot*}. Translated into the present notation, his result gives,
for fixed $b$ and $k\geq0$\,,
\begin{equation}
a_k(z)
=
-k-
\frac{z^{2k+b}}{k!\,\Gamma(k+b)}
e^{-z}
\left(1+\mathcal O(z^{-1})\right)\,,
\qquad z\to+\infty\,.
\end{equation}
Thus the $k$-th $a$-zero is exponentially close to the negative
integer $-k$\,.

\mn
The purpose of the following subsections is to extend this picture
beyond the fixed-index regime and to describe the $a$-zeros uniformly
when $k$ is allowed to grow with $z\,$, both in the subcritical and
supercritical regimes. Before stating our results, let us clarify the
indexing convention used throughout this appendix. The eigenvalues studied in this paper are
indexed by $k\geq1$, whereas the negative zeros of the Kummer function
are conventionally indexed starting from  $k=0$; their negativity will be
justified below. Consequently, there is a shift of one between the two
indexings: the $k$-th Kummer zero $a_k(z)$ corresponds to the $(k+1)$-st
eigenvalue.

\mn
\mn
Throughout the remainder of this appendix, we assume that $b\geq1$ is
fixed, in accordance with the range of parameters considered in our
spectral analysis. For $z>0$, all the zeros of
$a\mapsto M(a,b,z)$ are real and simple. Indeed, by
\eqref{spectral-equation}, they are in one-to-one correspondence
with the simple eigenvalues of the self-adjoint Sturm--Liouville problem
considered above. More precisely, setting $h=z^{-1}$ and $b=1+\nu$,
with $\nu\geq0\,$, this correspondence is given by
\begin{equation}\label{eq:kummer-dictionary}
	a_k(z)=\frac b2-\frac z4\lambda_{k+1}(z^{-1})\,.
\end{equation}
Since $\lambda_{k+1}(h)\geq e_{k+1}(h)$ and
\begin{equation}\label{eq:kummer-ek}
	e_{k+1}(z^{-1})=\frac{4k+2b}{z}\,,
\end{equation}
we obtain
\begin{equation}\label{eq:kummer-upper-bound}
	a_k(z)\leq-k\,.
\end{equation}
Moreover, $a_0(z)\neq0$, since $M(0,b,z)=1$. Hence
\eqref{eq:kummer-upper-bound} implies that all the zeros are strictly
negative (see also \cite[Theorem~1.6]{Vanlaere}). With the indexing convention introduced above, we therefore
write
\begin{equation}\label{eq:kummer-zero-ordering}
	\cdots<a_2(z)<a_1(z)<a_0(z)<0\,.
\end{equation}
Furthermore, Theorem~\ref{thm:gap} yields a non-asymptotic separation bound for consecutive $a$-zeros: for every fixed $b\geq1$ and every $z>0\,$,
\[
a_k(z)-a_{k+1}(z)>1,\qquad k\geq0\,.
\]
Indeed, the correspondence \eqref{eq:kummer-dictionary} gives
\[
a_k(z)-a_{k+1}(z)
=\frac z4\bigl(\lambda_{k+2}(z^{-1})-\lambda_{k+1}(z^{-1})\bigr)
>1\,.
\]
\subsection{The subcritical regime}
\subsubsection{The low-energy regime}

As a consequence of Theorem~\ref{thm:ext-Gannot} and the correspondence
\eqref{eq:kummer-dictionary}, we obtain the following asymptotics for the
Kummer zeros. As $z\to+\infty\,$,
\begin{equation}\label{eq:kummer-low-energy}
	a_k(z)
	=
	-k-
	\frac{z^{2k+b}}{k!\,\Gamma(k+b)}
	e^{-z}
	\bigl(1+o(1)\bigr)\,,
\end{equation}
uniformly with respect to the integers $k\geq0$ satisfying
$k+1\leq z^{\gamma}\,$, $0<\gamma<1/2\,$. In particular, \eqref{eq:kummer-low-energy} applies
to every fixed $k\geq0$ as $z\to+\infty\,$.
\subsubsection{The higher-energy subcritical regime}

We now turn to the regime described by Theorem~\ref{thm:main}.
Let $0<\rho^*<1$\,, and let $\ell:(0,1]\to\mathbb R_+$ satisfy
\begin{equation}
\ell(h)\longrightarrow0\,,
\qquad
h^{-1}\ell(h)\longrightarrow+\infty
\qquad
(h\to0^+)\,.
\end{equation}
Then, as $z\to+\infty\,$,
\begin{equation}\label{eq:kummer-subcritical}
	a_k(z)
	=
	-k-\frac{1}{2\pi}
	\exp\bigl(-2z S_{k,z}\bigr)
	\bigl(1+o(1)\bigr)\,,
\end{equation}
uniformly with respect to the integers $k\geq0$ satisfying
\begin{equation}\label{eq:kummer-subcritical-range}
	\ell(z^{-1})
	\leq
	\frac{4k+2b}{z}
	\leq
	\rho^*\,,
\end{equation}
where
\begin{equation}\label{eq:Skz}
	S_{k,z}
	=
	\int_{\sqrt{(4k+2b)/z}}^1
	\sqrt{s^2-\frac{4k+2b}{z}}\,ds\,.
\end{equation}

\begin{remark}
	The low-energy asymptotic formula \eqref{eq:kummer-low-energy} matches
	the subcritical asymptotic formula \eqref{eq:kummer-subcritical} in
	their common range of validity. More precisely, in the regime
	\[
	k\to+\infty\,,
	\qquad
	\frac{k^2}{z}\to0\,,
	\]
	the two asymptotic formulas for the exponentially small displacement
	of $a_k(z)$ from the negative integer $-k$ are equivalent.
\end{remark}


\subsection{The transition regime}

We now consider the transition regime between the subcritical
and supercritical asymptotics. For $k\geq0$, set
\begin{equation}\label{eq:kummer-tau-transition}
	\tau_{k,z}
	:=
	\frac{z^{2/3}}{2^{2/3}}
	\left(
	1-\frac{4(k+1)}{z}
	\right)\,.
\end{equation}
Let $T>0$ and assume that
$
|\tau_{k,z}|\leq T.
$
Let $\theta:\mathbb R\to\mathbb R$ be the continuous phase defined by
\begin{equation}\label{eq:kummer-Airy-phase}
	\Ai(t)=A(t)\sin\theta(t),
	\qquad
	\Bi(t)=A(t)\cos\theta(t),
	\qquad
	A(t)=\bigl(\Ai(t)^2+\Bi(t)^2\bigr)^{1/2},
\end{equation}
with the determination chosen so that
$\theta(t)\to0$ as $t\to+\infty\,$.
Then, uniformly for $|\tau_{k,z}|\leq T\,$,
\begin{equation}\label{eq:kummer-transition}
	\begin{aligned}
		a_k(z)
		={}&
		-k-\frac{\theta(\tau_{k,z})}{\pi}
		\\
		&+
		\frac{z^{-1/3}}{\pi\,2^{2/3}}
		\left(
		2b-4+\frac{4}{\pi}\theta(\tau_{k,z})
		\right)
		\theta'(\tau_{k,z})
		+\mathcal O(z^{-2/3})\,,
		\mbox{ as }  z\to+\infty\,.
	\end{aligned}
\end{equation}

\subsection{The supercritical regime}

Let
$
1<\rho_0<\rho_1<+\infty.
$
In view of the shift in the indexing described above, for $k\geq0$ we set
\begin{equation}\label{eq:kummer-rho-super}
	\rho=\frac{4(k+1)}{z}\,,
\end{equation}
and assume that
$
\rho_0\leq\rho\leq\rho_1\,.
$
We recall the function $E_0$ appearing in the supercritical spectral
asymptotics. For $\rho>1$, $E_0(\rho)>1$ is uniquely determined by
\[
I(E_0(\rho))=\frac{\pi\rho}{4}\,,
\qquad
I(E)=\frac12\sqrt{E-1}
+\frac{E}{2}\arcsin(E^{-1/2})\,.
\]
Then, as $z\to+\infty\,$,
\begin{equation}\label{eq:kummer-supercritical-E0}
	a_k(z)
	=
	-\frac z4 E_0(\rho)
	+\frac b2
	-\frac{\pi(2b-3)}
	{8\,\arcsin\bigl(E_0(\rho)^{-1/2}\bigr)}
	+O(z^{-1})\,,
\end{equation}
uniformly with respect to $\rho\in[\rho_0,\rho_1]\,$. Equivalently, setting
\begin{equation}
\vartheta(\rho)
=
\arcsin\bigl(E_0(\rho)^{-1/2}\bigr)
\in\left(0,\frac{\pi}{2}\right)\,,
\end{equation}
we have
\begin{equation}
E_0(\rho)=\frac{1}{\sin^2\vartheta(\rho)}\,,
\end{equation}
where $\vartheta(\rho)$ is the unique solution of
\begin{equation}\label{eq:kummer-theta}
	\vartheta+\sin\vartheta\cos\vartheta
	=
	\frac{\pi\rho}{2}\sin^2\vartheta\,.
\end{equation}
Hence \eqref{eq:kummer-supercritical-E0} can also be written as
\begin{equation}\label{eq:kummer-supercritical}
	a_k(z)
	=
	-\frac{z}{4\sin^2\vartheta(\rho)}
	+\frac b2
	-\frac{\pi(2b-3)}{8\vartheta(\rho)}
	+\mathcal O(z^{-1})\,,
\end{equation}
uniformly with respect to $\rho\in[\rho_0,\rho_1]\,$.


\end{document}